\documentclass{amsart}
\usepackage{
amsmath,amssymb, % useful math stuff
graphicx, % gives additional placement options for \includegraphics
epstopdf, % convert eps Figure~s to pdf (I don't think I actually use this anywhere)
bm, % good for bolding Greek symbols
hyperref % allow for links between references
}

\usepackage{float}
\usepackage{amsthm}
\usepackage{mathtools}
\usepackage[utf8]{inputenc}
\usepackage{pinlabel}
\usepackage{color}
\usepackage{rotating,threeparttable,booktabs,dcolumn}
\usepackage{hyperref}
\usepackage[caption=false,justification=raggedright]{subfig}
\usepackage{geometry}
\usepackage{breakcites}
\newtheorem{theorem}{\bf Theorem}[section]
\newtheorem{conjecture}{\bf Conjecture}[section]
\newtheorem{lemma}{\bf Lemma}[section]

\newtheorem{definition}{\bf Definition}[section]

\newtheorem{remark}{\bf Remark}[section]
\title{A study of 2-periodic weft-knitted textiles using the theory of knots and links}
\author{Miriam Kuzbary}\email{kuzbary@gatech.edu}\address{Department of Mathematics, Amherst College, Amherst, Massachusetts, 01002, USA}
\author{Shashank G. Markande}\email{markandeshashank@gmail.com}\address{Atelio by FIS}
\author{Elisabetta A. Matsumoto}\email{sabetta@gatech.edu}\address{{School of Physics, Georgia Institute of Technology, Atlanta, Georgia 30332, USA}}\address{{International Institute for Sustainability with Knotted Chiral Meta Matter (WPI-SKCM), Hiroshima University, Higashi-Hiroshima, 739-8526, Japan}}
\author{Stanley Pritchard}\email{spritcha@eagles.nccu.edu }\address{Department of Mathematics and Physics, North Carolina Central University, Durham, North Carolina 27707, USA}

\begin{document}

\maketitle

\begin{abstract}
In this study, we use a correspondence between two-periodic weft-knitted textiles and links in the thickened torus to study the former using link invariants. We establish a criterion to identify the set of links whose elements are realized through techniques of weft-knitting leading to new, unconventional types of weft-knitting stitch patterns. A crucial topological underpinning of these links is shown to be their correspondence with ribbon knots and links in Euclidean three-space and equivalently in the three-sphere. Using the mechanics of weft-knitting, we propose a protocol for constructing and enumerating links in the thickened torus that can be knitted as a motif of a weft-knitted textile, and we call such links \emph{swatches}. Based on our analysis of link invariants of swatches, we propose conjectures on hyperbolic structure of the link complements of swatches and their multivariable Alexander polynomials. 
\end{abstract}

\section{Introduction}

\begin{figure*}
\centering
\subfloat[This stitch starts with a bight on the left needle $n_l$ (blue). The right needle $n_r$ (yellow) is inserted through the bight, below $n_l$ (top). This bight now has two both needles through it (bottom). The working end of the yarn is the horizontal line segment here.]
 {\includegraphics[width=60mm]{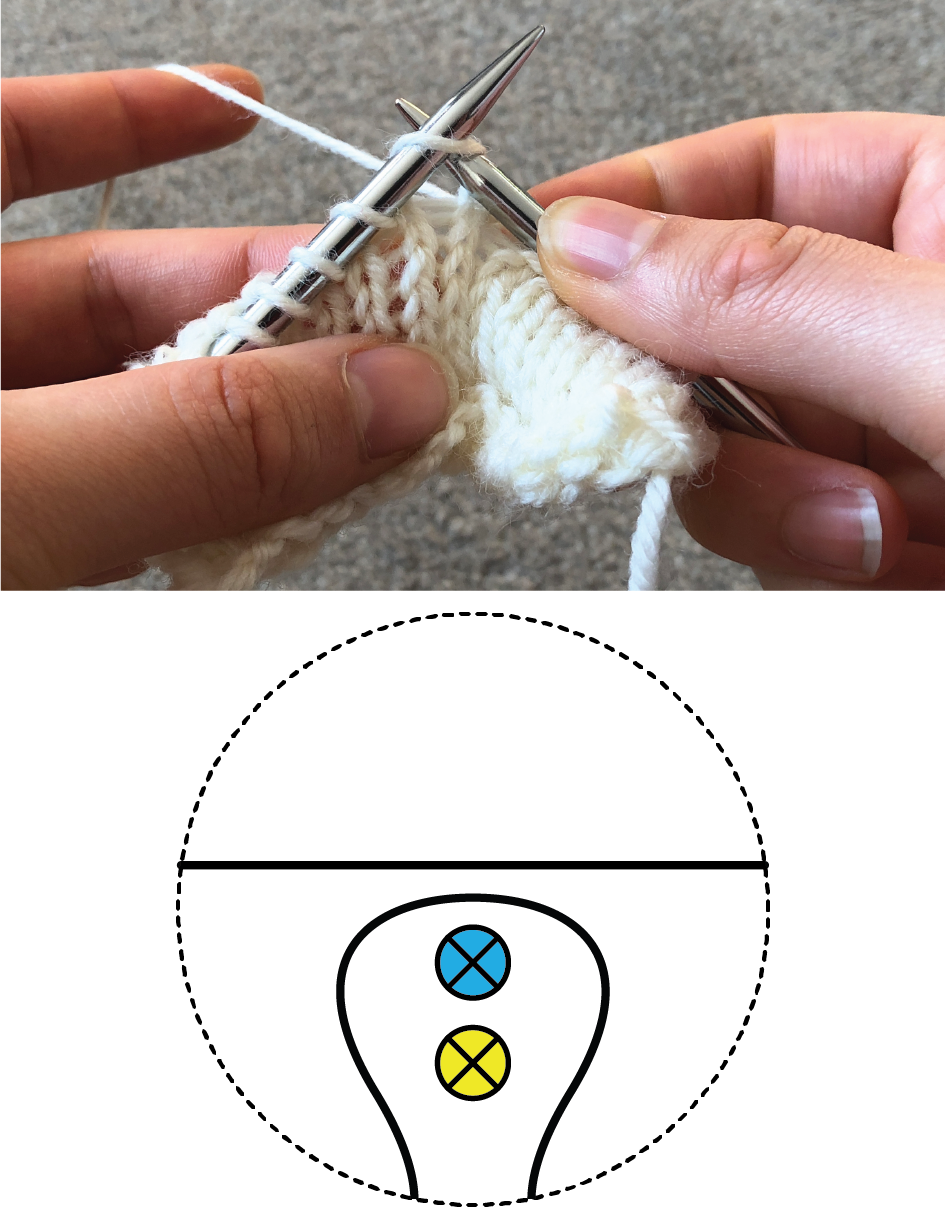} }
 \hspace{1.5mm}
\subfloat[The working end of the yarn (here, the left end) is wrapped around $n_r$ counterclockwise with the right needle facing away from us(top).] 
 {\includegraphics[width=60mm]{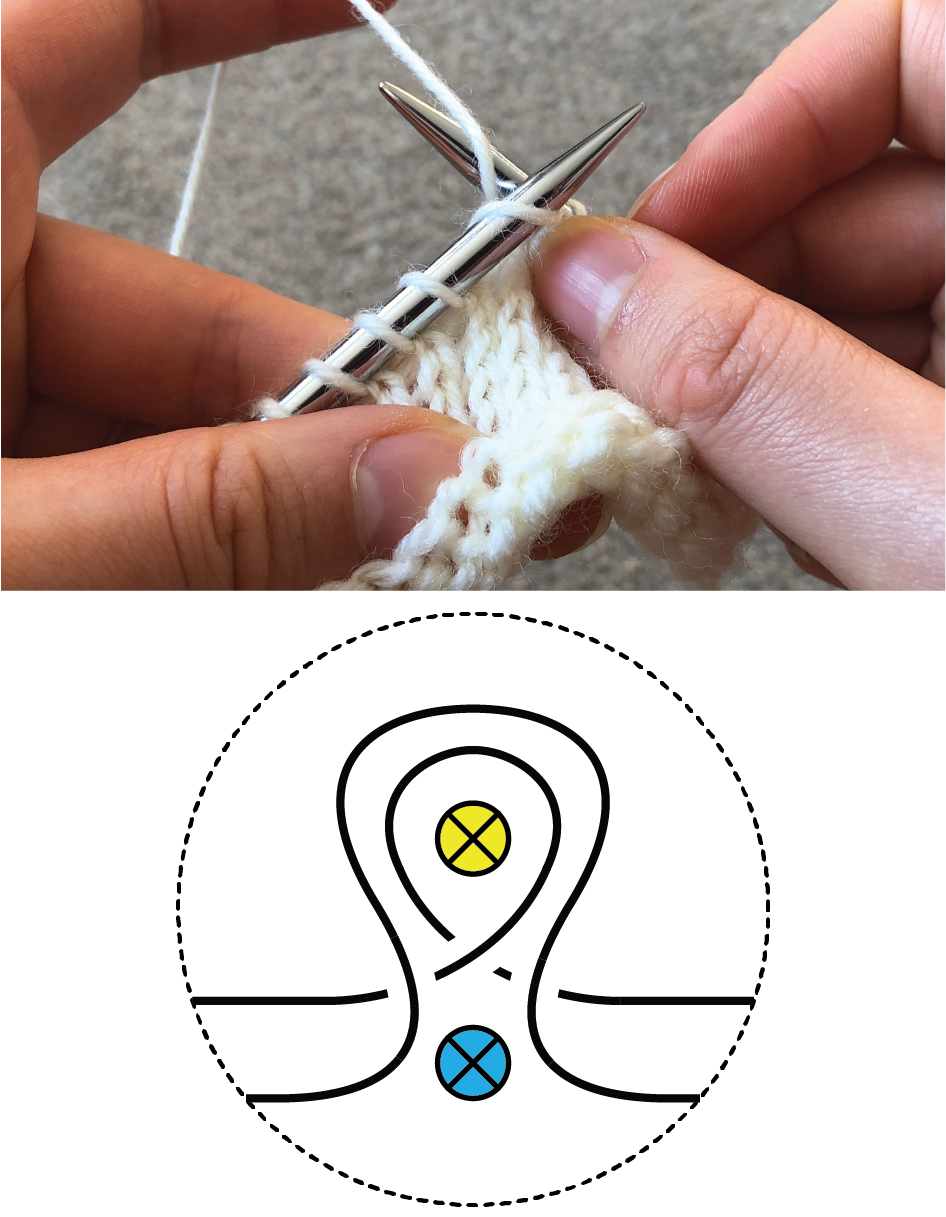}}\\
  \hspace{1.5mm}
\subfloat[The right needle $n_r$ along with the loop of working yarn is pulled through the bight (top). This creates a new bight caught around the right needle, and can be seen as a new stitch mounted on the right needle (bottom).] 
 {\includegraphics[width=60mm]{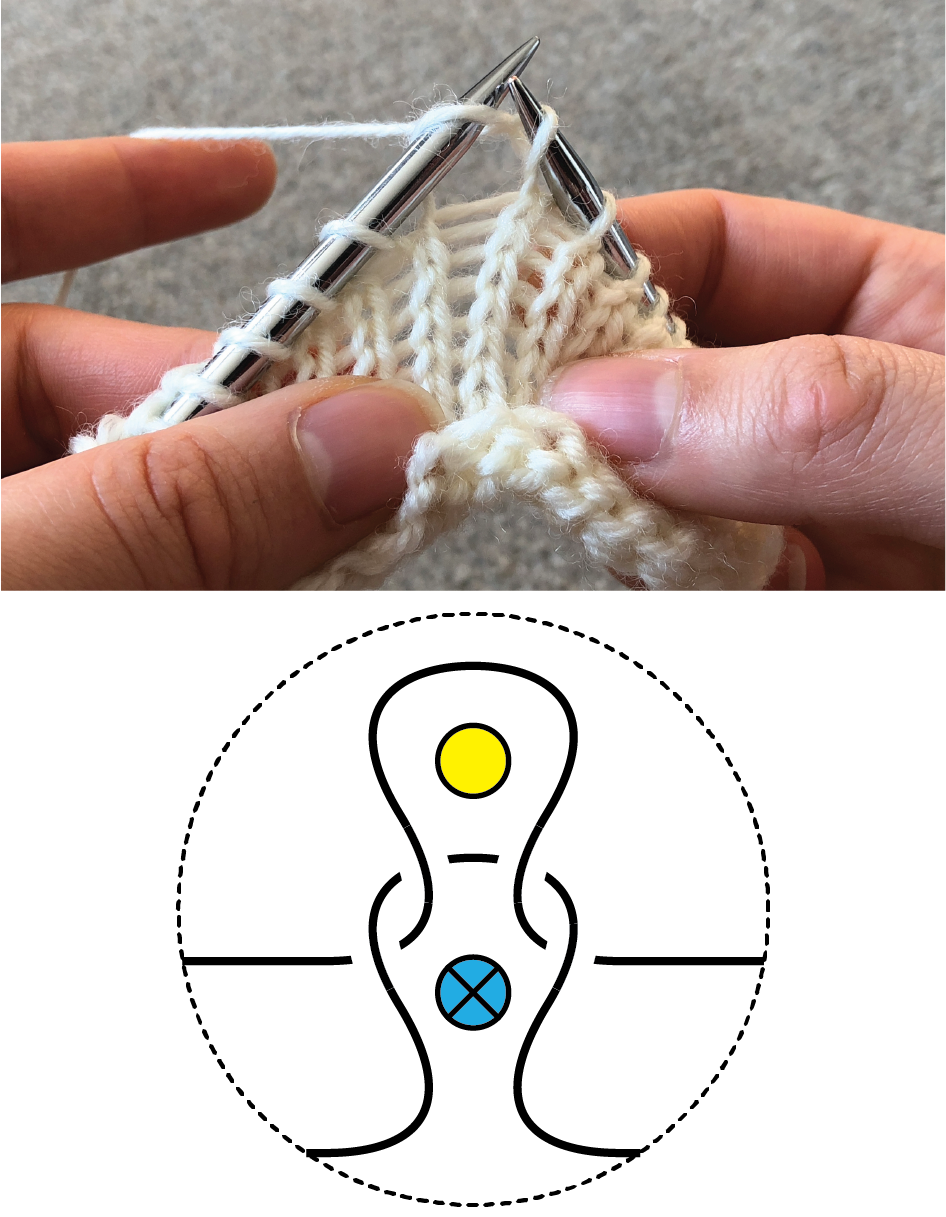} }
      \hspace{1.5mm}
    \subfloat[The initial bight through which $n_r$ was pushed through is now slipped off of $n_l$ (top). The blue disk is out of the frame since the newly created bight sits on $n_r$ holding the bight that we began with (bottom).]
     {\includegraphics[width=60mm]{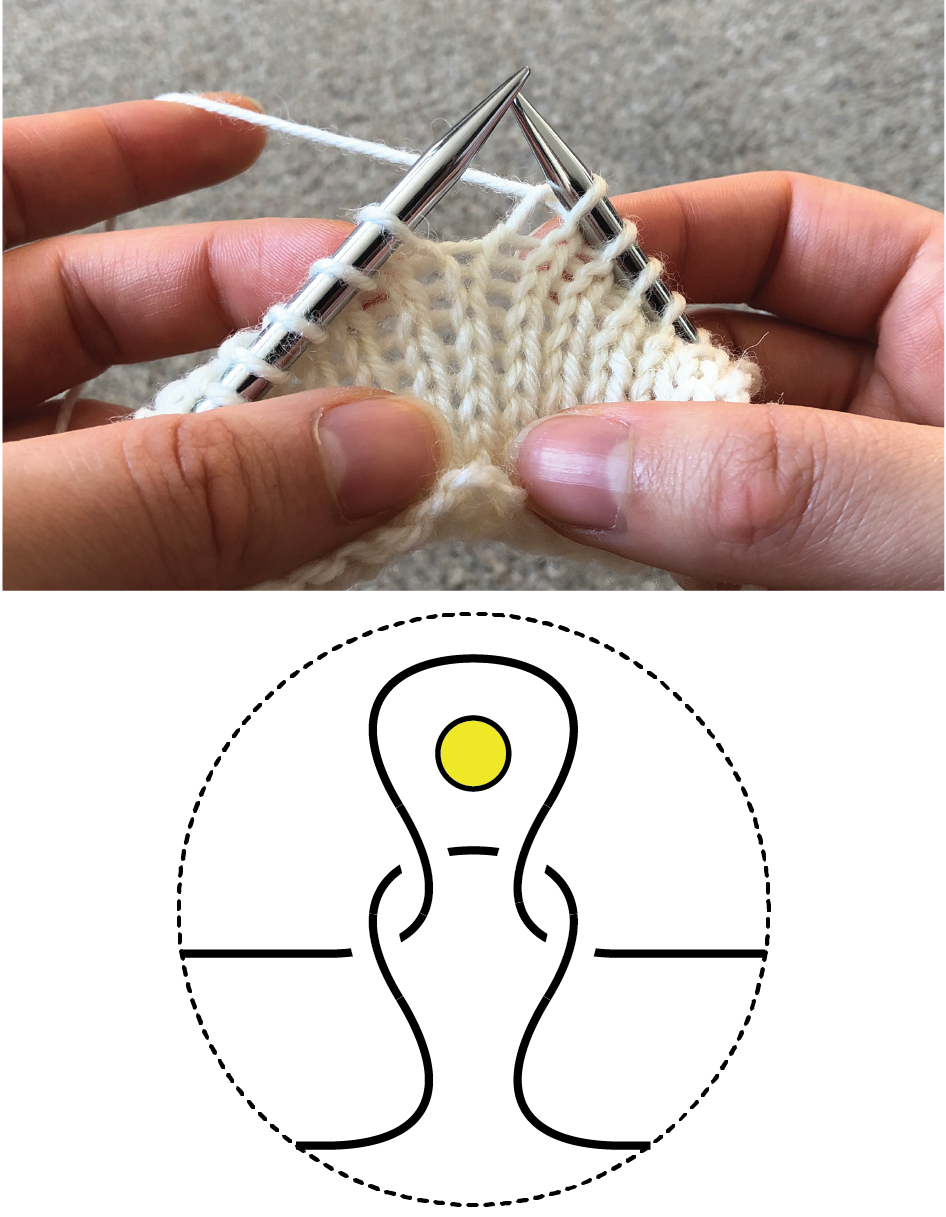} }\\

\caption{\label{fig:knitting}Construction of a basic knit stitch. The top images show 3D manipulations of the knitting needles and yarn. In planar knot diagrams (bottom images), the needles act like punctures in $\mathbb{R}^2$ that the yarn cannot pass through. When the pointed ends of the needles are facing into the page, they are marked with an (X) and when they are facing out of the page, they are marked with an (O).} 
\vspace{-5pt}
\end{figure*}

Knitting technology and knitted textiles have been a major part of the textile industry since the 16th century CE \cite{Albaron_1993}. Knitting technology plays a key role in the fashion \cite{Japel_2007}, arts $\&$ crafts \cite{Singh_2009} and knitwear industries \cite{Bernard_2008, Paden_2020}. However, it has also found a foothold in medicine \cite{Doser_2011, RAJENDRAN_2016}, electronics \cite{Bao_2012} \cite{Soin_2014}, and composite materials \cite{Tan_1997} \cite{Leong_2000}, \cite{ Weimer_2000} because of its high level of precision in designing materials with predetermined mechanical properties and 3D geometries -- even down to the composition of the fibers; unlike woven textiles, knits have a variety of physical attributes which can be systematically tuned and controlled. The emergent properties of a knitted textile are primarily determined by three things: the topology of the pattern with which the yarn is entangled to make the textile, the material properties of the filaments making up the yarn, and the geometry of how these filaments are assembled in each strand of yarn.

Broadly speaking, there are two kinds of knitted textiles: \emph{weft-knitted} textiles and \emph{warp-knitted} textiles. In a weft-knitted textile, the average path of the yarn traces a curve that goes horizontally along the width of the textile and then moves up by a row along the length of the textile. This is then repeated multiple times giving rise to a local planar sheet-like geometry. Therefore, a weft-knitted textile can be constructed by hand using yarn (a piece of long unknotted string) and a pair of needles. On the contrary, a warp-knitted textile is made of multiple disconnected pieces of yarn, each of which trace a two-column wide zigzag path along the entire length of the textile. Warp-knitted textiles cannot be made by hand.

The stockinette stitch pattern (also called jersey fabric) is widely used in making weft-knitted textiles. Its fundamental domain is a local motif called a \emph{knit stitch}, which is highlighted in Figure~\ref{fig:pattern_to_unitcell}(b)-(c). These motifs are the simplest possible slip loops or \emph{bights} \cite{AoB} -- \emph{a finite-length slack part of a piece of an unknotted string}. For example, the curved arc in Figure~\ref{fig:knitting}(a) bounding the yellow disk is a bight. As we show later, it turns out that weft-knitting entails different ways of tying knots into the bight and yields a lattice of slip knots. While making a piece of textile with stockinette stitch pattern, the knitter is working with two unknotted pieces of a single yarn, two needles (say, $n_r, n_l$), and a row of slip loops on $n_l$. The knitter then makes a new slip loop on the needle $n_r$ leading to an increment in the count of number of slip loops on $n_r$, and a decrement in the count of number of slip loops on $n_l$. The exact mechanical protocol involved in making a knit is illustrated in Figure~\ref{fig:knitting} starting from leftmost panel and ending at the rightmost panel (note that different knitting traditions have different orientation conventions). Below, we summarize this procedure in four steps. 

First, as shown in Figure~\ref{fig:knitting}(a), needle $\textrm{n}_r$ is pushed through the bight closest to an end of needle $\textrm{n}_l$. Second, as shown in Figure~\ref{fig:knitting}(b), the other piece of yarn is wrapped around needle $\textrm{n}_r$ in the clockwise direction. Next, as shown in Figure~\ref{fig:knitting}(c), needle $\textrm{n}_r$ is pulled out of the bight on needle $\textrm{n}_l$ along with the wrapped segment of yarn. Finally, the initial bight on needle $\textrm{n}_l$ is slipped off decreasing the count of total number of bights in the row by one. In the bottom panels of Figure~\ref{fig:knitting}, the planar diagram notation is used to describe the local configuration of yarn (in the upper panels) as a knit is constructed. 

In this paper, we begin to topologically classify a subset of weft-knitted textiles. To capture the topological properties (namely, the entanglement of strands of yarn and the knottiness of each strand) of a weft-knitted textile, we model the yarn as a smooth simple curve embedded in $\mathbb{R}^3$. This model, along with the lattice of translation symmetries of the textile's bulk structure, yields a correspondence between its two-periodic stitch patterns and the set of finite collections of entangled, knotted loops in the thickened torus $T^2\times[0,1]$. We call a knot or a link in $T^2\times[0,1]$ arising as a motif in the stitch pattern of a textile a \emph{textile link}. We study the topological properties of textiles by applying tools from the theory of knots and links to these textile links.  

Given an $n$-component link $L$ in $T^2\times[0,1]$, we obtain an $(n+2)$-component link $H\cup f(L)\subset S^3$ by embedding the thickened torus, $T^2\times[0,1]$ into the 3-sphere, $S^3$ as the exterior of a Hopf link denoted by $H$. This embedding is obtained by Dehn filling $T^2\times[0,1]$ along the marked curves (shown in Figure~\ref{fig:T2xI}(a)) on the boundary tori. Note that Dehn fillings are uniquely defined by specifying the filling curves as detailed in \cite{Rolfsen_1976}.

The bottom panels of Figure~\ref{fig:knitting}(a)-(d) use arcs and colored disks to depict the linking between two segments of a single strand of yarn resulting from pulling one loop through the other. The last panel indicates that if these arcs were parts of the boundary of a (possibly disconnected) smooth surface, then the surface would intersect itself in a slit lying in its interior. This observation, along with the correspondence between links in $T^2\times[0,1]$ and links in $S^3$, informs our hypothesis that textile links derived from two-periodic weft-knitted textiles give rise to ribbon links in $S^3$ via Dehn filling the two boundary components of $T^2 \times I$. Recall that a n-component link $L$ in $S^3$ is slice if its components bound n disjoint smooth, properly embedded disks in $B^4$, and is ribbon if there is a Morse function on $B^4$ such that the aforementioned disks only have minima with respect to this function. Equivalently, $L$ is ribbon if its components bound $n$ disjoint, smooth, immersed disks in $S^3$ where all intersections are ribbon intersections \cite{foxmilnor}.

We prove that all textile links give rise to ribbon links in $S^3$ in Section 2 and state the result in the following theorem. Therefore, we can obstruct a link from being a textile link using invariants of slice and ribbon links.

\vspace{.5cm}
\noindent \textbf{Theorem}~\ref{theorem:ribbon}: Let $S$ be the stitch pattern of a two-periodic weft-knitted textile, and let $L^{S}\subset T^2\times[0,1]$ be the corresponding $n$-component textile link. Then the link $f(L^{S})\subset S^3$ (where $f$ is the Dehn filling described above) is an $n$-component ribbon link.
\vspace{.5cm}

As we are using links to model the physical act of knitting, we will use different types of links in order to capture a variety of knitting data. The knitting motif itself is captured by a \emph{textile link} in $T^2 \times [0,1]$ as the motif is doubly periodic. However, to represent the process of knitting we must introduce a more subtle idea. This is because a textile link does not have a clear row being knit as described in the process in Figure \ref{fig:knitting}; therefore we have the notion of a \emph{swatch} (defined in Definition \ref{def:swatch} ) to better capture a motif with a row that is being actively worked. While a swatch is also a knot or link in $T^2\times[0,1]$, they must be obtained in a specific way that exactly represents how physical stitches are knitted. A swatch is obtained by performing ambient isotopies of \emph{unknits} (unlinks in the thickened torus whose components are either null-homotopic or essential loops that are homotopic to the longitude of the base torus)~\ref{def:unknit} followed by \emph{band surgeries}~\ref{def:bs}. As we detail in the next section, working with swatches instead of arbitrary textile links allows us to set an upper limit to how complex a weft-knitted textile motif can be. We also use this notion to prove \textbf{Theorem}~\ref{theorem:ribbon}.

\begin{figure*}
\centering
{\includegraphics[width=140mm]{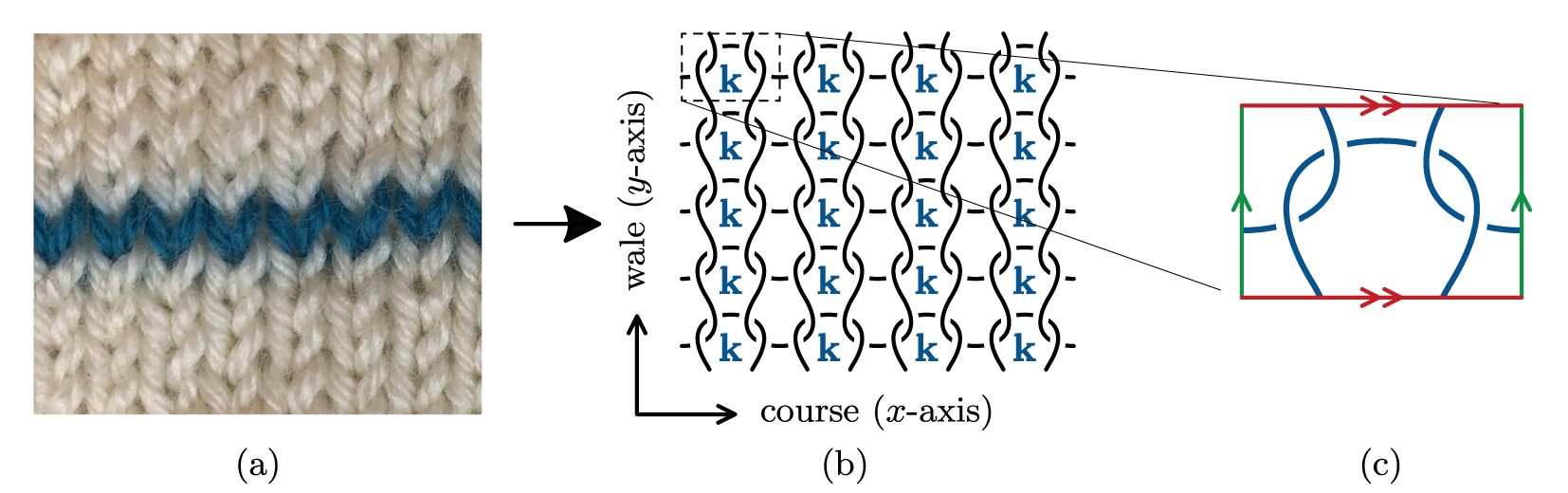}}
\caption{\label{fig:pattern_to_unitcell}(a) The bulk of stockinette fabric. (b) A space curve representation of a finite patch of the two-periodic knit projected orthogonally on to the $xy$-plane. The block in red is the fundamental translational unit and the symbol \emph{k} denotes the repeating motifs known as \emph{knits}. (c) The planar diagram of the corresponding textile knot in $T^2\times[0,1]$.} 
\end{figure*}

Throughout this paper, the discussion is driven by mainly by two questions. The first question is: \emph{how can we define weft-knitability?} That is, how can we distinguish weft-knitable two-periodic patterns of space curves from the rest? We realized upon careful examination that the notion of weft-knitability is not well-defined mathematically unless we consider only machine weft-knitting. However, even in case of machine weft-knitting, whether a pattern is knittable depends on the machine. This is due to differences in the set of mechanical operations each machine is capable of. Thus, the complexity of a stitch pattern for a weft-knitted textile is determined by the physical limitations of the person or machine knitting it. In our study, we ignore the dependency on the knitter and introduce a set of swatches that includes all textile links arising as motifs of two-periodic weft-knitted textiles. 
 
The second question: \emph{can we formulate a language of two-periodic weft-knitted fabrics with syntax and structure governed by the machine's (or knitter's) mechanical moves}?  McCann et al \cite{McCann_2016} developed a programming language to generate codes which are fed into a knitting machine. In this language, a symbol is assigned to every elementary operation performed by the needles in the needle bed. These symbols serve as the ``letters" of the alphabet. We take a different approach in specifying the alphabet and syntax, where the former is determined by when a 3-manifold is irreducible under an associative binary operation acting on swatches. This operation acts on 3-manifolds obtained by cutting the link complements of swatches along the meridional or the longitudinal cross-sections of the thickened torus followed by gluing the resulting annular boundaries. This operation is called an \emph{annulus sum}. A basic example of an annulus sum is given by cutting and gluing the link complements of a knit swatch (motif in stockinette fabric) and a purl swatch (motif in reverse stockinette fabric) to obtain the link complement of a $1\times1$ rib swatch. The syntax and grammar are determined by how we cut and glue link complements of swatches along their meridional and longitudinal cross-sections.   

While answering these questions of knitability and developing a weft-knitting grammar, we consider the problem of classification and irreducibility of textile links which arise as swatches. Irreducible swatches are those that cannot be simplified using annulus sums, which yield two ways of combining and reducing swatches. Thus, annulus sums define associative binary operations acting on the link complements of swatches, which in turn induce an algebra on link invariants of links associated with swatches. Previously, there have been studies along these lines by Grishanov \textit{et al} on the Kauffman polynomial \cite{GRISHANOV_2007}, Morton \textit{et al} on multivariable Alexander polynomials \cite{MORTON_2009}, and Champanerkar \textit{et al} on hyperbolic volume \cite{Champanerkar_19}.

\begin{figure*}
\centering
\subfloat[{reverse stockinette}.]
 {\includegraphics[width=70mm]{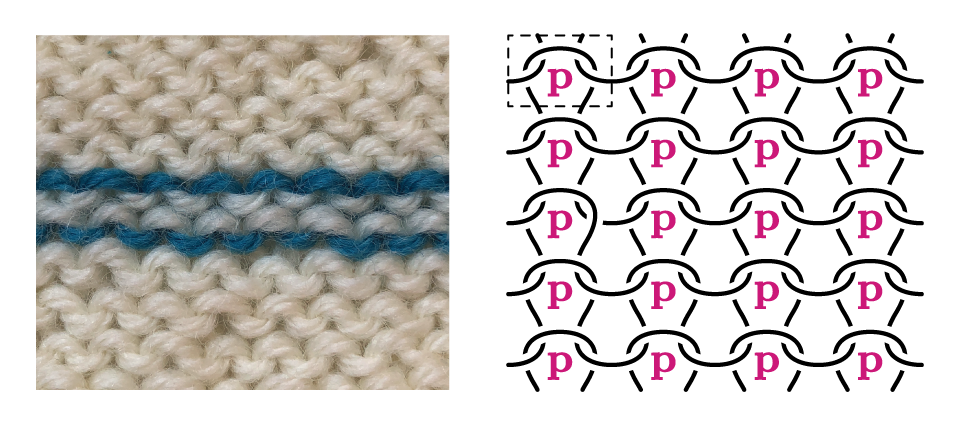} }
\hspace{4mm}
\subfloat[{1$\times$1 rib}.]
 {\includegraphics[width=70mm]{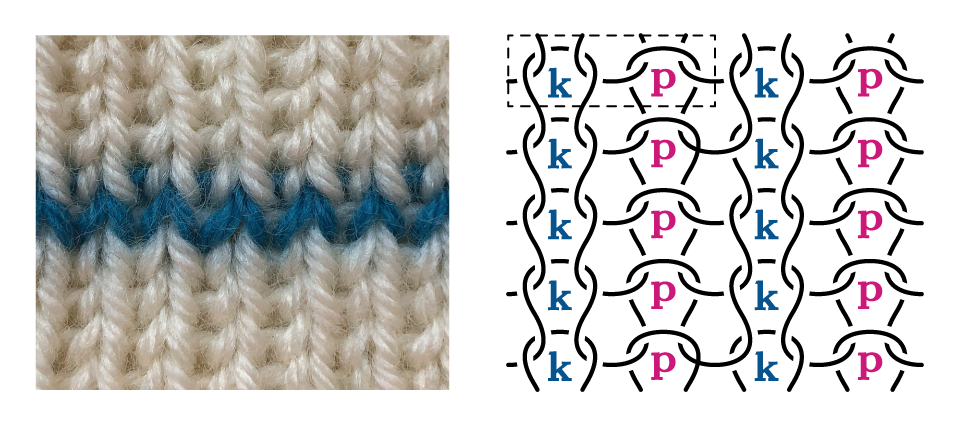} }\\
 \subfloat[{garter}.]
 {\includegraphics[width=70mm]{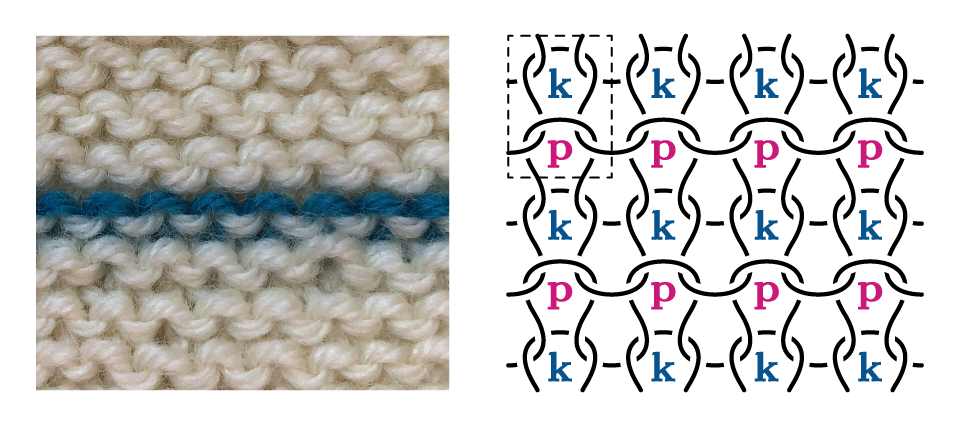} }
\hspace{4mm}
\subfloat[{seed}.]
 {\includegraphics[width=70mm]{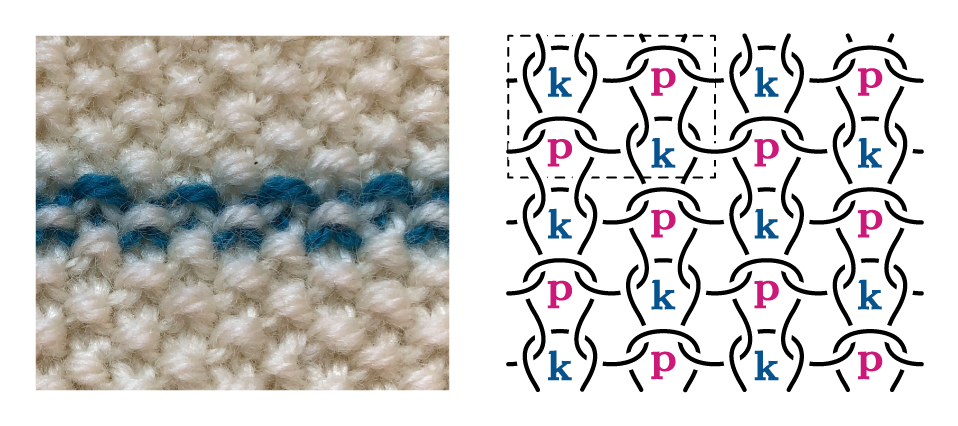} }
\caption{\label{fig:stitch_patterns} Images of some weft-knitted fabrics made from knit and purl stitches and the corresponding two-periodic versions.}
\end{figure*}   

We use \texttt{SnapPy} \cite{SnapPy} -- an open source python based software for studying low-dimensional topology -- and a {\tt Mathematica} based package {\tt KnotTheory} \cite{KTP}, to compute invariants of links in $S^3$. By examining the data for patterns, we propose the following conjectures governing some of the link invariants of links associated to swatches (the set of which includes textile links that can be knit into weft-knitted textiles).
\\\\
\textbf{Conjecture}~\ref{conjecture:hyp} [Hyperbolic swatches]: Let $L\subset T^2\times I$ be a swatch whose link complement can be obtained by combining link complements of a finite number of knits (knit swatches) and purls (purl swatches) using annulus sums~\ref{def:gluing}. Then the link $H\cup f(L)\subset S^3$ is hyperbolic.
\\\\ We have observed that there exist swatches that give rise to non-hyperbolic links, and these consist of irreducible swatches other than just knits and purls.
\\\\
\noindent\textbf{Conjecture}~\ref{conjecture:vol_sum} [Hyperbolic volume of one-component hyperbolic swatches]: Given two three-component hyperbolic links $H\cup f(L_1)\subset S^3$ and $H\cup f(L_2)\subset S^3$ corresponding to swatches $L_1\subset T^2\times I$ and $L_2\subset T^2\times I$, the hyperbolic volume of the three component link $H\cup f(L_1*_mL_2)\subset S^3$ corresponding to the swatch $L_1*_mL_2\subset T^2\times I$ is equal to the sum of the hyperbolic volumes of links $H\cup f(L_1)\subset S^3$ and $H\cup f(L_2)\subset S^3$. 
\\\\The link complement of the swatch $L_1*_mL_2$ is obtained by meridional annulus sum of the link complements of swatches $L_1$ and $L_2$. 
\\\\
\noindent\textbf{Conjecture}~\ref{conjecture:mva} [MVA of links corresponding to swatches]: Let $L_1,L_2\subset T^2\times I$ be two $n$-component links that arise as swatches. Let $H\cup f(L_1)\subset S^3$ and $H\cup f(L_2)\subset S^3$ be the corresponding $(n+2)$-component links. The multivariable Alexander polynomial of the $(n+2)$-component link $H\cup f(L_1*_mL_2)\subset S^3$ is given by the product of multivariable Alexander polynomials of $H\cup f(L_1)\subset S^3$ and $H\cup f(L_2)\subset S^3$ divided by the factor $(t_1-1)^n$. The variable $t_1$ is associated with the component $f(m)\subset H$ of the Hopf link.
\\\\
Many existing studies have focused on using topological invariants to distinguish between textile links, which is only a part of the classification problem. We add to the existing body of work on classifying weft-knitted stitch patterns (and thus characterizing weft-knitted textiles) by using link invariants as a tool to translate between knitting moves and rules for combining swatches.

\section{Two-periodic weft-knitted textiles: construction of swatches}
Although sweaters and socks may be the most recognizable knitted objects, commercial knitted textiles come in large bolts and are comprised of many repeating configurations. Ignoring the boundaries is a convenient way to capture the entanglement and knottiness of the strands of yarn in the bulk fabric. The process of knitting acts on a square lattice. 
All knitted textiles with translationally symmetric configurations can be extended to infinite \emph{two-periodic knits}.

By tracing the core of the yarn in a two-periodic knit, we obtain a countably infinite collection of infinitely long, embedded space curves. 
Since we are interested in the topological properties, we need only consider these space curves up to \emph{ambient isotopy}.

\subsection{Fundamental translational units and textile links}
\label{sec:fundamental_translational_unit}

In a general context, the notion of stitch is nuanced as it is used in different contexts in different communities of scholars. The process of making a ``stitch" in hand-knitting starts by taking loop on the left needle then using the right needle to pull a loop of working yarn through the loop on the left needle. This loop is then secured on the right needle. The process repeats with the next loop on the left needle. All of the new loops created on the right needle are held by the loops that were on the left needle. In case of stockinette textile, the making of which is described in Figure~\ref{fig:knitting}, the repeating motif in which one loop is held by a loop in the row immediately below is roughly what one means by a stitch. Although the idea of a ``stitch" is nuanced, we can associate a unique knot or link to a two-periodic knitted textile. 

Given a two-periodic weft-knitted textile, there is an average horizontal direction along which infinitely long strands of yarn are situated. In knitting literature, this direction is called the \emph{course direction} (or weft direction). Joining approximate centers of two or more interlocked bights across adjacent rows yields another average direction, called the \emph{wale direction} (or warp direction). The course and the wale directions are linearly independent and span a plane. 
A distinguishing feature between the wale and the course directions is, while bights in the course direction belong to a single string or a space curve, bights in the wale direction are parts of distinct space curves. The course and wale directions for five different textile samples are shown in Figure~\ref{fig:stitch_patterns}
and Figure~\ref{fig:pattern_to_unitcell}.

The weft and wale directions together provide a natural choice for picking coordinate axes. Without loss of generality, we choose the positive $x$-axis and $y$-axis to point along the course and the wale directions respectively. The orientation of the $z$-axis follows from the right hand rule. The periodicity due to the underlying translational symmetries induces a quotient map yielding a fundamental translational unit or a tile, which is homeomorphic to the 3-manifold $T^2\times\mathbb{R}$. We will call the curve traced by an edge parallel to the $y$-axis (aligned along the wale direction) on $T^2$ a \emph{meridian}, and one that is parallel to the $x$-axis (aligned along the course direction) is called a \emph{longitude}. This protocol is described in Figure~\ref{fig:pattern_to_unitcell} for the stockinette fabric, where the symbol \textbf{k} denotes a single \emph{knit} motif within the square block. The identification of the boundary edges of the block indicate that the ambient 3-manifold is homeomorphic to $T^2\times\mathbb{R}$. Similarly, the fundamental translational units of reverse stockinette (stockinette fabric viewed from back or the wrong side), $1\times1$ rib fabric, garter fabric and seed fabric are shown in Figure~\ref{fig:stitch_patterns}(a)-(d), respectively. The symbol \textbf{p} denotes a \emph{purl} motif. Since two-periodic weft-knitted textiles extend only a finite amount in the $z$-direction, we choose to describe the tiling units 
by the 3-manifold $T^2\times [0,1]$. The 3-manifold $T^2\times [0,1]$, called the thickened torus, is homeomorphic to the closure of the 3-manifold $T^2\times(0,1)$ (which is homeomorphic to $T^2\times\mathbb{R}$) inside $S^3$ or $\mathbb{R}^3$. Throughout the paper we will be denoting the thickened torus as $T^2\times I$.

The restriction of the quotient map to the corresponding collection of embedded space curves yields a \emph{knot} (or a \emph{link}) -- an embedding of a circle (or a finite collection of circles) into a 3-manifold. We will call links in $T^2\times I$ generated by this procedure as \emph{textile links} \cite{Markande_2020}. Now that we have established a correspondence between stitch patterns of two-periodic weft-knitted textiles and links in $T^2\times I$, we can study topological properties of the former by analyzing topological invariants of the latter. Many standard computational knot theory software packages (eg. \texttt{SnapPy} and \texttt{Knottheory}) are designed to work with knots and links in $S^3$. Therefore, we will Dehn fill $T^2\times I$ to get the set of links in $S^3$.

\subsection{From links in $T^2\times I$ to links in $S^3$} 

\begin{figure}
\centering
 \includegraphics[width=84mm]{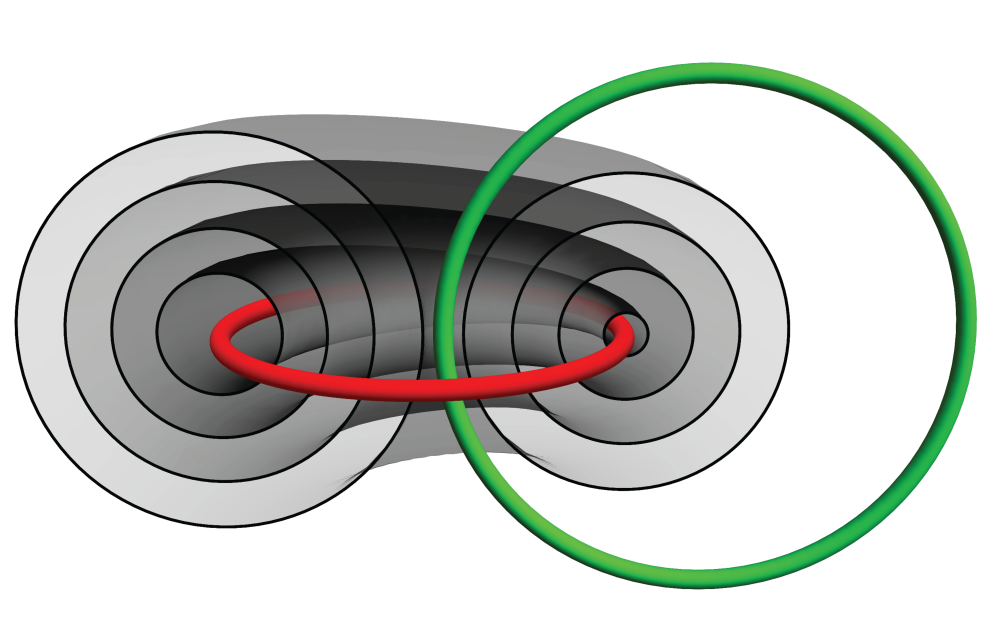}
\caption{\label{fig:hopf_link}The Hopf link, $H\subset\mathbb{R}^3$ along with multiple two-dimensional open sliced tori embedded in $\mathbb{R}^3\setminus H$. The complement of the Hopf link is $T^2 \times I$.} 
\end{figure}

Though we have used the term link complement throughout the introduction, in the interest of accessibility we will precisely define it here.

\begin{definition}[Link complement]
\label{definition:lc}
Given an n-component link $L$ in a 3-manifold $M$, let $\nu(L)$ be an open tubular neighborhood around the link $L$. We call the 3-manifold $M\setminus \nu(L)$ the link complement of $L$ in $M$. 
\end{definition}   
 A tubular neighborhood of the Hopf link, shown in Figure~\ref{fig:hopf_link}, is given by the union of a pair of disjoint tubular neighborhoods of its two components.
\noindent As a result, we see the ambient space that textile links are embedded in is homeomorphic to the complement of the Hopf link in $S^3$: the thickened torus $T^2\times I$. Thus, by choosing a map to embed the complement $Y=T^2 \times I \setminus \nu(L)$ of a textile link into $S^3$ (through \emph{Dehn filling} the boundary tori of $T^2\times I$), one can obtain a 3-manifold $Y^{\prime}$ that is homeomorphic to the complement of a link in $S^3$ that is associated with the original textile link.

 Note that the 3-manifold $Y^{\prime}$ in the above definition is uniquely determined up to homeomorphism using the Dehn filling construction \cite{Rolfsen_1976}.

\begin{figure}[h!]
\centering
\subfloat[A tubular neighborhood of a textile knot $K$, $\mathcal{T}_K\subset T^2\times I$. The green (left-right) and red (top-bottom) surfaces in $T^2 \times I$ are a pair of annuli that intersect along a single line. The horizontal and vertical markings denote meridians of the two boundary components (tori) that are highlighted as checkered purple and orange rectangles at the front and rear.]
 {\includegraphics[width=50mm]{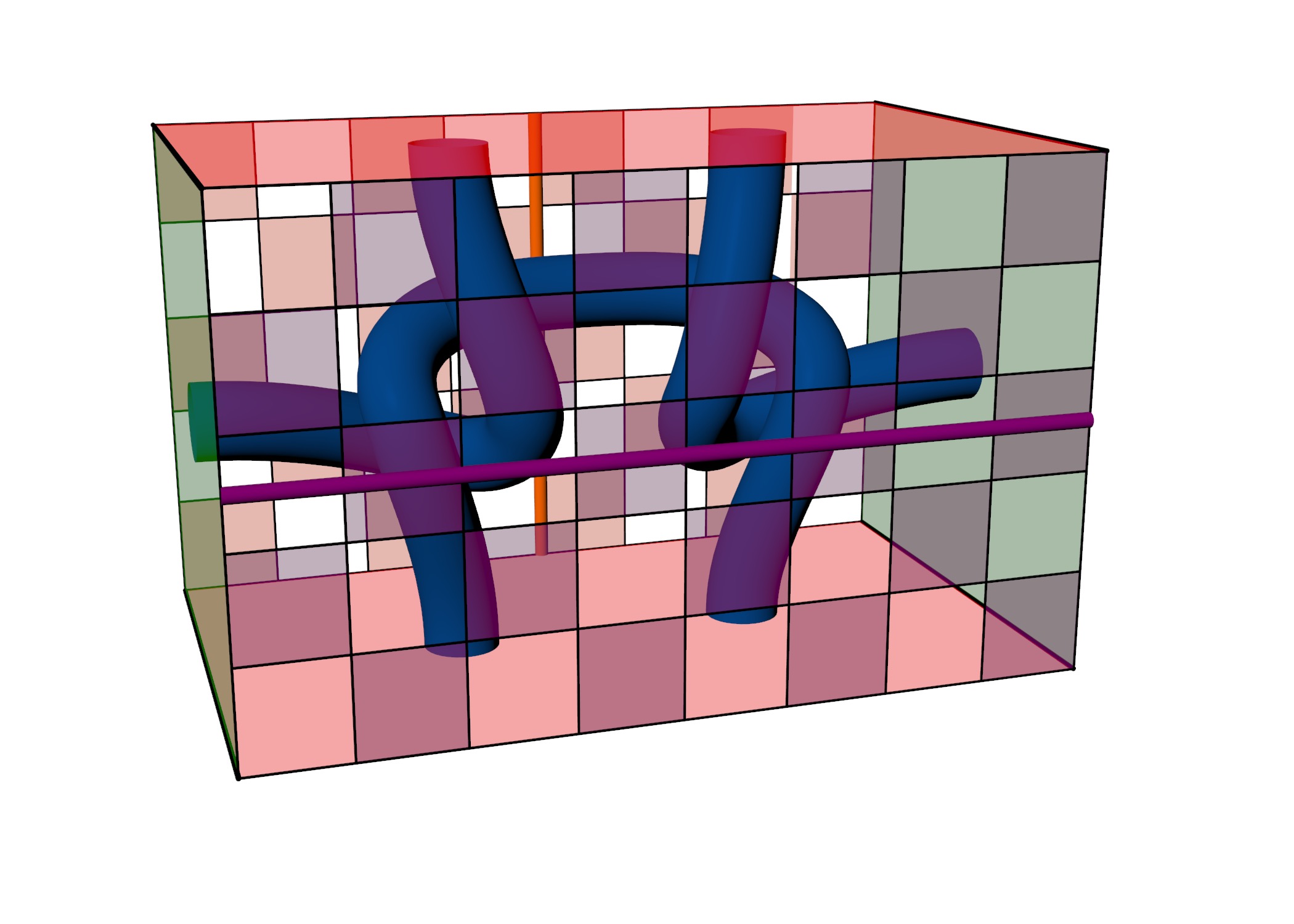} }
 \hspace{2mm}
\subfloat[A tubular neighborhood of the Hopf link $H\subset S^3$. The circular markings denote meridians of the boundary tori.]
 {\includegraphics[width=50mm]{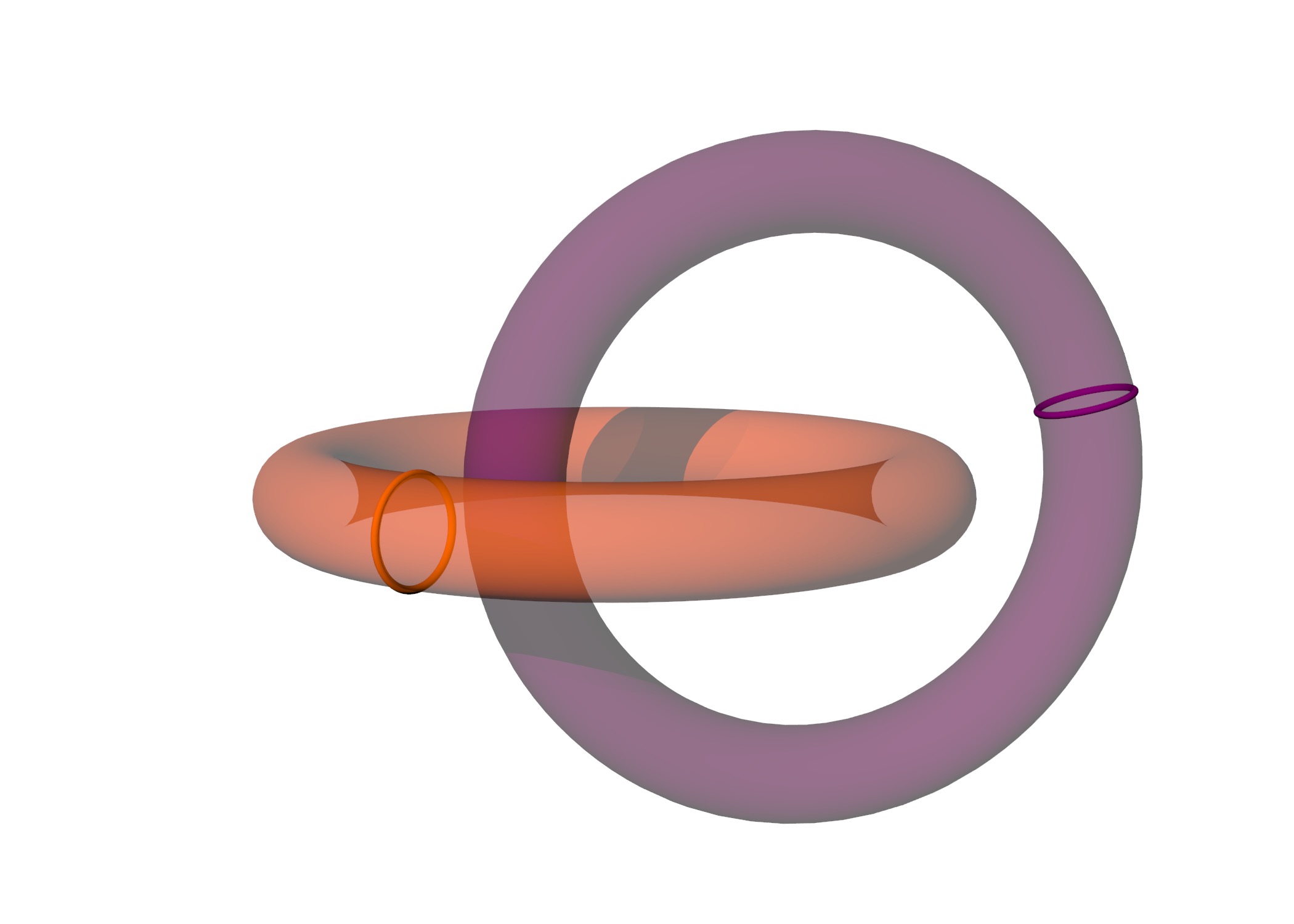} }
 \hspace{2mm}
\subfloat[Dehn filling the 3-manifold $(T^2\times I)\setminus\mathcal{T}_K$ using the mapping $f$ specified by the slopes of the purple and orange markings in (a) and (b) leads to $S^3\setminus\mathcal{T}_{f(K)}$. By the above discussion, the former is homeomorphic to the 3-manifold, $S^3\setminus\mathcal{T}_{H\cup f(K)}.$]
 {\includegraphics[width=50mm]{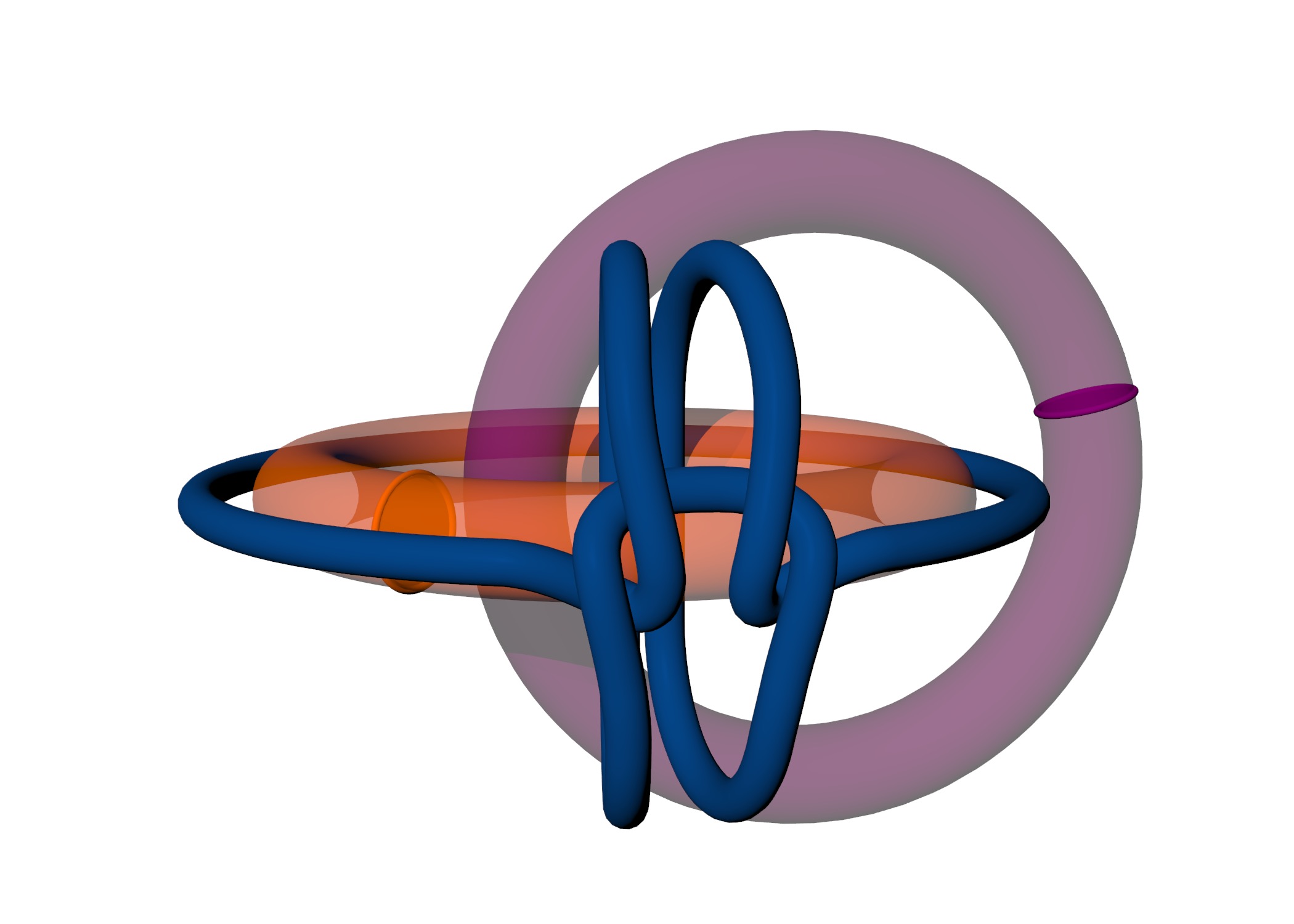} }
\vspace{-5pt}
\caption{\label{fig:T2xI}Textile knots naturally live in $T^2\times I$. However, we can construct textile knots as three component links in $S^3$ as illustrated in the above figures.} 
\end{figure}

Therefore we can view the complement of an $n$-component textile link $L$ in two ways: as
$T^2~\times~I~\setminus~\nu(L)$, or as $S^3 \setminus \nu(H \cup f(L))$ where $H$ is the Hopf link formed by the two boundary components of $T^2 \times I$ and $f(L)$ is the image of $L$ after Dehn filling $H$ as in Figure \ref{fig:T2xI}.

Given a two-periodic weft-knitted textile, the basis vectors along the \emph{course} and the \emph{wale} directions are used to obtain a translational unit. The translational units tile the two-periodic stitch pattern, forming an integer lattice. Any parallelogram of unit area with its vertices in this integer lattice tiles the underlying stitch pattern. 

Suppose the vectors $\{(p_1,q_1),(p_2,q_2)\}$ define such a parallelogram tile, where $p_1,p_2,q_1,q_2\in\mathbb{Z}$ are subject to the unit area condition. Then, $p_1q_2-p_2q_1$ must be equal to one. For $p_1,q_2$ equal to one and $p_2,q_1$ equal to zero, we have the standard tile with a square base, and the Dehn filling of $T^2\times I$ that yields $S^3$ is specified by the purple and orange curves shown in Figure~\ref{fig:T2xI}. Notice that in the basis $\{(1,0),(0,1)\}$, Dehn filling $T^2\times I$ with the square base along curves of slopes $p_1/q_1$ and $p_2/q_2$ is homeomorphic to the Dehn filling $T^2\times I$ with the parallelogram base (i.e. $\{(p_1,q_1),(p_2,q_2)\}$, along curves of slopes equal to $1/0$ and $0/1$). 
\begin{remark}\label{remark:dt}
Any $2 \times 2$ matrix chosen from $SL(2,\mathbb{Z})$, the group of determinant one matrices with integer entries,  gives rise to an area-preserving automorphism of $T^2\times I$. 
\end{remark}

The Dehn filling procedure and the \textbf{Remark}~\ref{remark:dt} together imply that any translational unit with a unit area parallelogram as its base yields a homeomorphism between the complement of a textile link $L\subset T^2\times I$ and the complement of a link $H\cup f(L)\subset S^3$.   
  
\subsection{Swatches}
Now that we have discussed textile links (which model a finished fabric), we will move on to introduce swatches, which model the act of knitting the fabric itself. In order to do this, we must first understand what it means for two knots or links to be the same.

Intuitively, we say two knots (or links with the same number of components) in a 3-manifold are \textit{ambiently isotopic} to each other if one can be continuously transformed into the other without cutting and pasting any segment of the associated space curve or without passing a pair of segments through each other. Such a transformation yields a continuous deformation of the ambient 3-manifold called an \emph{ambient isotopy}, which is a continuous family of orientation-preserving homeomorphisms, starting from the identity homeomorphism, from the ambient space to itself. 

We consider knots equivalent up to ambient isotopy (for discussion of knots up to ambient isotopy in arbitrary 3-manifolds, see \cite{ABCPR}). Similarly, for two links $L_1$ and $L_2$ in a 3-manifold $M$, the link $L_1\subset M$ is isotopic to the link $L_2\subset M$ if there exists an isotopy of $M$ whose restriction to the submanifold $L_1\subset M$ deforms it to the submanifold $L_2\subset M$ through a continuous family of homeomorphisms. If $M$ is either $S^3$ or $\mathbb{R}^3$, then an isotopy of $M$ acting on links is equivalent to performing a finite number of local moves on the link diagrams \cite{Reidemeister_1927, Alexander_1926} up to planar isotopies. These moves shown in Figure~\ref{fig:rm}(a) are called the \emph{Reidemeister moves} \cite{Reidemeister_1927}. An ambient isotopy between two links in $S^3$ in terms of Reidemeister moves and planar isotopies acting on link diagrams is shown in Figure~\ref{fig:rm}(b). We will consider textile links equivalent up to ambient isotopy in the same way; two textile links are isotopic if their respective components are related by a sequence of Reideimester moves and planar isotopies. For brevity, in the rest of the paper we will refer to the first Reidemeister move by RM1, second Reidemeister move by RM2, and the third Reidemeister move by RM3. 

\begin{figure*}[h!]
{\includegraphics[width=140mm]{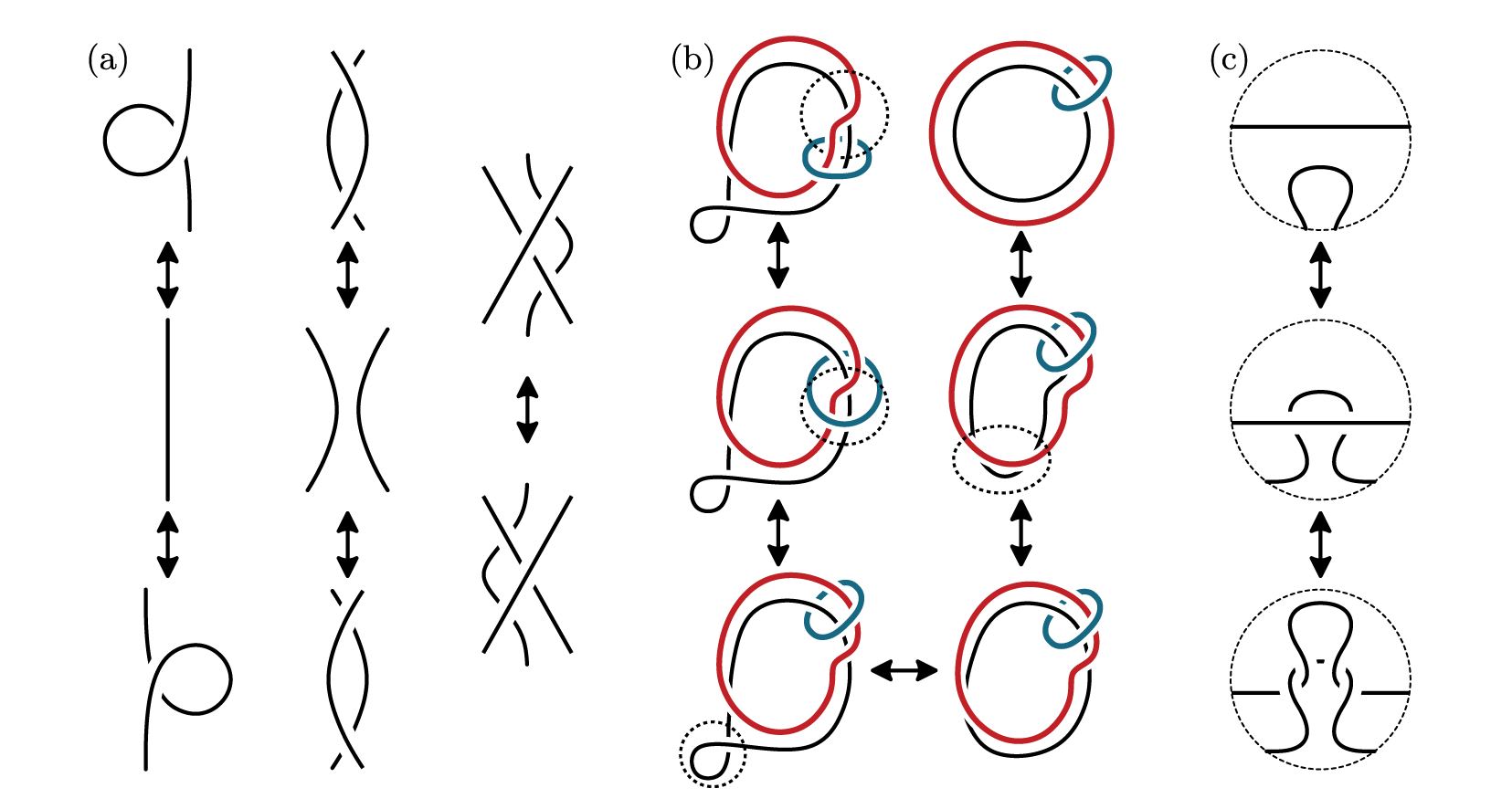}}

\caption{(a) The three Reidemeister moves. (b) An example of isotopic $3$-component links and the series of Reidemeister moves and planar isotopies showing their equivalence. (c) An illustration of the progression in making the \emph{knit} motif in the planar diagram notation\ref{fig:knitting}.} 
\label{fig:rm}
\end{figure*} 

Recall that while knitting a stockinette fabric, a slip loop (or a bight)  is constructed via the following four steps: 1) Push the right needle into the bight closest to the end of the left needle, 2) wrap the working end of the yarn around the right needle in clockwise direction (when facing into the page), 3) pull the right needle along with the wrapped segment of yarn through the bight on the left needle, and 4) slip the left needle out of the bight it is holding. The local motifs shown in Figure~\ref{fig:knitting} illustrating the above steps are shown again in Figure~\ref{fig:rm}(c) without the colored disks that represent the knitting needles. We note by an inspection of Figure~\ref{fig:rm}(a) $\&$ Figure~\ref{fig:rm}(c) that making a slip loop while knitting stockinette fabric is equivalent to simply performing two consecutive RM2 moves. Thus, at this stage in the description of weft-knitting, the motif resulting from the construction of a slip loop locally (shown at the bottom of Figure~\ref{fig:rm}(c)) is topologically trivial. In order to arrive at a non-trivial textile link starting from this topologically trivial motif, we use the technique of \emph{band surgery}. To describe the transition from a topological trivial configuration of strands of yarn to a textile link, we need to introduce the concepts of \emph{unknits} and \emph{band surgeries}. Recall that an unknot in $S^3$ can be defined as a knot that bounds an embedded disk. This definition can be extended to knots in a $3$-manifold $M$: we say $U \subset M$ is an unknot if it is the boundary of a disk $D$ such that $D$ lives entirely in a subset of $M$ homeomorphic to a $3$-ball. Similarly, an $n$-component unlink in $M$ is a disjoint collection of $n$ disjoint, simple closed curves which bound $n$ disjoint, embedded disks, all of which can be isotoped to live entirely inside a subset of $M$ homeomorphic to a $3$-ball.

\begin{definition}[Unknit]
\label{def:unknit}
An unknit of $n$ components $\mathcal{U}_n$ is an ordered $n$-component unlink in the thickened annulus $ A \times I$. 
\end{definition}
\noindent Each of the component of $\mathcal{U}_n$ represents a loop of yarn on the knitting needle. A planar diagram of the unknit is shown in Figure~\ref{fig:unknit}a. From now on through out the paper, whenever we refer to an unknit drawn as a planar diagram, we specifically mean link diagrams of the kind shown in FIG~\ref{fig:unknit}a where the annulus (or later torus) is in the plane of the page and the thickening direction is perpendicular to it.

\begin{figure}[h!]
\centering
\subfloat[An $n$-component unknit.]{
\includegraphics[width = 60mm]{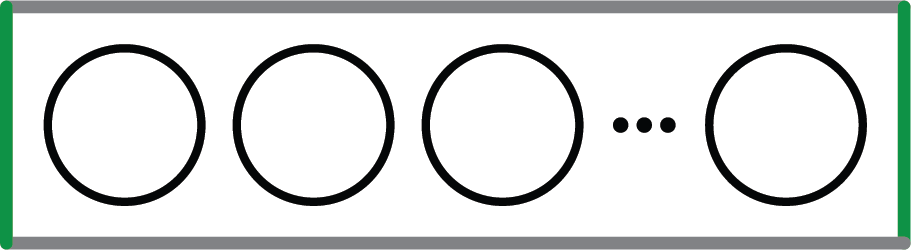}
} \hspace{2mm}
 \subfloat[A trivial row is the union of an $m$-component unlink and a loop in the annulus direction, representing the working yarn that is not yet on the needle.]{
 \includegraphics[width = 60mm]{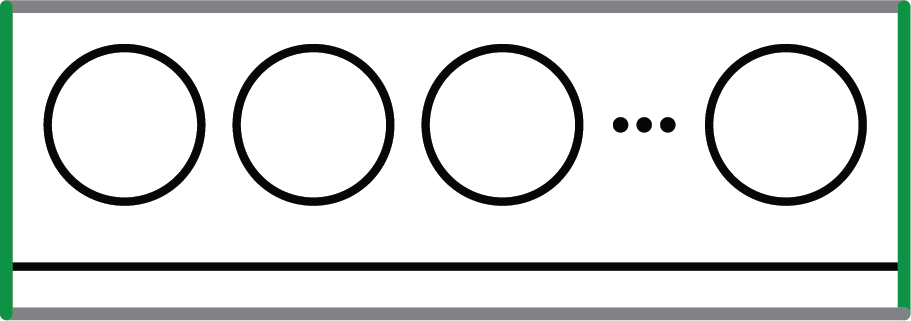}
 }
\caption{\label{fig:unknit} } 
\end{figure}

\begin{definition}[Row]
\label{def:row}
A trivial row of $n$ components $\mathcal{L}_n$ is an $n+1$-component link in the thickened annulus $ A  \times I$, which is composed of $n$ component unlink with an ordering $(1,\cdots,n)$ and an essential curve generating $H_1(A \times I, \mathbb{Z})$.
\end{definition}
\noindent A planar diagram of the trivial row is shown in Figure~\ref{fig:unknit}b.

The process of knitting uses working yarn -- represented by a longitudinal loop on the annulus -- to pull bights through the loops on the needle and pass those loops onto the other needle. We can think of each of the loops on the needles as if they are temporarily surrounding punctures $D^2 \times I$ in the thickened annulus $A \times I$. When taking a stitch, the loop of yarn on the needle interacts with a bight of yarn from the working yarn. It is dropped from the needle, thus removing the temporary puncture, and a bight from the working yarn is then transferred to the needle. Instead of working with the idea of dropping bights and picking new ones up using the needles as punctures in the manifold, we will consider the yarn on the needles as fixed and use \emph{band surgery} to connect the working yarn to the stitches on the new needle. The ordering on the contractible loops in the unknit $\mathcal{U}_n$ or the row $\mathcal{L}_n$ is given by the order in which they appear on the needle.

\begin{definition}[Band surgery, \cite{tristram_1969}]
\label{def:bs}
Let $L$ be a link in a 3-manifold $M$, which can be $S^3$, $\mathbb{R}^3$ or $\Sigma\times I$, where $\Sigma$ is an oriented closed 2-manifold. And let $b:[0,1]\times[0,1]\rightarrow M$ be an embedding, referred to as a band. The band $b\subset M$ is said to be compatible with the link $L\subset M$ if $b([0,1]\times[0,1])\cap L = b([0,1]\times\{0\})\cup b([0,1]\times\{1\})$ meaning, two non-adjacent boundary components of the band $b$ overlap with the link $L$ in two disjoint arcs. In this case the link $(L-b([0,1]\times\{0\})\cup b([0,1]\times\{1\}))\cup b(\{0\}\times[0,1])\cup b(\{1\}\times[0,1])$ will be denoted by $bL$, and it is the link obtained by the band surgery of link $L\subset M$ with respect to band $b\subset M$.
\end{definition}

\begin{figure}[h!]
{\includegraphics[width = 84mm]{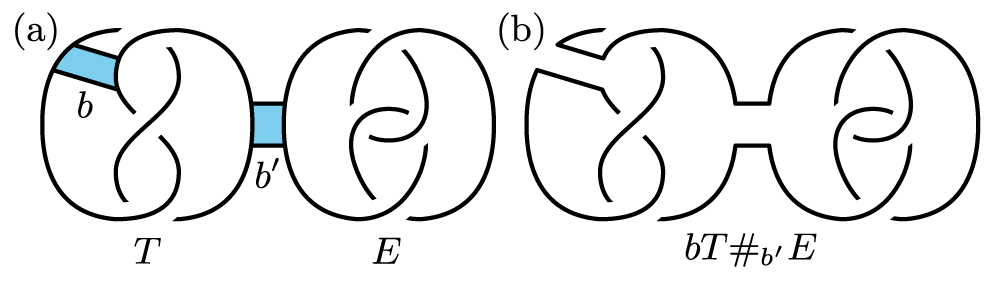}}
\caption{(a) The configuration of knots and bands before band surgery. $T$, $E$ denote the right handed trefoil knot and the figure eight knot respectively. (b) The 2-component link after the band surgeries is depicted.
}
\label{fig:connect_sum}
\end{figure}

\noindent A band surgery involving a single band between a pair of knots in $\mathbb{R}^3$ separated by a disjoint plane is called the \emph{connected sum} of those two knots. If the band surgery labeled by $b$ in Figure ~\ref{fig:connect_sum} were not done, the diagram would illustrate the connected sum operation between the right-handed trefoil knot and the figure eight knot by doing surgery on the remaining band $b^{\prime}$.

We now define knitting mathematically.

\begin{definition}[A knit of $q$ rows]\label{def:knitting}As knitting is an inductive process, we will define a knit by induction.
\begin{enumerate}
    \item (For $q=1$): Let $A$ be an annulus $I\times I /((0,t)\sim(1,t))$. The horizontal direction is the the direction we knit in (also known as the course direction or the weft direction) and the vertical direction with coordinate $t$ assigned to it is the direction that new rows are added (also called the wale direction or the warp direction.) The annulus $A$ is thickened to allow for strands to cross.

Let $L_0$ be an unknit of $n$ components $\mathcal{U}_n$ in the image of $(I \times [0, 1/4)) \times I$ and $L_1$ be a trivial row which is the union of an $m$ component unlink $\mathcal{U}_{1_m}$ and an essential curve $\ell_1$ living in $A \times \{0\}$ which generates $H_1(A \times I)$ and is in the image of $(I \times (1/4, 3/4)) \times I$. Furthermore, $\mathcal{U}_{1_m}$ is in the image of $(I \times (3/4,1])$ and is arranged horizontally according to the ordering of the components of the unlink. See Figure~\ref{fig:knit_row}a. 

We will define knitting diagrammatically in the thickened rectangle avoiding the edges (and then projecting down to the thickened annulus) by two operations: 
\begin{enumerate}
    \item A finite number of Reidemeister moves and planar isotopies can act on the longitude $\ell_1$ in the image of the image of $(I \times [0,3/4))\times I$. These isotopies can let $\ell_1$ interact with itself and with the components of $L_0$, but not with the contractible components of $L_1$. 
    \item Each of the $m$ components of the unlink of $L_1$ are joined via band surgery to $\ell_1$ (Figure~\ref{fig:knit_row}c).  
The original $n$ components of $L_0$ remain and the $m+1$ components of $L_1$ have now been joined into a single essential curve generating $H_1(A \times I)$. The result of this process is called the first row of our knit $\mathcal{R}_1$.
\end{enumerate}
\item (For $q=2$): To add a second row of knitting, we shall begin with a row $\mathcal{R}_1$ in a thickened annulus $A \times I$, defined in part (1). To this, place a trivial row $L_2$ above $\mathcal{R}_1$ consisting of $k$ contractible components and a longitude $\ell_2$, as shown in Figure~\ref{fig:knit_row}e. Knitting this row consists of nearly the same operations as before. 

\begin{enumerate}
    \item A finite number of Reidemeister moves and planar isotopies act on $\ell_2$. These can be between $\ell_2$ and itself or any component of $R_1,$ but not with the $k$ contractible components of $L_2$.
    \item Each of the $k$ contractible components of $L_2$ are joined via band surgery to $\ell_2$. The resulting link $\mathcal{R}_2$ has $n+2$ components.
\end{enumerate}
\item(For $q>2$): The procedure to add new rows to the knit is a generalization of definition~\ref{def:knitting}(2). Each new row consists of adding a trivial row $L_{r+1}$ of $n_{r+1}$ trivial components and one longitude $\ell_{r+1}$ to the annulus ``above'' the link $\mathcal{R}_r$. The longitude $\ell_{r+1}$ can be isotoped in $A\times I$ such that it interacts with itself and any of the components of $\mathcal{R}_r$. Each of the contractible loops of $L_{r+1}$ are joined by band surgery to $\ell_{r+1}.$
\end{enumerate}
\end{definition}

\begin{figure}[t!]
\subfloat[]{
\includegraphics[width = 60mm]{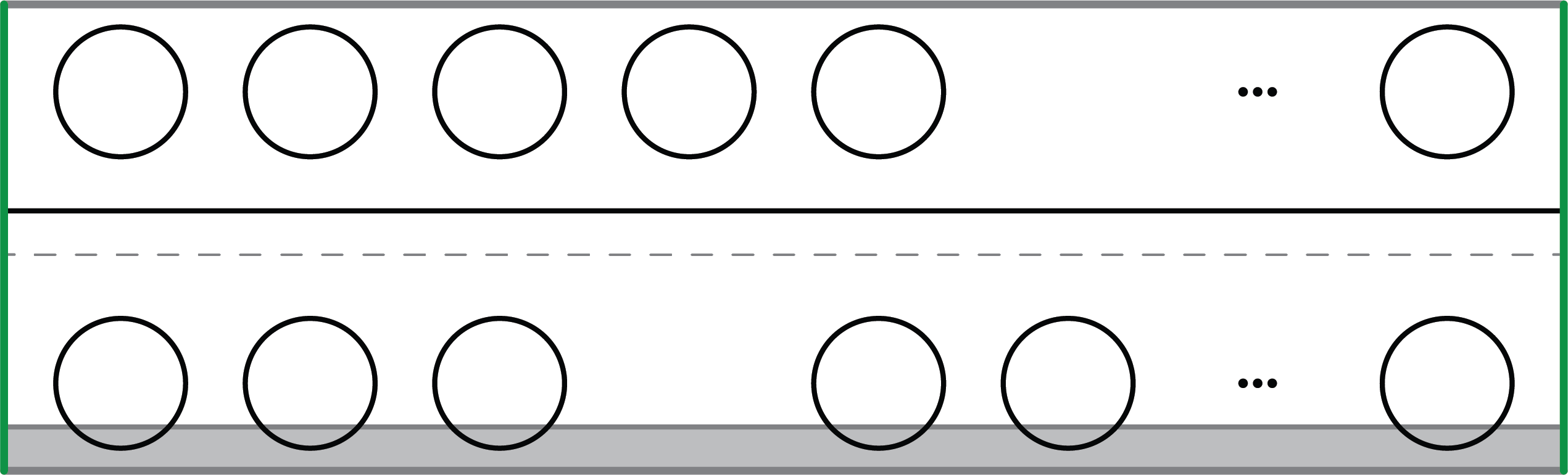}
}
\subfloat[]{
\includegraphics[width = 60mm]{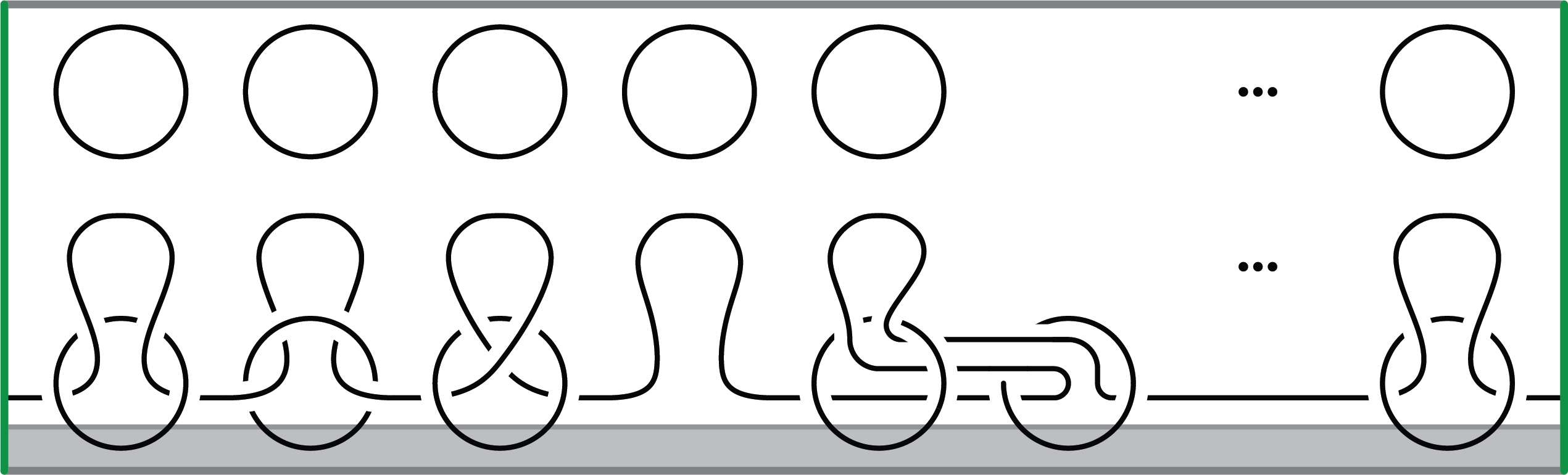}
}

\subfloat[]{
\includegraphics[width = 60mm]{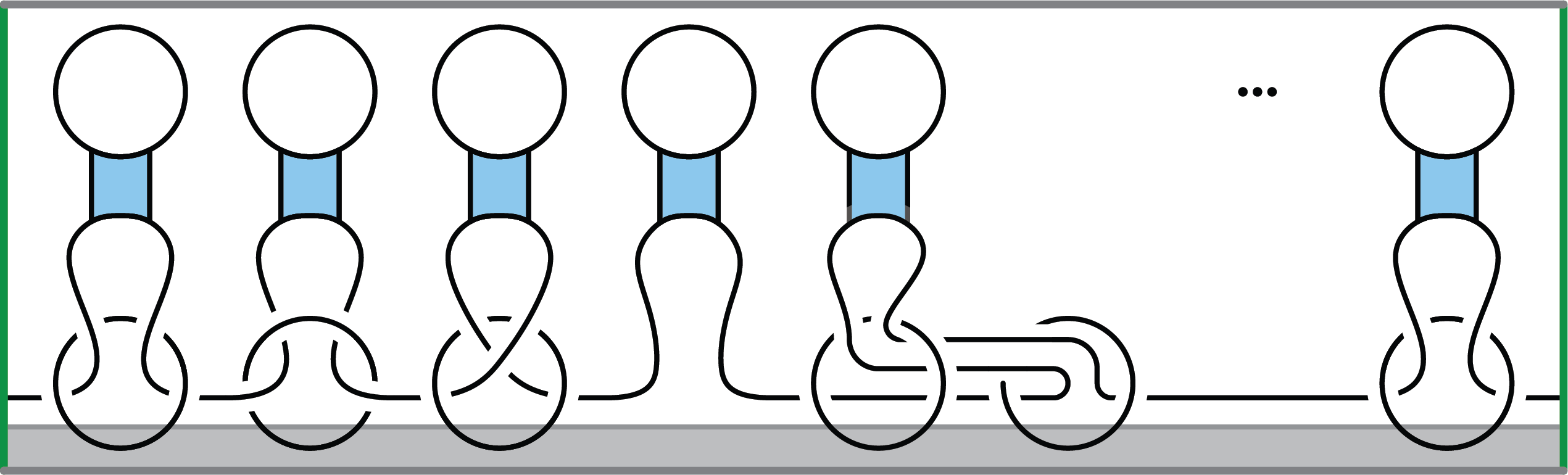}
}
\subfloat[]{
\includegraphics[width = 60mm]{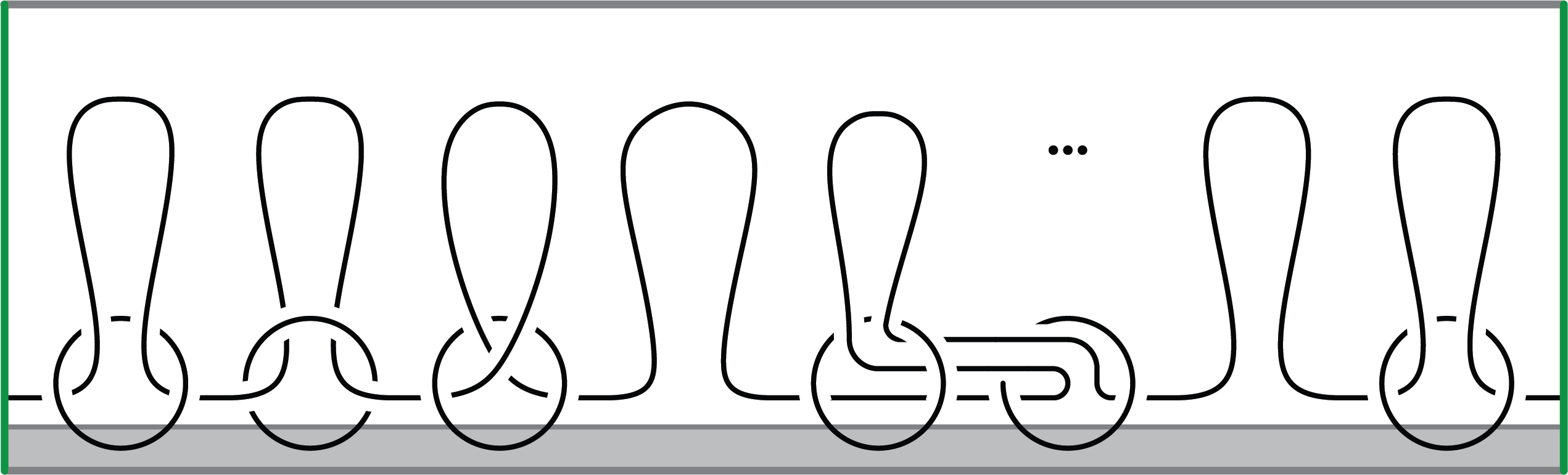}
}

\subfloat[]{
\includegraphics[width = 60mm]{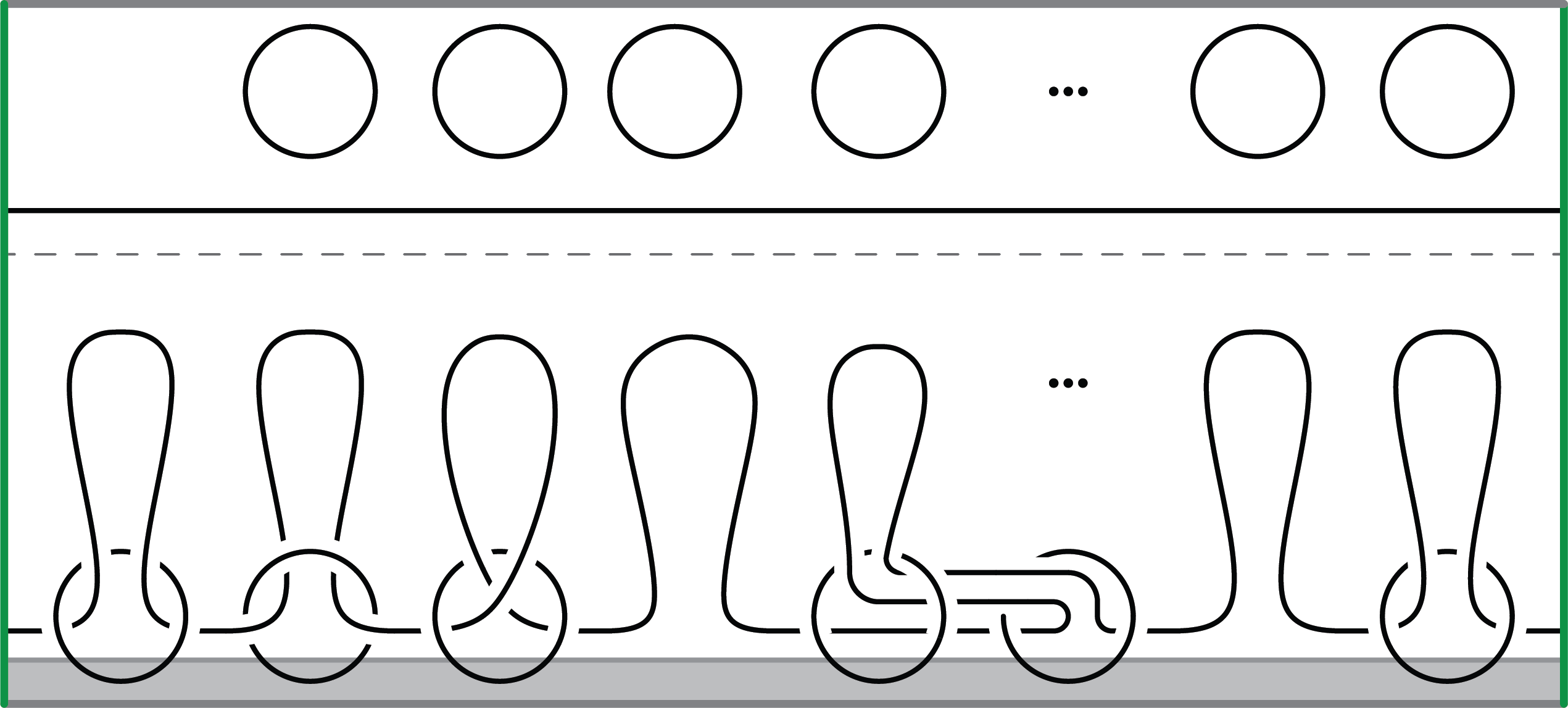}
}
\subfloat[]{
\includegraphics[width = 60mm]{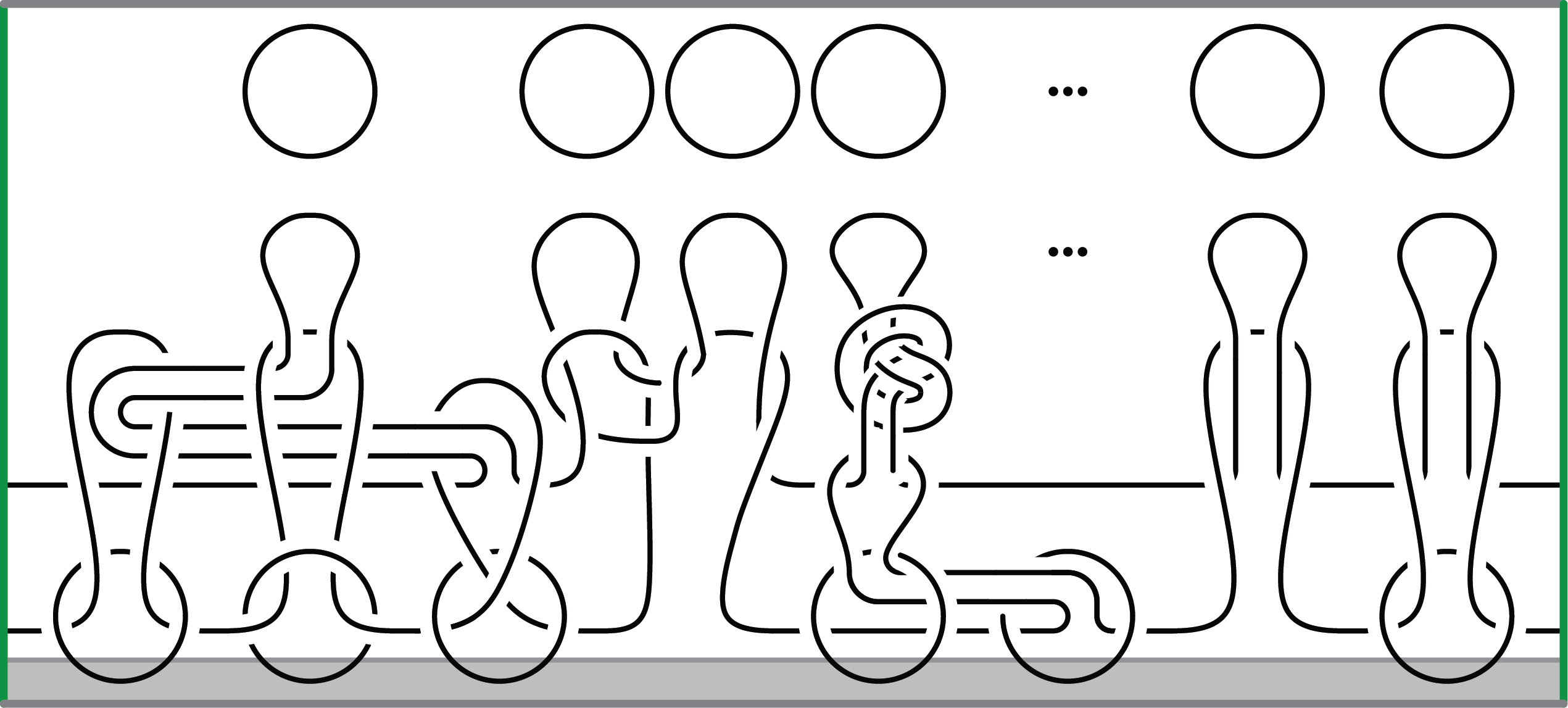}
}

\subfloat[]{
\includegraphics[width = 60mm]{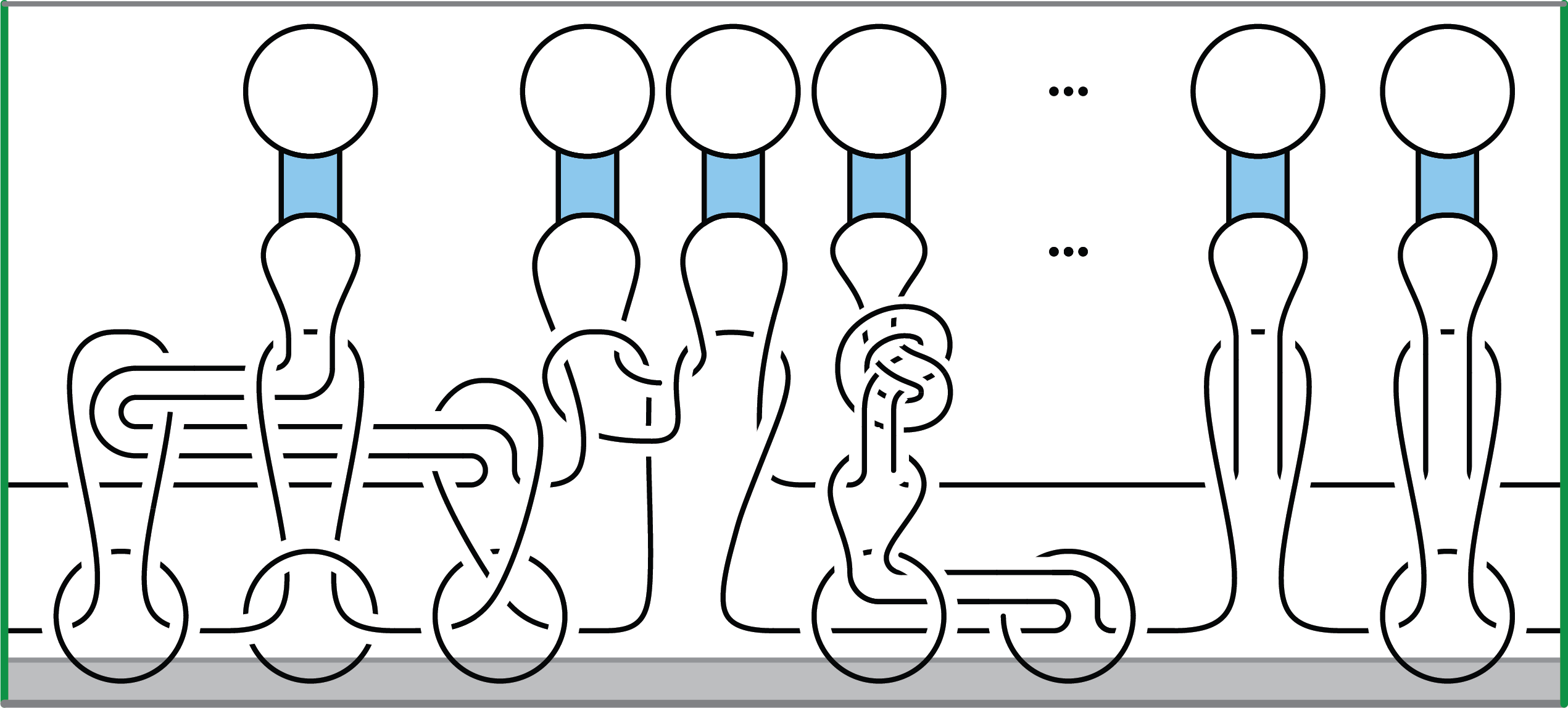}
}
\subfloat[]{
\includegraphics[width = 60mm]{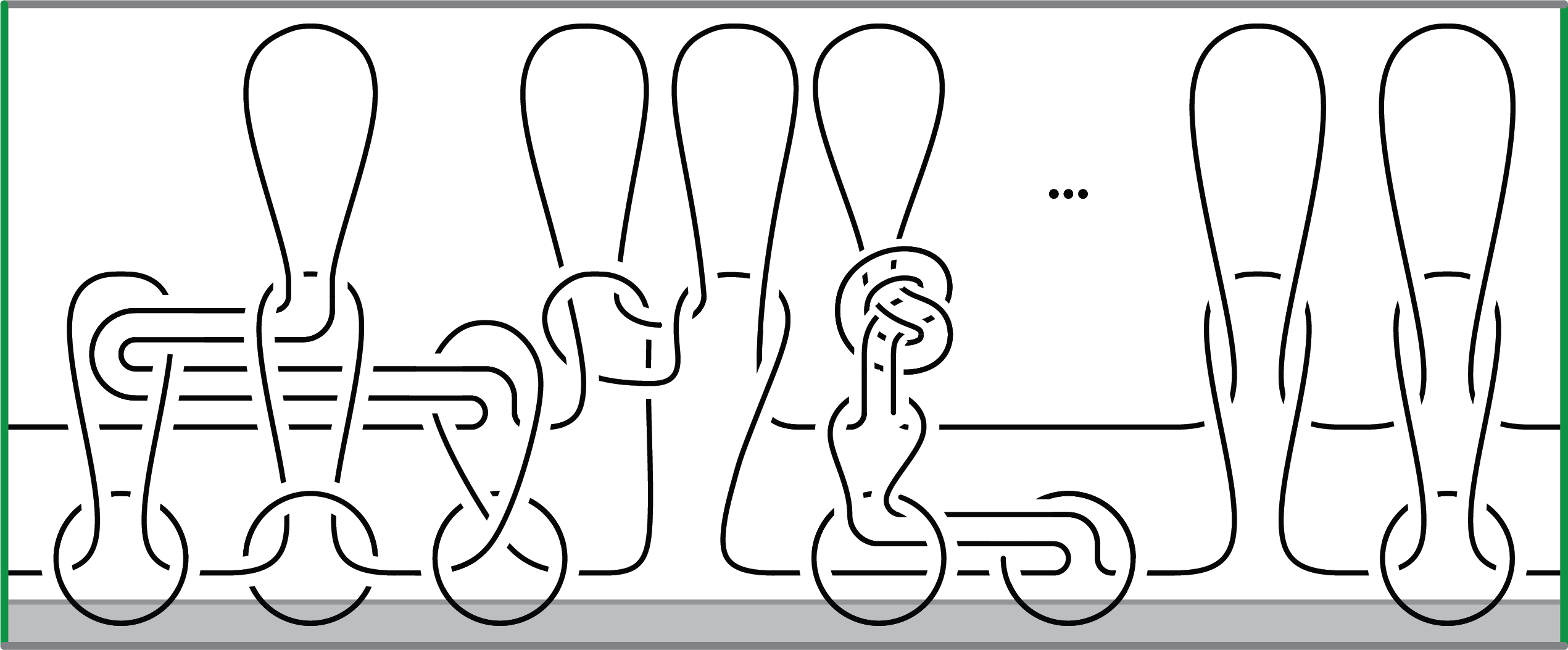}
}

\caption{
}
\label{fig:knit_row}
\end{figure}

Therefore, each row with $m$ stitches that we have ``knit'' consists of taking a longitude (the working yarn) and joining it by a band to each loop on the needle (represented by the components of an unlink).

Next we define the concept of a swatch by formalizing different stages involved in the construction of a two-periodic knitted textile. Note that a knit lives inside a single thickened annulus, and our goal is to study textiles using links inside a thickened torus. Swatches emulate the process of knitting captured by the definition above, but live in the desired ambient space.

A schematic of our method to describe textile links, in terms of link diagrams is illustrated in Figure~\ref{fig:swatch_construction}(a)-(d).

  \begin{definition}[Swatch \cite{Markande_2020}]\label{def:swatch}
Let $A$ be an annulus contained in $T^2\times\{0\}$. Let $\mathcal{R}_{k-1}$ be a knit of $k-1$ rows contained in $A \times I$ which started with an unlink $U_{n}$ of $n$ components. To this construction we add a final knitted row $L_k$ with $n$ ordered contractible components and one longitude $\ell_k$. Each of the $n$ contractible components of $\mathcal{U}_n$ are arranged along the lower boundary of the annulus according to their order and the $n$ contractible components of $L_k$ are arranged according to their order along the top of the annulus. This is shown in Figure ~\ref{fig:swatch_construction}(a), where we see the grey boundary of the annulus sitting inside the torus obtained by identifying opposite sides of the rectangle. Following Definition~\ref{def:knitting}(3), $\ell_k$ can be isotoped in the thickened annulus and then connected to the contractible components of $L_k$ via band surgery retaining the ordering of contractible components, as shown in Figure~\ref{fig:swatch_construction}(b). The boundaries of the thickened annulus are identified, with no monodromy along the annulus direction. 

The final step, shown in Figure~\ref{fig:swatch_construction}(c) and Figure~\ref{fig:swatch_construction}(d), is creating a link in $T^2 \times I$ via band surgery. During this step, every contractible loop $U_{n_i}$ from $U_n$ with ordering $i$ gets connected to each contractible loop $L_{k_i}$ from $L_k$ with ordering $i$, where the ordering on each set of loops is identified, as shown in FIG~\ref{fig:swatch_construction}(c). 
As the whole point of this procedure is to build a mathematical object that models the act of knitting, this final step should replicate a fixing procedure on the fabric motif rather than introducing new structure to the motif. Therefore, we require these final bands be in \textbf{knitting position}.

We say that a finite collection of bands $\{b_i:[0,1]\times[0,1]\rightarrow T^2\times I\}_{i=1}^n$ is in the knitting position if the following hold: 
\begin{enumerate}
\item The bands are mutually disjoint, unknotted and simple meaning that there are no crossings involving the bands and the components of the unknit when drawn as a planar diagram.
\item All the bands cut through the lower boundary of the 3-manifold $A\times I$ i.e., if $A$ is given by $[0,1]\times S^1$, then the bands intersect (pass through) the surface $(\{0\}\times S^1)\times I$ exactly once.  
\item All the bands are compatible with the components of the unknit $L\subset T^2\times I$ such that every band connects a contractible component to a longitudinal component of the unknit.
\end{enumerate} 

 Once we have found these knitting position bands, such as the cyan bands shown in Figure~\ref{fig:swatch_construction}(c). Then, we can perform band surgeries with respect to these bands $\{b_i\}_{i=1}^n$, we obtain an $k$-component link given by $b_n(\cdots(b_2(b_1L)))\subset T^2\times I$. The order in which the band surgeries are done is immaterial, and the final link is called an $n\times k$ swatch. For example, the link shown in Figure~\ref{fig:swatch_construction}(d) is a swatch.  
\end{definition}

\begin{figure}[h!]
\subfloat[]{
\includegraphics[width = 60mm]{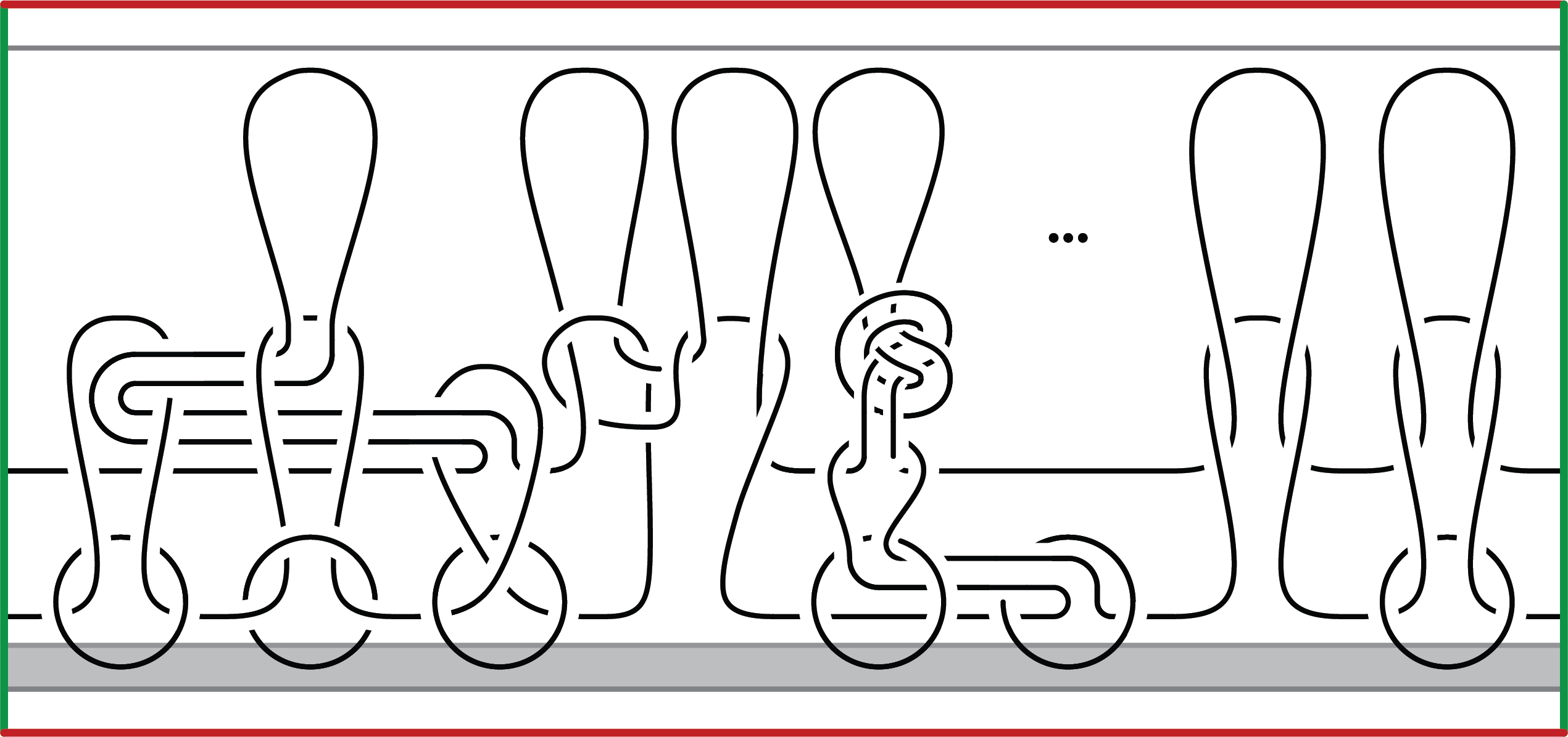}
}
\subfloat[]{
\includegraphics[width = 60mm]{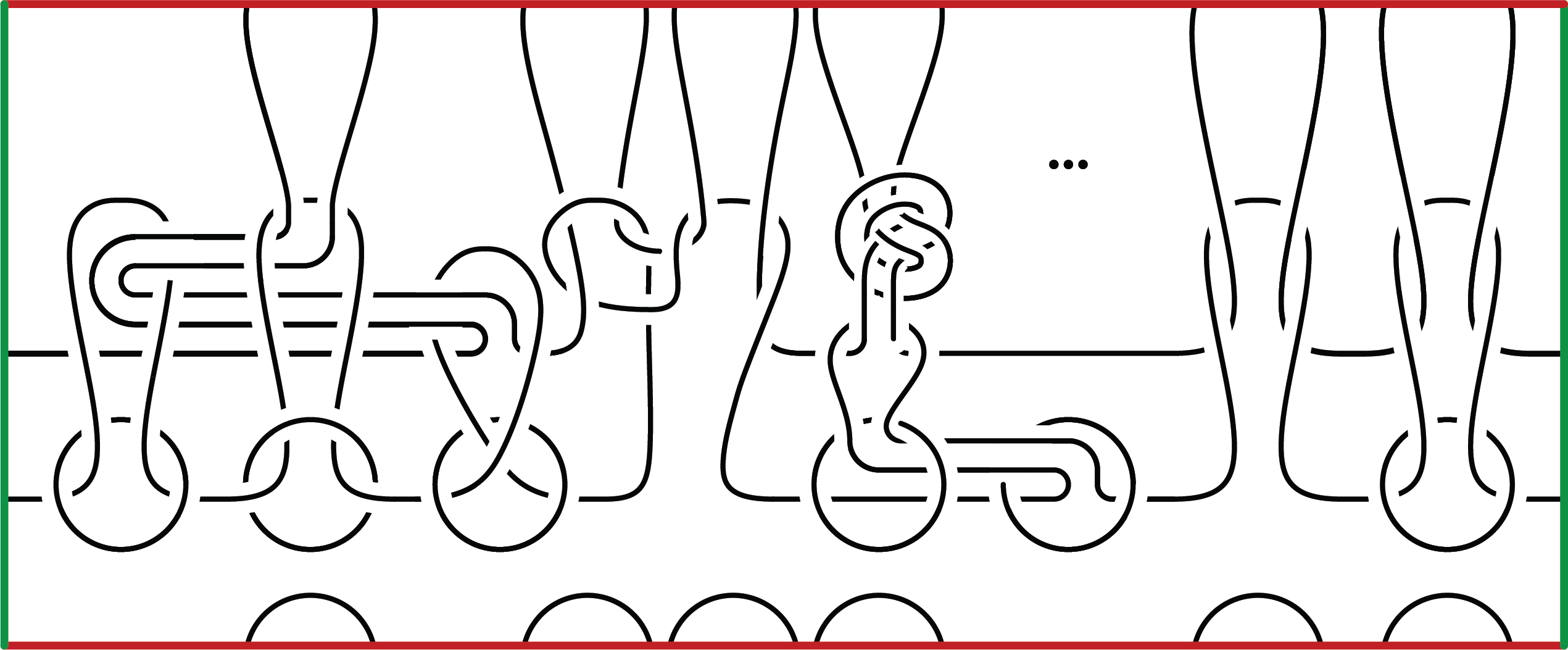}
}

\subfloat[]{
\includegraphics[width = 60mm]{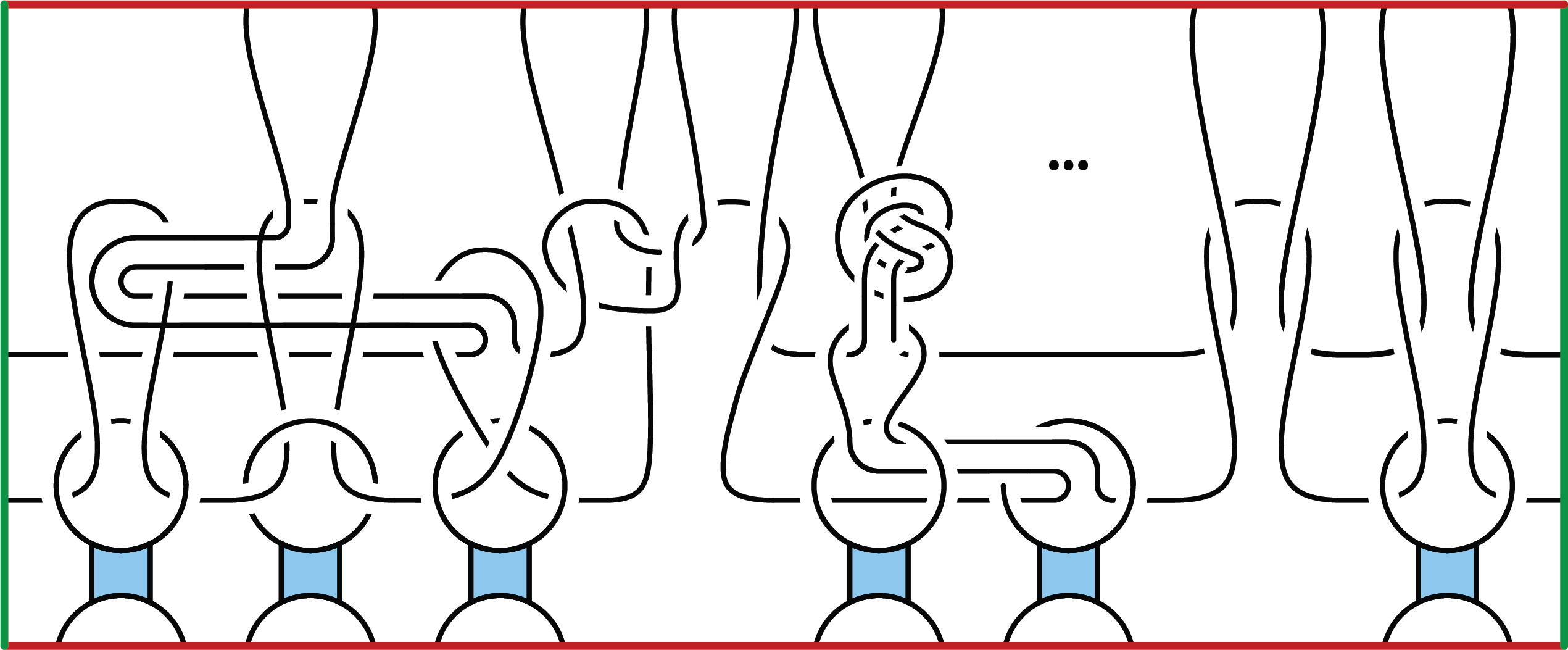}
}
\subfloat[]{
\includegraphics[width = 60mm]{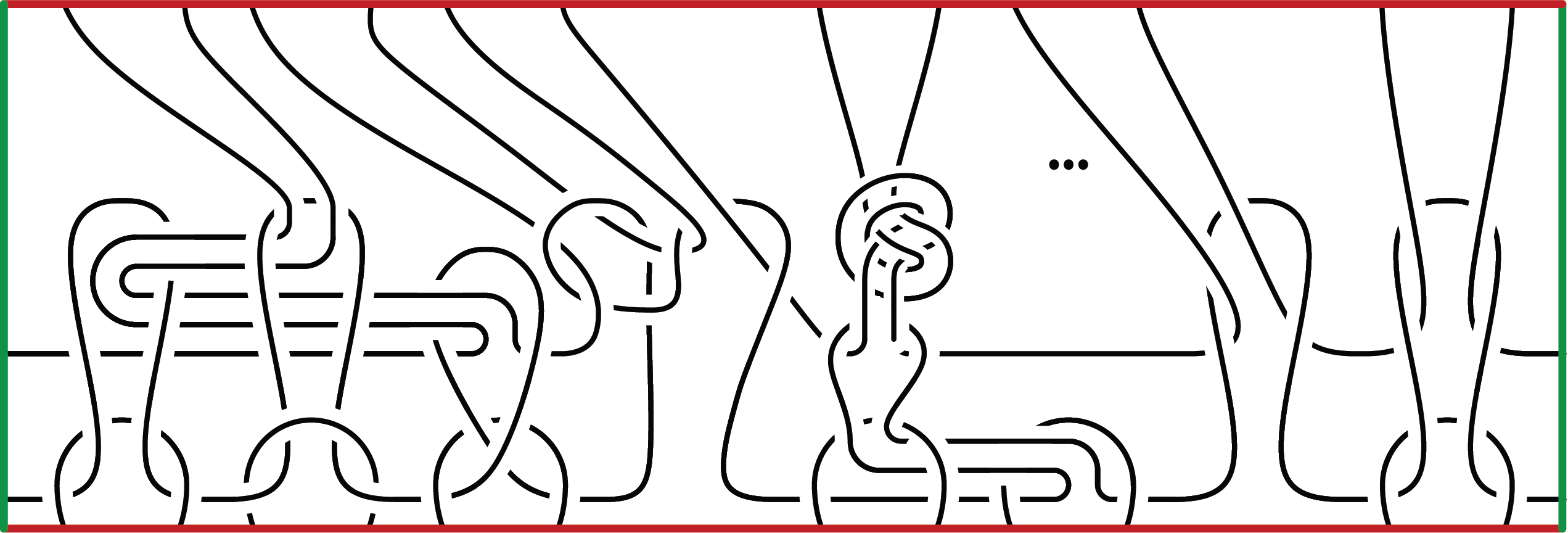}
}

\caption{
}
\label{fig:swatch_construction}
\end{figure}

A key aspect of the definition of an $n\times k$ swatch and its construction is that the homology classes of the components of a swatch are identical to the homology classes of the non-contractible loops in the unknit that it is constructed from. Thus, each component of a swatch is homologous to the longitude of the base torus of $T^2\times I$. This means, despite of having infinitely many choices for fundamental tiling units, our choice of unknit (due to the homology classes of its components) dictates that the basis vectors are along the course and wale directions. 

In other words, if we were to choose to use a different quotient map to construct tiles, the links inside may be topologically different from their course-wale basis counterparts depending on the slopes of our chosen basis. As a result, we could accordingly modify the definition of an $n\times k$ unknit and an $n\times k$ swatch by starting with trivial links whose non-contractible components belong to the homology classes of the components of the links in the modified tiles of our choice.

\begin{remark}[Trivial $n$-component swatch]
In the definition of a swatch, if no isotopies are performed on the unknit before doing the knitting position band surgeries, then the only possibility is attaching all the null-homotopic components to the uppermost longitudinal component. Then the $n$-component unknit yields the $n$-component trivial swatch.
\end{remark}

\noindent The construction of textile links corresponding to stockinette and reverse stockinette fabrics starting from  $1\times1$ swatches is illustrated in Figure~\ref{fig:knit_purl}.
\begin{figure}[h!]
\centering
\includegraphics[width=174mm]{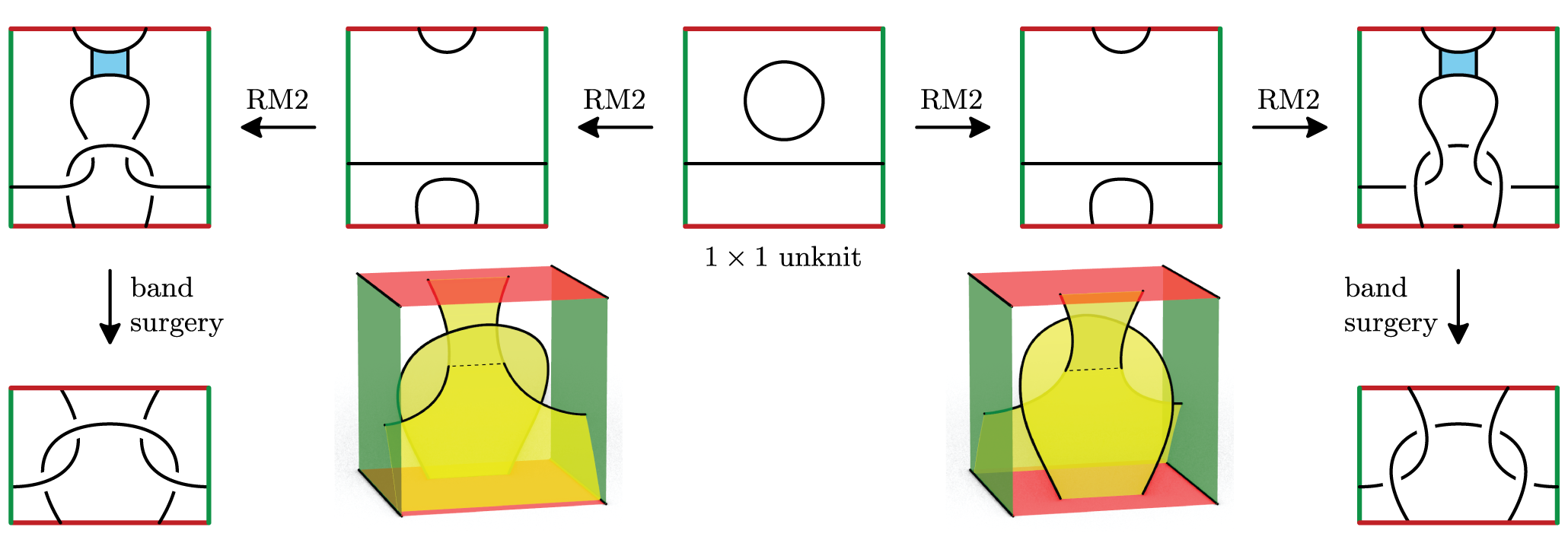}
\caption{\label{fig:knit_purl}(a) A purl swatch and its ribbon surface. (b) A knit swatch and its ribbon surface. The auxiliary unknots and the ribbon singularities of the ribbon surfaces are shown as black and red dashed arcs respectively.} 
\end{figure}

\begin{remark}\label{remark:textile_link_to_swatch}
Given a two-periodic stitch pattern associated with any weft-knitted textile, the corresponding textile link can be constructed as a non-trivial $m\times n$ swatch for some $m,n\in\{1,2,\cdots\}$. 
\end{remark}
\begin{proof} This is clear from the physical construction of a textile link (as the quotient of a doubly-periodic motif of a stitch pattern) and the mathematical construction of a swatch, the swatch simply organizes components of a textile link into distinct rows of knitting.   
\end{proof}
\noindent However, the question of whether an arbitrary motif (which is not necessarily doubly periodic) of a weft-knitted textile can be knitted into a non-trivial swatch is not well posed unless all the permissible mechanical moves are specified.

\begin{remark}\label{remark:reflected_swatch}
The motifs described by a swatch and the swatch obtained by rotating it about $y$-axis (in $\mathbb{R}^3$) generate two sides of the same textile.
\end{remark}
\noindent For example, the knit and the purl generate two sides of the stockinette fabric. Nevertheless, it is necessary to distinguish between such pairs of swatches because, $1\times1$ rib, garter, seed and stockinette fabric, all give rise to topologically distinct motifs that are not related by a rotation in $\mathbb{R}^3$.

The description and the method described above for generating motifs that can be realized through weft-knitting excludes textile links consisting of essential loops that are not homologous to the longitude of the base torus. However, depending on the choice of basis, the translational units that tile the stitch pattern can give rise to links with essential components which are not homologous to a longitude. In spite of this subtlety, as long as we fix our choice of basis, which is dictated by the course and wale directions, we are able to define links whose motifs can be knit and formulate the concept of swatches consistently.

\subsection{Two-periodic weft-knitted fabrics and ribbon links}
This work has so far centered on modeling motifs in knitted fabric using link in $3$-manifolds other than $S^3$, however, every textile link (or swatch) can be turned into a link in $S^3$ in a well-defined way using the Dehn filling procedure defined in previous sections. Therefore, this section will examine the image of the set of textile links and prove that techniques developed for the study of slice and ribbon links (defined in the introduction) are useful in this setting. We will refer to the image of a textile link $L \subset T^2 \times I$ under this Dehn filling as $f(L)$ and conversely refer to $L$ as the \emph{precursor link} to the link $f(L) \subset S^3$.

We will use the formulation of ribbon disks as immersed disks in $S^3$ or $\mathbb{R}^3$ whose singularities exist in pairs known as \emph{ribbon singularities} or \emph{ribbon intersections} (surveyed nicely in \cite{eisermann2009}). We can further think of general immersed surfaces (perhaps with genus) with only ribbon singularities and call them ribbon surfaces. For example, next to the $1\times 1$ swatches in Figure~\ref{fig:knit_purl}, a 3D rendering of a pair of immersed surfaces with knit and purl swatches as one of their boundary components is shown. Notice that the geometry of the self-intersection of these immersed surfaces is identical to that of a slit (dashed arc) cut open by a ribbon (yellow annulus) passing through the interior of itself. These intersections are particularly special as if we view the aforementioned $S^3$ or $\mathbb{R}^3$ as the boundary of $B^4$ or a component of the boundary of $\mathbb{R}^4$, respectively, we can push the interior of a ribbon surface into this 4-dimensional space in order to obtain an embedded surface in $4$-dimensions. 

We will follow the work of \cite{eisermann2009} and represent ribbon surfaces with band diagrams as below. Note that not all links are ribbon; for example, any link with nonzero linking number (defined in section 3) cannot be ribbon \cite{Rolfsen_1976}, thus the Hopf link is not ribbon.

\begin{definition}[Band diagrams and ribbon surfaces, \cite{eisermann2009}] 
\label{bds}
A planar diagram of a ribbon link $L$ in $S^3$ that is composed of only the local motifs shown in Figure~\ref{fig:band_diagrams}(a), is called a band diagram of $L\subset S^3$. As a result, there exists a surface $S_L\subset S^3$ whose boundary is given by the ribbon link $L$, $\partial S_L \cong L$. The surface $S_L\subset S^3$ is called a ribbon surface.
\end{definition}
\noindent The ribbon surfaces need not be orientable. However, for a \emph{ribbon knot}, as we define below, the ribbon surface must be a disk, and thus, a ribbon knot bounds an orientable ribbon surface. %

 \begin{figure}[h!]
\centering
\subfloat[A ribbon knot.]{
\includegraphics[width = 60 mm]{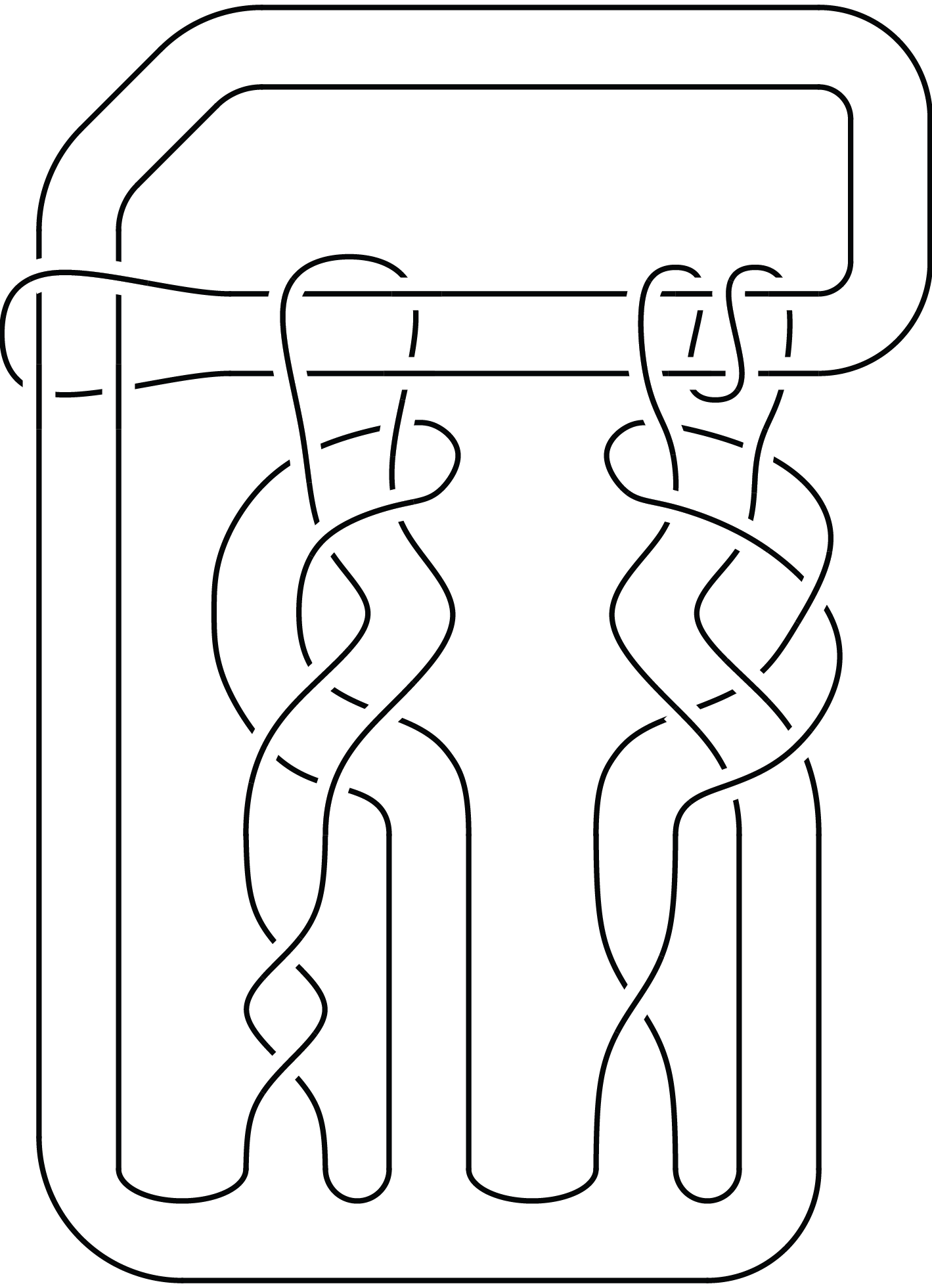}
}
\hspace{2mm}
\subfloat[A two-component ribbon link.]{
\includegraphics[width = 41 mm]{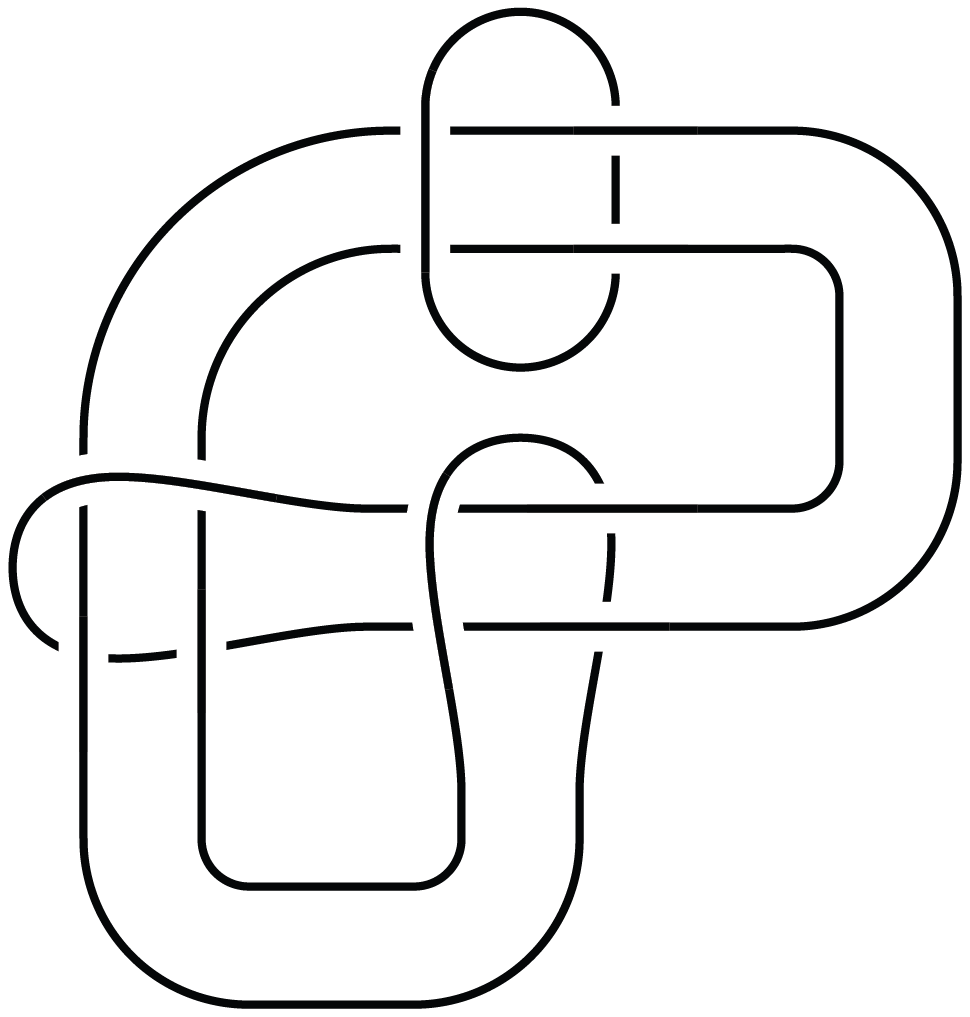}
}
\caption{\label{fig:band_diagrams}} 
\end{figure}

\begin{remark}[Ribbon knots, links, and band diagrams]
If a knot in $S^3$ admits a band diagram such that the corresponding ribbon surface is homeomorphic to a disk immersed in $S^3$, then the knot is a ribbon knot and the disk is a ribbon disk. 

Similarly, given an $n$-component link $L\subset S^3$, if all of its components are ribbon knots and the only intersections of the corresponding collection of $n$ ribbon disks are ribbon singularities, then the link $L$ is a ribbon link. 
\end{remark}

\noindent Band diagrams of a ribbon knot and a ribbon link in $\mathbb{R}^3$ are shown in Figure~\ref{fig:band_diagrams}(a) and Figure~\ref{fig:band_diagrams}(b) respectively. A ribbon singularity occurring between a pair of components of a ribbon surface is called a \emph{mixed ribbon singularity}, otherwise it is called a \emph{pure ribbon singularity} \cite{eisermann2009}.

\begin{figure}[h!]
\centering
\includegraphics[width = 41 mm]{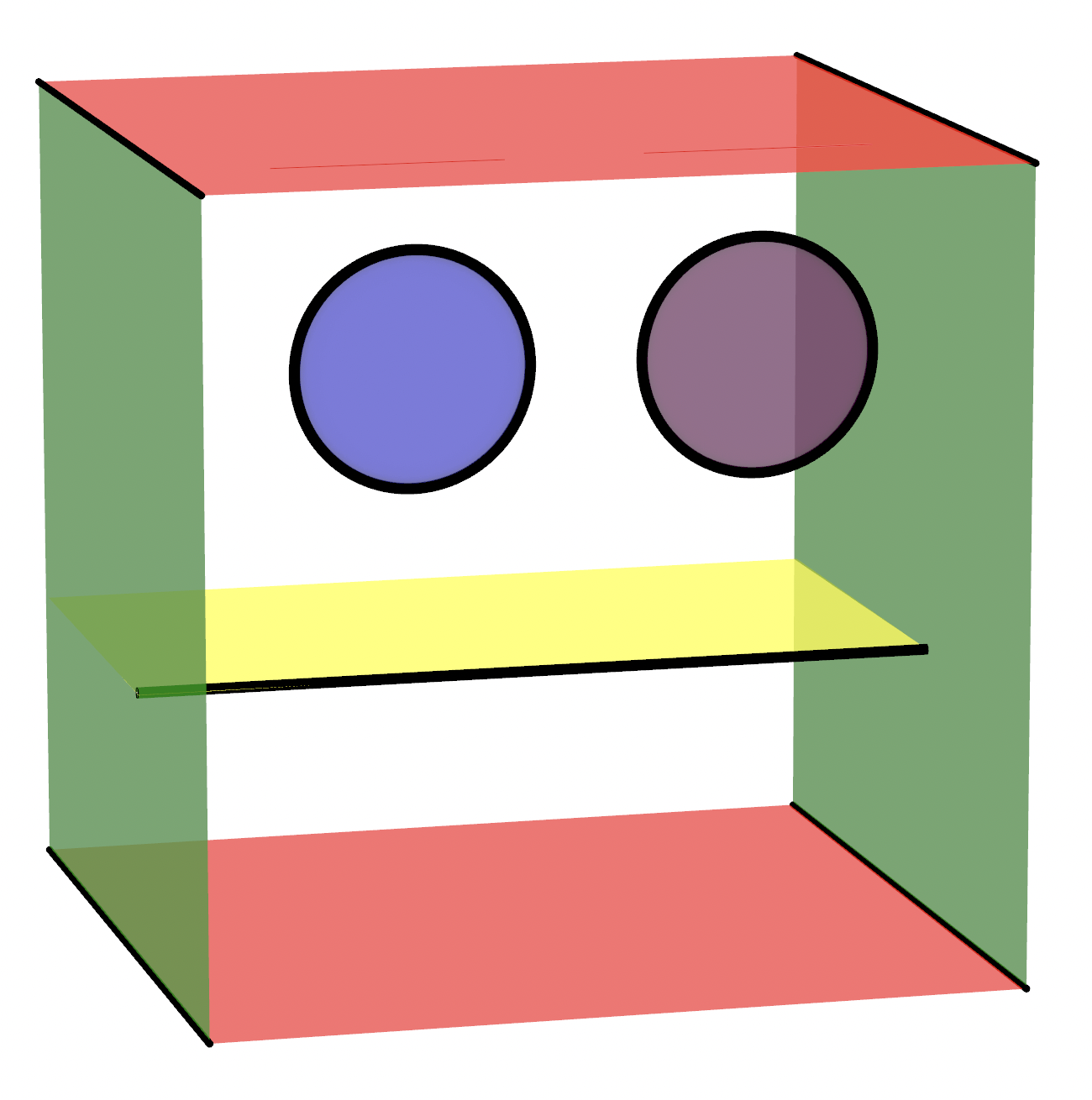}
\caption{\label{fig:ribbon_4D} Precursor ribbon in $T^2\times I$.} 
\end{figure}

\noindent We can now state and prove a necessary property of links in $T^2\times I$ that can be knit into motifs of two-periodic knitted textiles.
\begin{theorem}[Two-periodic weft-knitted textiles and ribbon links]
\label{theorem:ribbon}
Let $S$ be a two-periodic weft-knitted textile, and let $L^{S}\subset T^2\times I$ be the corresponding $n$-component textile link. Then the link $L\subset T^2\times I$ is a precursor to an $n$-component ribbon link $f(L^{S})\subset S^3$.   
\end{theorem}

\begin{proof}[Proof of Theorem \ref{theorem:ribbon}] 
Given that the boundary homeomorphism is defined by curves of slopes $1/0$ and $0/1$ such that the corresponding Dehn filling of $T^2\times I$ yields $S^3$, an $n$-component link $L = K_1\cup...\cup K_n$ in $T^2\times I$ is a precursor to an $n$-component ribbon link $f(L)\subset S^3$ if the following hold: Miriam edit: why is part 2 true?
\begin{enumerate}
\item for every null-homotopic component $K_i\subset L$, there exists an embedded disk $D_i\subset T^2\times I$ such that $\partial D_i=K_i$.
\item for every  homotopically non-trivial component $K_j\subset L$, there exists an annulus $A_j\subset T^2\times I$ such that $\partial A_j = K_j\cup\alpha_j$, where $\alpha_i$ is a simple closed curve on $T^2\times\{0\}$ -- the orange boundary torus in Figure~\ref{fig:T2xI} that is homotopic to either the orange or the purple curve in Figure~\ref{fig:T2xI}.  
\end{enumerate}
Furthermore, if the link $L\subset T^2\times I$ is a precursor to a ribbon link $f(L)\subset S^3$, then the 2-manifold $\bigcup_j A_j\bigcup_i D_i\subset T^2\times I$ is a precursor to a collection of $n$ ribbon disks in $S^3$ with their boundary along $f(L)$. 
%explain how we do the filling so that we always have these disks (esp since we are not in S3)
An $n$-component ribbon link is a link which bounds $n$ disjoint disks in $B^4$ such that the radial function on the 4-ball restricts to a Morse function on the disk that has only local minima. Equivalently, the link results from adding bands to an unlink . This is a standard fact in the study of knot concordance. 

\begin{figure}[h!]
\centering
\includegraphics[width = 160 mm]{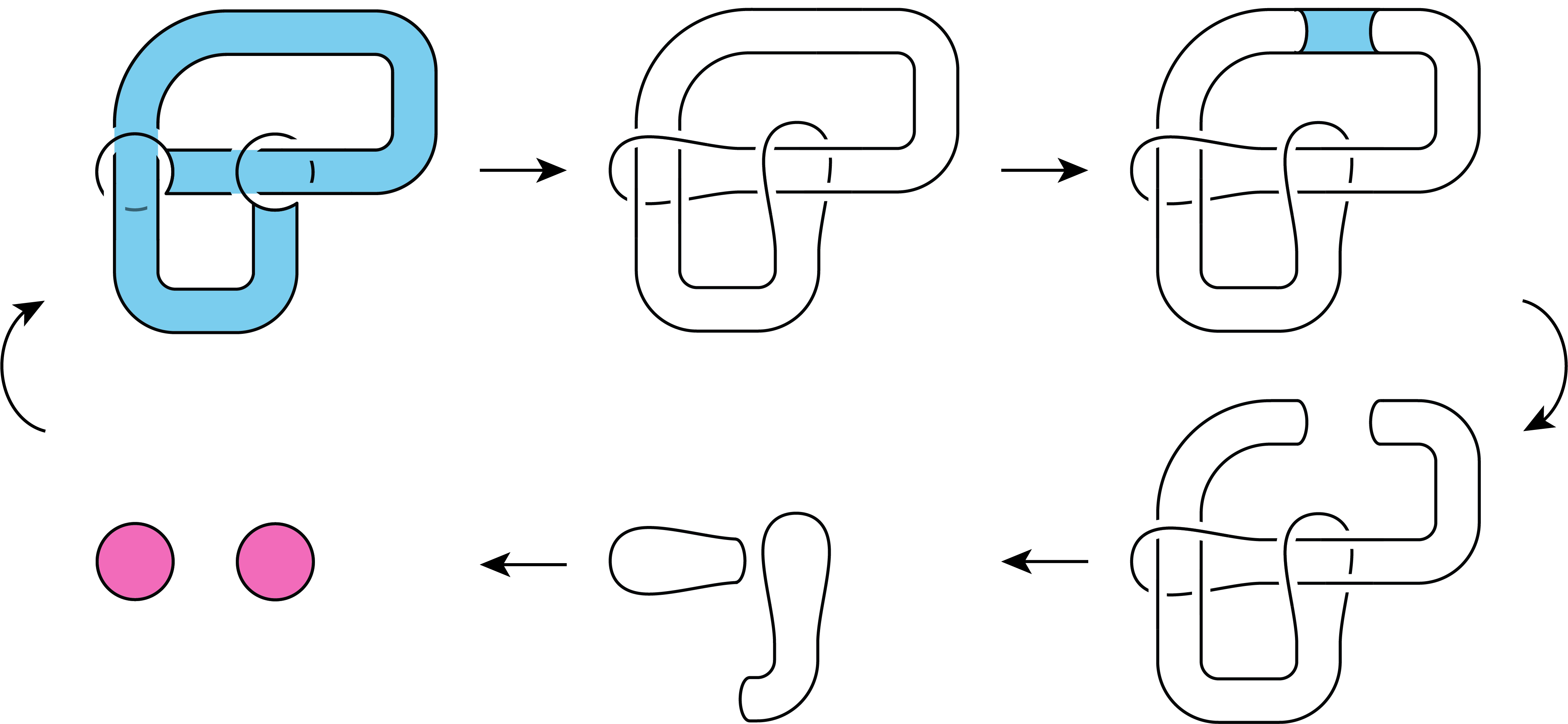}
\caption{\label{fig:ribbon_4D} } 
\end{figure}

The final step in the construction of an $m\times n$ swatch is a band surgery involving $m$ bands, which is illustrated in Figure~\ref{fig:swatch_construction}(c)-(d). 

\end{proof} 

\noindent The converse of \textbf{Theorem}~\ref{theorem:ribbon} 
does not hold. The link shown in Figure~\ref{fig:non_swatch}(a) is a counterexample. 

\begin{figure}[h!]
\centering
\includegraphics[width =174 mm ]{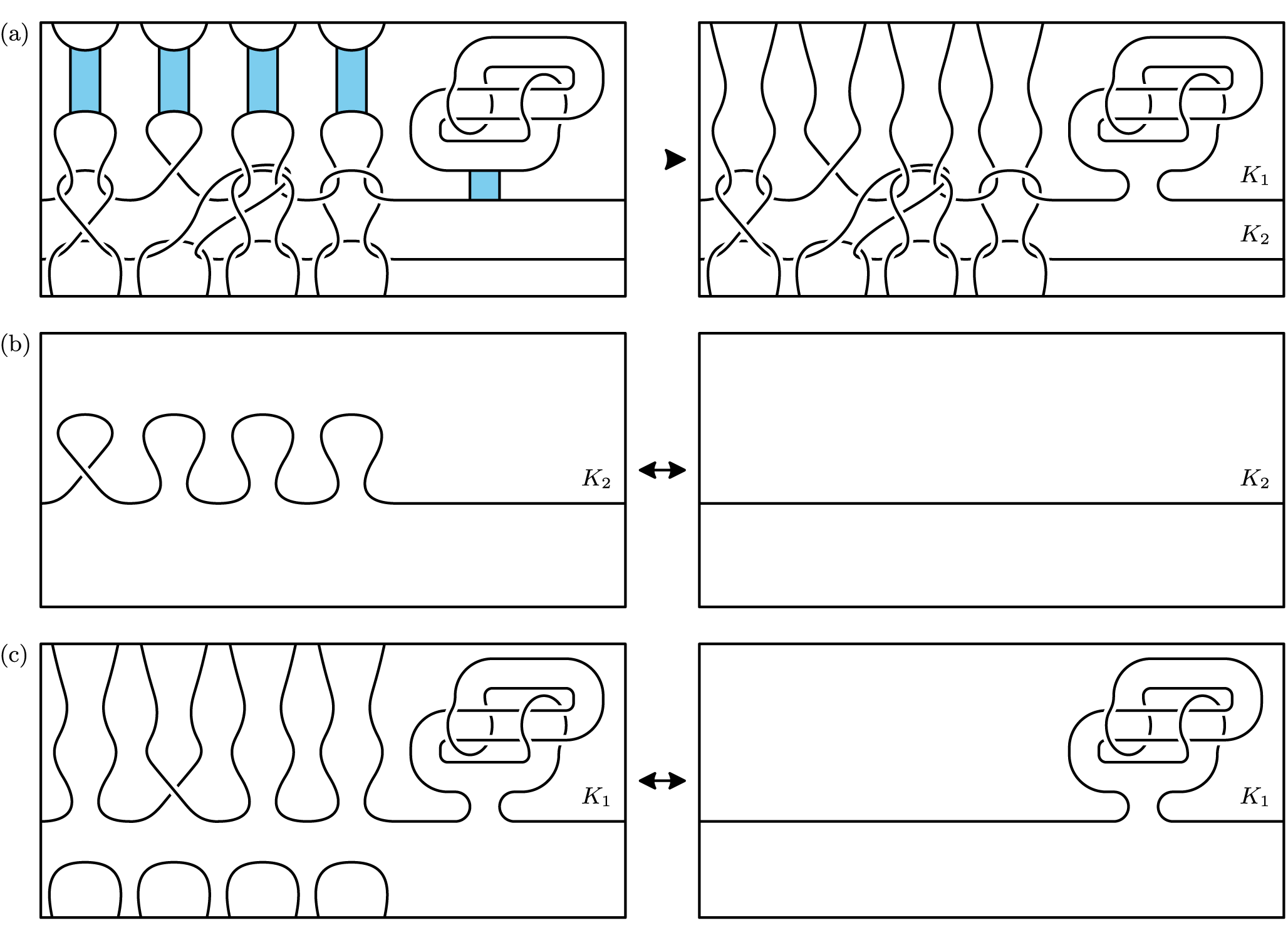}
\vspace{-10pt}
\caption{\label{fig:non_swatch}(a) A non-swatch link $L = K_1\cup K_2\subset T^2\times[0,1]$ with the ribbon property obtained by performing band surgery with respect to the blue rectangles followed by the band surgery with respect to the yellow rectangle. (b) Knot $K_2$ is the $0\times1$ unknit. (c) Knot $K_1$ is not an unknit.} 
\end{figure} 

\begin{lemma}
\label{lemma:non_swatch}
There exist links in $T^2\times I$ that are not swatches but lead to ribbon links in $S^3$ after Dehn filling $T^2\times I$ using the boundary homeomorphism with slopes $1/0$ and $0/1$.
\end{lemma}
\begin{proof}
Consider the link $L = K_1\cup K_2\subset T^2\times I$ resulting from the band surgery with respect to cyan bands as shown in Figure~\ref{fig:non_swatch}(a). We claim that this link cannot be constructed as a swatch. To see this, note that every component of an $m\times n$ swatch, say $L'\subset T^2\times I$, is an unknot that winds around the meridian of the base torus once unless the following is the case: one or more ribbon disks with their boundaries in $f(L')\subset S^3$ consist of pure ribbon singularities. Furthermore, if a component $f(K)\subset f(L')$ bounds a ribbon disk that has a pure ribbon singularity, then the sublink $f(K)\cup f(l)\subset H\cup f(L')$ in $S^3$ is not split; the link $H = f(l)\cup f(m)$ is the Hopf link, where $l$ (longitude), $m$ (meridian) are equivalent to the curves along the core of purple and orange tori in Figure~\ref{fig:T2xI}. In other words, there exists no $S^2\subset S^3$ such that on one side it encloses the component $f(K)\subset f(L')$ and on the other side lies the component $f(l)$ of the Hopf link. 

Going back to the link $L\subset T^2\times I$ shown in FIG~\ref{fig:non_swatch}(a), one of the ribbon disks with a component of the link $f(L)\subset S^3$ as its boundary gives rise to two pure ribbon singularities. Specifically, as shown in Figure~\ref{fig:non_swatch}(b) the component $K_2\subset L$ is equivalent to an unknot in $T^2\times I$, and the component $K_1\subset L$ with a pair of pure ribbon singularities, as shown in Figure~\ref{fig:non_swatch}(c), is not equivalent to an unknot in $T^2\times I$. Therefore, in this case, for the link $L\subset T^2\times I$ to be a swatch, we expect that the sublink $f(K_1)\cup f(l)\subset H\cup f(L)$ is not split in $S^3$. On the contrary, as shown in Figure~\ref{fig:non_swatch}(c), the knot $K_1$ lies within a longitudinal annular strip implying that the sublink $f(K_1)\cup f(l)\subset H\cup f(L)$ is split in $S^3$. Thus, even though the link $f(L)\subset S^3$ is ribbon, the link $L\subset T^2\times I$ is not a swatch.  

\end{proof}

\section{Topological properties of swatches}
Let $L \subset T^2\times I$ be an $n$-component textile link corresponding to a two-periodic weft-knitted textile, and let knots $K_i$ for $i\in\{1,2,\cdots,n\}$ be its components. The $(n+2)$-component link $H\cup f(L)\subset S^3$ is obtained by embedding the link $L\subset T^2\times I$ into $S^3$, where $H = f(l)\cup f(m)$ is the Hopf link. Let $[K]$ denote the homology class of a curve $K\subset T^2\times I$. 
Based on \textbf{Remark}~\ref{remark:textile_link_to_swatch}, the link $L$ is an $m\times n$ swatch for some $m\in\{1,2,\cdots\}$. As a result, the link $L\subset T^2\times I$ and the link $H\cup f(L)\subset S^3$ have the following properties:
\begin{enumerate}
\label{list:properties} 
\item The link $f(L)\subset S^3$ is ribbon.
\item The homology class of each and every component of the link $L$ is given by $[K_i] = 0.[m] + 1.[l] = [l]$, where $m$ and $l$ are the meridian and the longitude of the base torus of $T^2\times I$ respectively.
\item For $n\geq2$, every pair of components of the link $f(L)\subset S^3$ are \emph{algebraically unlinked} implying the same for the link $L\subset T^2\times I$.
\item Every component of the link $L$ is homologous to the longitude. However, the sublink $f(L)\cup f(l)\subset H\cup f(L)\subset S^3$ is \emph{not split} meaning that there does not exist any 2-sphere separating the component $f(l)$ and the link $f(L)$ in $S^3$. Equivalently, the link $f(L)$ and the component $f(l)$ are algebraically unlinked in $H\cup f(L)\subset S^3$, but not split.
\item Consider the ribbon disks corresponding to the components of the link $f(L)\subset S^3$ arising as a result of the Dehn filling with slopes $1/0$ and $0/1$. The boundaries of the disks having only mixed ribbon singularities are equivalent to unknots that are mutually split with the component $f(l)$ of the Hopf link. In contrast, the ribbon disks having even a single pure ribbon singularity have boundaries that are, in general, neither unknotted nor mutually split with the component $f(l)$ of the Hopf link. 
\end{enumerate}

\begin{figure}[h!]
\centering
 \includegraphics[width = 84 mm]{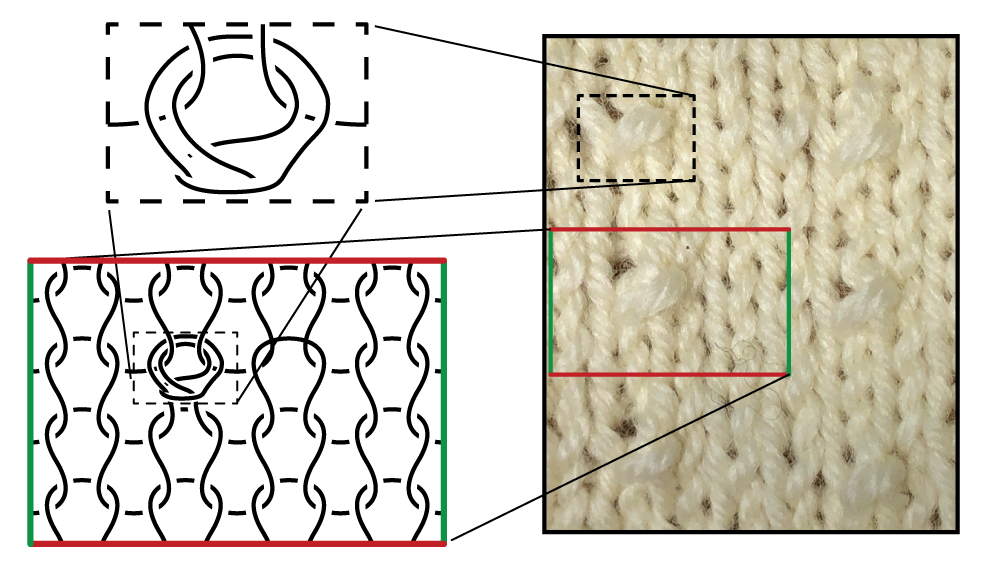}
\caption{\label{fig:cow_hitch}The cow hitch swatch, and a fabric sample with a few scattered cow hitch motifs.} 
\end{figure}

Based on the first two properties above, we discovered a new swatch that can be hand knitted. The open loop structure of the slip loop is replaced by a structure obtained by tying a knot into the bight that resembles the \emph{cow hitch} \cite{AoB}. The resulting swatch and a picture of a hand-knitted textile sample highlighting the cow hitch motif are shown in Figure~\ref{fig:cow_hitch}.

Let the set of all swatches be $\mathcal{W}$, the set of all textile links corresponding to two-periodic weft-knitted fabrics be $\mathcal{F}$, and the set of all links in $T^2\times I$ for which properties one to five in list~\ref{list:properties} hold be $\mathcal{L}$. The set of all swatches $\mathcal{W}$ is a proper subset of the set $\mathcal{L}$ because, we have showed in \textbf{Lemma}~\ref{lemma:non_swatch} that there exist non-swatch links in $T^2\times I$ leading to ribbon links in $S^3$ after Dehn filling.
However, it is not known if the set $\mathcal{F}$ is a proper subset of $\mathcal{W}$ as there is no mathematically well defined notion of weft-knitability that encompasses all the mechanical moves used in hand-knitting. In conclusion, for the sets $\mathcal{F}$, $\mathcal{W}$ and $\mathcal{L}$, the following relation holds: $\mathcal{F}\subseteq\mathcal{W}\subsetneq \mathcal{L}$.

\subsection{The linking number and algebraic linking} 
\label{sec:linking_number}
Before defining the linking number it is worth mentioning that there are many notions of linking in the literature for links in $S^3$ and equivalently for links in $\mathbb{R}^3$ \cite{Milnor_57}. Among those, the concept of \emph{algebraic linking}, which is associated with the \emph{linking number} defined below, is simple to understand and widely used.

\begin{definition}(The linking number)
\label{def:lkn}
Let $L = K_1\cup...\cup K_n\subset S^3$ be an n-component link. The linking number between the components $K_i$ and $K_j$ is given by 
\begin{equation}\label{eq:lkn}
\emph{lk}(K_i,K_j) = \Bigg|\frac{1}{2}\sum\limits_{k=1}^N\epsilon_k\Bigg|,
\end{equation}
where $\epsilon_k=\pm1$ is assigned to $k^{\textrm{th}}$ crossing based on the local orientation of under and over strands at that crossing as per the convention in described in Figure~\ref{fig:linking_number}(a).
\end{definition} 

Consider a map defined on the set of all links in $S^3$. If the image of the map is preserved under ambient isotopies, then the map defines a topological invariant of links in $S^3$ or a \emph{link invariant}. The linking number between two components of a link in $S^3$, is preserved under ambient isotopies. Thus, for an $n$-component link $K_1\cup...\cup K_n\subset S^3$, the set of pairwise linking numbers is a link invariant.  
Let us consider some simple examples: 1) the linking number of the Hopf link which is shown in FIG~\ref{fig:linking_number}(b) is equal to one, and therefore, it is not equivalent to the two-component unlink in $S^3$. 2) the linking number between any pair of components in the Borromean rings, shown in Figure~\ref{fig:linking_number}(c), is zero, which is same as that of the three component unlink in $S^3$. However, the three component unlink is not equivalent to Borromean rings. This illustrates the fact that, a pair of links with identical sets of pairwise linking numbers are not necessarily equivalent. Therefore, even though the linking number between two components of a link, as given in equation~\ref{eq:lkn}, is easy to compute, it yields a rather weak link invariant.

\begin{definition}[Split links]
\label{defn:split_links}
Given an $n$-component link $L \subset S^3$ with $n\geq2$. If there exists an embedding of $S^2$ that does not intersect the link but separates it into two non-empty subsets of components or sublinks, then the link $L\subset S^3$ is split.
\end{definition}
We say a two-component sublink in a link algebraically linked if their pairwise linking number is non-zero. Note that the property of algebraically unlinking is weaker than being split. 
For example, all three pairs of components of Borromean rings shown in Figure~\ref{fig:linking_number}(c) are algebraically unlinked but the link itself is not split. 
   
Let $L\subset T^2\times I$ be the $n$-component textile link corresponding to a two-periodic weft-knitted textile. 
The Dehn filling $f$ combined with the second property in the list~\ref{list:properties} about the homotopy of the components of a swatch imply that the linking numbers lk$(f(l),f(K_i))$ and lk$(f(m),f(K_i))$ are zero and one respectively, for all components $K_i\subset L$.
The component $f(l)$ and the components of the sublink $f(L)$ are algebraically unlinked, but they are not split, which is stated as the fourth property in the list~\ref{list:properties}. Similarly, consistent with the third property in the list~\ref{list:properties} the linking number lk$(K_i,K_j)$ is zero for all pairs of components $K_i,K_j\subset L$. Thus, the components $K_i\subset L$ and $K_j\subset L$ are algebraically unlinked, but may not be split for all $i,j\in\{1,2,\cdots,n\}$.

\begin{figure*}
\centering
\subfloat[A convention for assigning signs to crossings. $\epsilon=+1$ for the crossing motifs that resemble one on the left, and $\epsilon=-1$ for the rest.]
 {\includegraphics[width=40 mm]{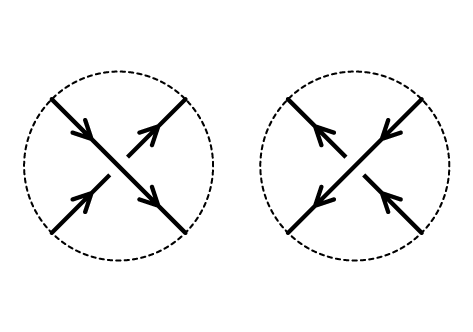} }
 \hspace{10pt}
\subfloat[The Hopf link: $\textrm{lk}(u_1,u_2)  = 1$.]
 {\includegraphics[width=40 mm]{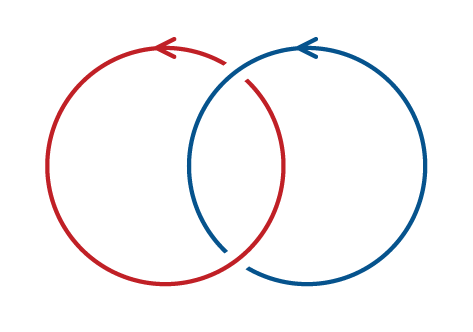} }
 \hspace{10pt}
\subfloat[The Borromean rings: $\textrm{lk}(v_i,v_j)  = 0$ for all $i,j\in\{1,2,3\}$ such that $i\neq j$.]
 {\includegraphics[width=40 mm]{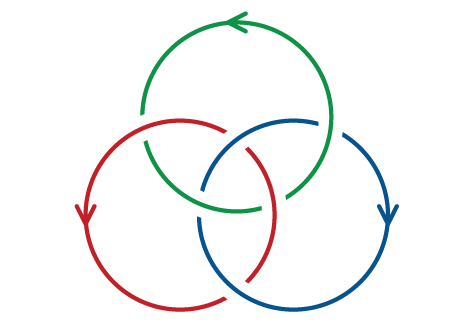} }
\caption{\label{fig:linking_number}Linking number and a few examples.}
\end{figure*}

\subsection{Are swatches hyperbolic?}
A link in $S^3$ is said to be \emph{hyperbolic} if its link complement in $S^3$ admits a metric with constant negative sectional curvature \cite{Purcell_2020}. Extending this notion to swatches, we say an $m\times n$ swatch $L\subset T^2\times I$ is hyperbolic if the $(n+2)$-component link $H\cup f(L)\subset S^3$ is hyperbolic. Based on our numerical experiments in \texttt{SnapPy} \cite{SnapPy} we conjecture that the swatches consisting of only the knit and the purl motifs give rise to hyperbolic links in $S^3$. \texttt{SnapPy} software is used to check whether a given link is hyperbolic or not based on the attribute $.solution\_type()$ of the link complements as 3-manifolds. 

Hyperbolicity is a very useful notion as many link invariants follow from the hyperbolic geometry of the link complement such as a representation of the fundamental group of the link complement also known as the \emph{link group}, \emph{hyperbolic volume}, \emph{invariant trace field}, \emph{cusp shapes} etc. We will discuss each of these separately in section~\ref{sec:invariants}. Our data, based on computations in \texttt{SnapPy}, indicates that there are many more hyperbolic swatches apart from just those consisting of the knit and the purl motifs. However, in contrast to the characteristic property of being ribbon stated in \textbf{Theorem}~\ref{theorem:ribbon}, hyperbolicity does not hold for all the swatches. 

\begin{remark}\label{remark:non_hyp}
The swatches giving rise to layered textiles are not hyperbolic because, the planes of separation in between the layers yield inessential splitting tori \cite{Purcell_2020} that separate the swatch into sublinks corresponding to different layers in the textile \cite{MORTON_2009}. 
\end{remark}

\begin{figure*}
\centering
\subfloat[]{\includegraphics[width = 65 mm]{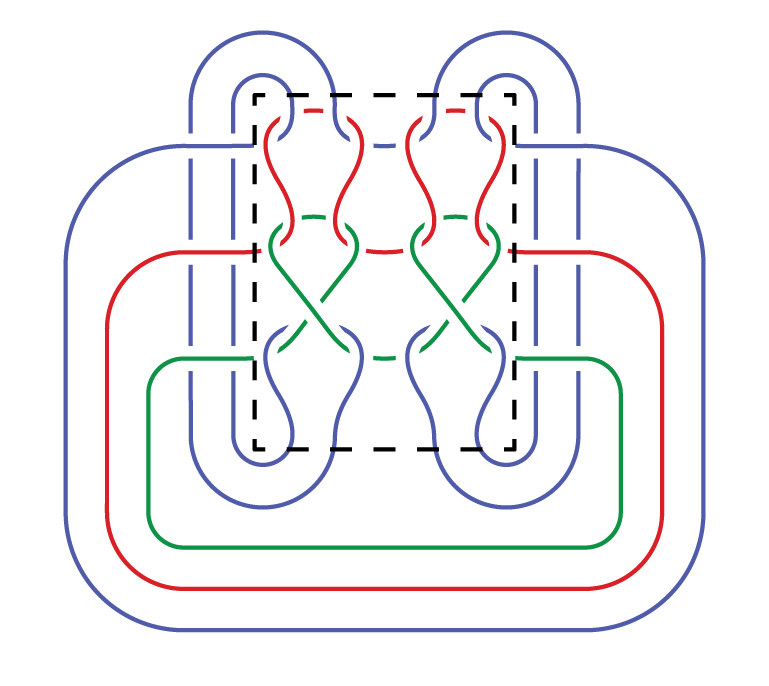}}
\hspace{10 mm}
\subfloat[]{\includegraphics[width = 65 mm]{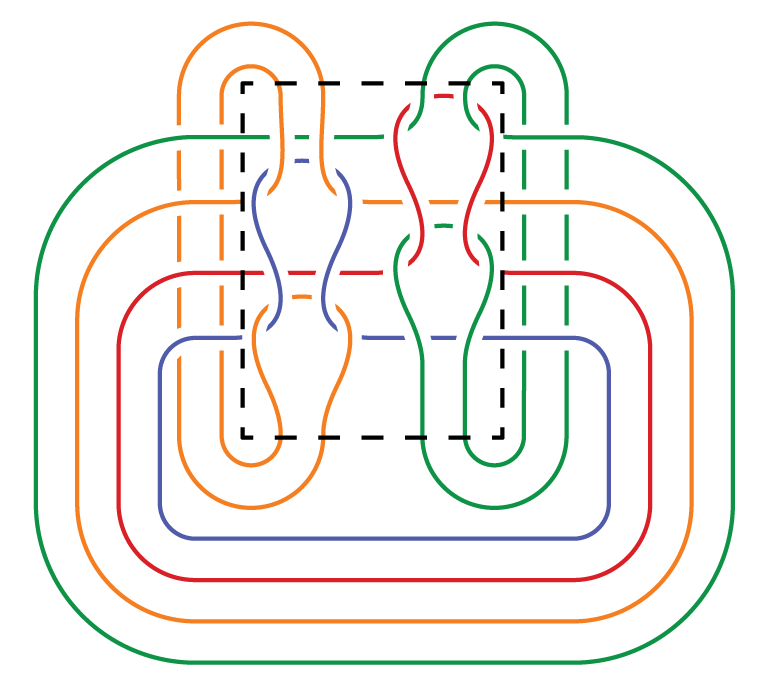}}
\caption{\label{fig:bru}(a) An example of a Brunnian link $f(L)\subset S^3$ corresponding to the Brunnian swatch which is shown within the dotted rectangular block. (b) Link $f(L)\subset S^3$ corresponding to the swatch which is shown within the dotted rectangular block is not Brunnian.} 
\end{figure*} 

\subsection{Brunnian swatches}
Many swatches have the following property: deletion of a component of the link leads to a trivial swatch. For such links, each and every proper sublink is a trivial swatch. Therefore, given an $m\times n$ swatch $L\subset T^2\times I$ with the above property, every proper sublink of the corresponding $n$-component link $f(L)\subset S^3$ is an unlink. These links are called \emph{Brunnian links} \cite{Brunn_97}.  

\begin{definition}[Brunnian links and Brunnian swatches]
Let $L = \bigcup\limits_{i=1}^n K_i\subset S^3$ be an n-component link, where $K_i$ are the components of the link $L$. If the $(n-1)$-component sublink $L_j = L\setminus K_j  = \bigcup\limits_{i\neq j}K_i \subset L\subset S^3$ is equivalent the $(n-1)$-component unlink for all $j\in\{1,2,...,n\}$, then the link $L$ is said to be Brunnian. Similarly, let $L' = \bigcup\limits_{i=1}^n K'_i\subset T^2\times I$ be an $m\times n$ swatch, where $K'_i$ are its components. If the $(n-1)$-component sublink $L'_j = L'\setminus K'_j  = \bigcup\limits_{i\neq j}K'_i \subset L'\subset T^2\times I$ is equivalent to the $(n-1)$-component trivial swatch for all $j\in\{1,2,...,n\}$, then we say that the swatch $L'$ is Brunnian.
\end{definition}

The \emph{Borromean rings}, shown in Figure~\ref{fig:linking_number}(c), is a Brunnian link. A Brunnian swatch and a non-brunnian swatch are highlighted within dashed rectangular blocks in Figure~\ref{fig:bru}(a) and Figure~\ref{fig:bru}(b) respectively. The latter can be obtained as a motif in a weft-knitted textile constructed via hand-knitting by slipping every other slip loop on the left needle without pulling a new loop through it. This amounts to alternatively skipping steps two and three shown in Figure~\ref{fig:knitting}(a)-(d). 

The special nature of Brunnian swatches is reflected in their values of the topological invariants. For instance, every proper sublink of a Brunnian link is an unlink implying all the proper sublinks are split, and thus, the ordered list of linking numbers is an $n$-tuple of zeros. Furthermore, all the components of a Brunnian link are unknots. In section~\ref{sec:mva} on the multivariable Alexander polynomial -- a multivariable Laurent polynomial link invariant -- we will see that the multivariable Alexander polynomial of the links in $S^3$ corresponding to Brunnian swatches simplifies to a reduced form given by equation~\ref{eqn:brun}.   
   
\section{An algebra for swatches}
\label{section:annulus_sums}

\subsection{The meridional and the longitudinal annulus sums}
In weft-knitting, a common technique of making complex patterns of yarn embeddings is combining motifs of simple ones. For example, one obtains the motifs of the one-by-one rib and the garter stitch patterns by `adding'  or combining the motifs of the stockinette stitch patten i.e., the knit and the purl. In terms of swatches, this idea of combining motifs translates to cutting open link complements of the swatches, along either their meridional or longitudinal annular cross-sections, followed by gluing them. We restrict only to those instances where the result of the operation is the link complement of another swatch, and as a consequence, the set of swatches is closed under the binary operation of interest by construction. Combining link complements of the knit swatch and the purl swatch in this fashion leads to link complements of one-by-one rib swatch and the garter swatch, which are shown in Figure~\ref{fig:gluing}(a) and Figure~\ref{fig:gluing}(b) respectively. 
   
Based on the homotopy type, the 3-manifold $T^2\times I$ has two kinds of cross-sections -- meridional and longitudinal -- which are shown in Figure~\ref{fig:annulus}(a) and Figure~\ref{fig:annulus}(e) respectively. These correspond to the two ways in which the link complements of the knit swatch and the purl swatch are glued leading to the one-by-one rib swatch (meridional) and the garter swatch (longitudinal). Given a pair of swatches, we exclude a discussion of the general case and consider only slicing and gluing along the green annulus (meridional cross-section) and the red annulus (longitudinal cross-section). The green and red annuli correspond to the green and red edges in Figure~\ref{fig:swatch_construction}. For a single swatch, the slicing leads to a connected 3-manifold with two punctured annuli boundaries as shown in Figure~\ref{fig:annulus}(a) and Figure~\ref{fig:annulus}(e). These boundaries are then glued by identifying the punctures to get the link complement of another swatch.

\begin{definition}[Annulus sums]\label{def:gluing}
Let $L_1\subset T^2\times I$ be an $m\times n$ swatch and $L_2\subset T^2\times I$ be an $m'\times n$ swatch. Then the link complements of $L_1$ and $L_2$ are denoted by $X_{L_1}$ and $X_{L_2}$ respectively. Recall from the definition of swatch a diagram of a swatch incl have chosen a specific quotient map for the torus as a quotient of the plane. Therefore for each swatch we have two distinguished annuli sitting in $T$

Cut open the 3-manifold $X_{L_1}$ along its green punctured annulus (which is specified through its construction as a swatch). Similarly, cut open the 3-manifold $X_{L_2}$ along its green punctured annulus. Glue the resulting 3-manifolds along their punctured annuli boundaries such that the punctures are identified seamlessly forming a continuous tunnel. Thus, we obtain the 3-manifold $X_{L_1*_mL_2}$, which is the link complement of an $(m+m')\times n$ swatch $L_1*_mL_2\subset T^2\times I$. We call this operation that acts on the link complements of two swatches, a meridional annulus sum.

Let $L'_2\subset T^2\times I$ be an $m\times n'$ swatch. Then the link complement of $L'_2$ is given by $X_{L'_2}$. Slice the 3-manifolds $X_{L_1}$ along its red punctured annulus (which is specified through its construction as a swatch). Similarly, slice the 3-manifolds $X_{L'_2}$ along its red punctured annulus. Glue the resulting 3-manifolds along their punctured annuli boundaries such that the punctures are identified seamlessly forming a continuous tunnel. Thus, we obtain the 3-manifold $X_{L_1*_lL'_2}$, which is the link complement of an $m\times(n+n')$ swatch $L_1*_lL'_2\subset T^2\times I$. We call this operation that acts on the link complements of two swatches, a longitudinal annulus sum.
\end{definition}

The annulus sums form associative binary operations, and thus can be extended to combine more than two swatches. The cyclic permutation symmetry in the order of gluing the sliced link complements gives rise to redundancies in the number of swatches that are created through the annulus sums.

\begin{remark}\label{remark:cyclic_permutation}
Let $L_i\subset T^2\times I$ be an $m_i\times n$ swatch for $i\in\{1,2,\cdots,N_1\}$. The swatches $L_{\sigma_1(1)}*_mL_{\sigma_1(2)}*_m\cdots*_m L_{\sigma_1(N_1)} \subset T^2\times I$ and $L_{\sigma_2(1)}*_mL_{\sigma_2(2)}*_m\cdots*_mL_{\sigma_2(N_1)} \subset T^2\times I$ are identical, where $(\sigma_{1,2}(1),\sigma_{1,2}(2), \cdots, \sigma_{1,2}(N_1))$ denote a pair of cyclic permutations of the $N_1$-tuple $(1,2,\cdots,N_1)$. Similarly, if $L'_i\subset T^2\times I$ is an $m\times n_i$ swatch for $i\in\{1,2,\cdots,N_2\}$. Then the swatches $L'_{\sigma_1(1)}*_lL'_{\sigma_1(2)}*_l\cdots*_l L'_{\sigma_1(N_2)} \subset T^2\times I$ and $L'_{\sigma_2(1)}*_lL'_{\sigma_2(2)}*_l\cdots*_lL'_{\sigma_2(N_2)} \subset T^2\times I$ are identical, where $(\sigma_{1,2}(1),\sigma_{1,2}(2), \cdots, \sigma_{1,2}(N_2))$ denote a pair of cyclic permutations of the $N_2$-tuple $(1,2,\cdots,N_2)$.
\end{remark} 

\begin{figure}[h!]
\centering
\subfloat[Gluing a knit swatch and a purl swatch along their meridional annular cross-section yields a 1$\times$1 rib swatch.]
{\includegraphics[width=40 mm]{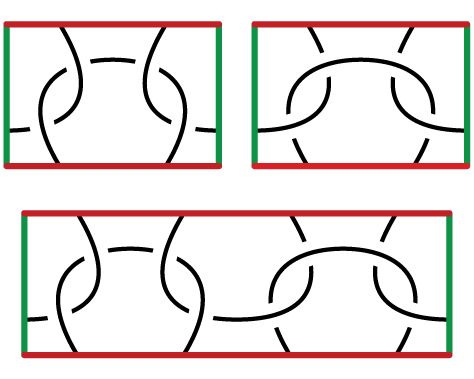}}
\hspace{4mm}
\subfloat[Gluing a knit swatch and a purl swatch along their longitudinal annular cross-section yields a garter swatch.]
 {\includegraphics[width=40 mm]{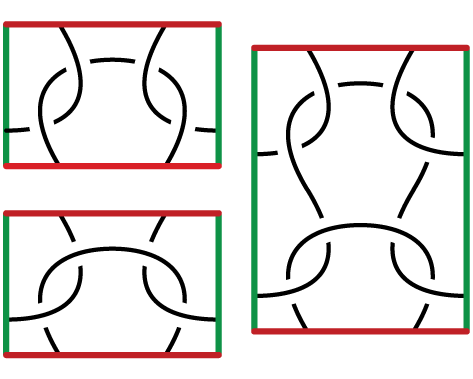}}
 \caption{\label{fig:gluing}Examples of the meridional gluing and the longitudinal gluing of swatches.}
\end{figure} 

\begin{figure}[h!]
\centering
\subfloat[Combining swatches horizontally by cutting 2-tori along their meridional cross-section.]
{\includegraphics[width=35mm]{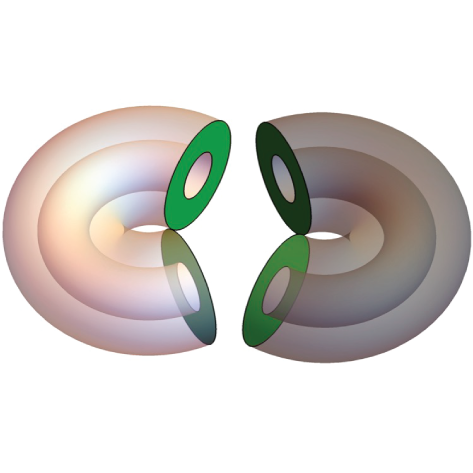}}
\hspace{2mm}
\subfloat[The cut 3-manifolds are then glued together along the boundary annuli.]
{\includegraphics[width=35mm]{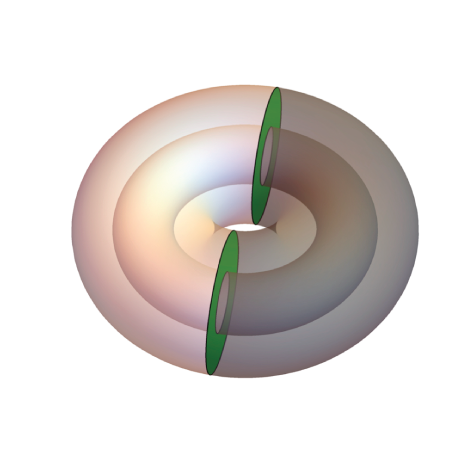}}
 \hspace{2mm}
\subfloat[In the $S^3$ picture, this can be algebraically realized using band surgery.]
{\includegraphics[width=35mm]{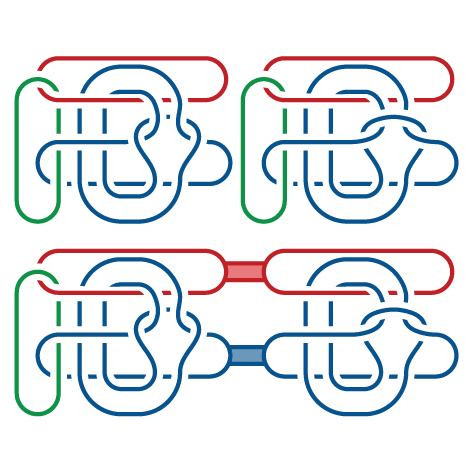}}
\hspace{2mm}
 \subfloat[$1\times1$ ribbing from a meridional annulus sum.]
{\includegraphics[width=35mm]{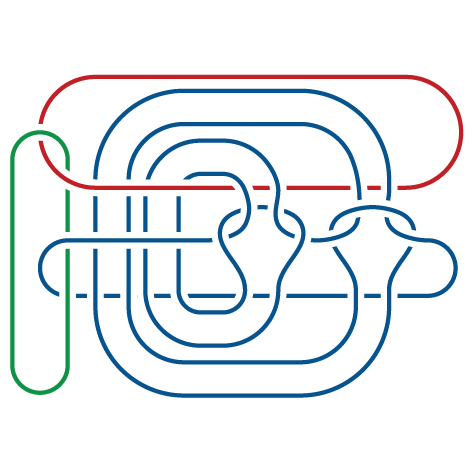}}
\\
\subfloat[Combining swatches vertically by cutting 2-tori along their longitudinal cross-section.]
{\includegraphics[width=35mm]{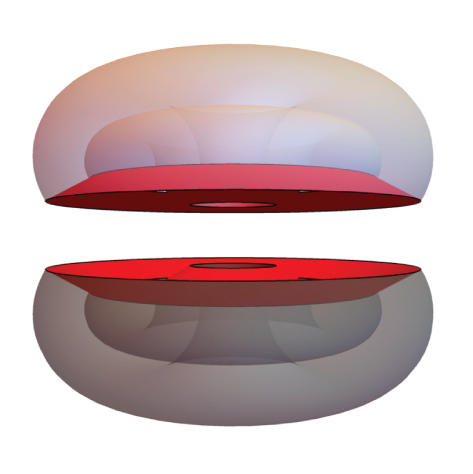}}
 \hspace{4mm}
\subfloat[The cut 3-manifolds are then glued together along the boundary annuli.]
 {\includegraphics[width=35mm]{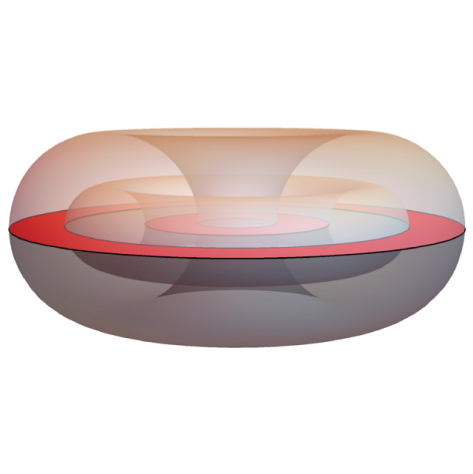}}
\hspace{2mm}
\subfloat[In the $S^3$ picture, this can be algebraically realized using band surgery.]
{\includegraphics[width=35mm]{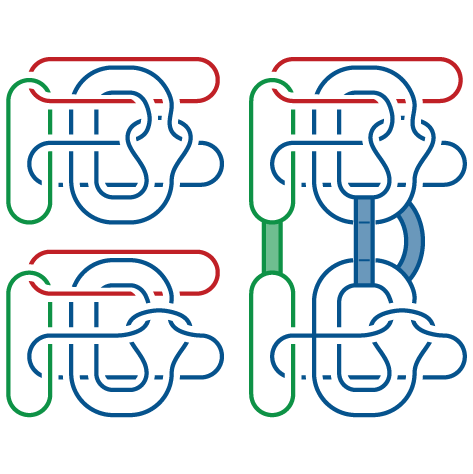}}
 \hspace{2mm}
 \subfloat[Garter stitch from a longitudinal annulus sum.]
{\includegraphics[width=35mm]{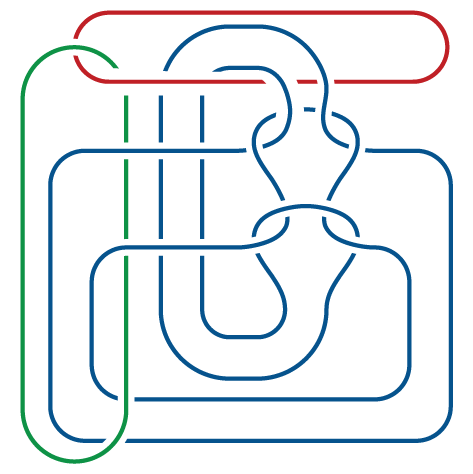}}

\caption{Each type of \emph{annulus sum} performed on the link complements yield a procedure for combining knit and purl swatches leading to a compound swatch corresponding two-periodic weft-knitted fabrics.}
 \label{fig:annulus}
 \vspace{-5pt}
\end{figure}

\subsection{Partitioning the set of all swatches $\mathcal{W}$}\label{subsec:partition}

The set of all swatches can be partitioned into subsets indexed by the number of components such that each partition by itself is a closed set with the meridional annulus sum as the associative binary operation. A meridional annulus sum acting on an $m\times n$ swatch and the $n$-component trivial swatch leaves the $m\times n$ swatch unaltered, and therefore, each closed subset in the partition has an identity element with respect to meridional annulus sums. The closure property and the existence of identity element implies that the set of all $m\times n$ swatches (where $n\in\mathbb{N}$ is fixed and $m\in\mathbb{N}$ is arbitrary) with the meridional annulus sum is a \emph{monoid}.

\begin{definition}[A monoid and a semigroup]
A Semigroup is a set closed under an associative binary operation. A Monoid is a semigroup with an identity element.
\end{definition}
Let $\mathcal{W}$ and $\mathcal{P}_n$ be the set of all swatches and the set of all $m\times n$ swatches for some fixed $n\in\mathbb{N}$ and arbitrary $m\in\mathbb{N}$ respectively. Then as a result of the partitioning we have 
\begin{equation}
\label{eqn:partition}
\mathcal{W} = \bigcup\limits_{n=1}^{\infty} \mathcal{P}_n.
\end{equation}
To get some insight into the composition of the monoid $(\mathcal{P}_n,*_m)$ for some $n\in\mathbb{N}$, we start by enumerating the elements in a submonoid of $(\mathcal{P}_1,*_m)$. In the discussion below, we do not explicitly give the slicing and gluing data while referring to annulus sums of swatches as the details are not necessary.

Suppose $e$, $k$ and $p$ denote the one-component trivial swatch, the knit swatch and the purl swatch respectively. Then the set of strings of letters in $k$ and $p$ generated using concatenation as the associative binary operation, which are referred to as words, is a monoid. The generating set is given by $\{k,p\}$ (the alphabet) and the identity element or the empty word is given by $e$. Then 
\begin{align*}
&ke = ek = k\\ 
&pe = ep = p, 
\end{align*}
and thus, we obtain the free monoid $\mathcal{M}_1 = \{e, k, p, kk, kp, pk, pp, kkk, kkp, ...\}$. The generators or the generating elements of a set are the elements in the set that do not admit any non-trivial decomposition with respect to the underlying binary operation. 
 
\begin{figure}[h!]
\centering
\subfloat[The cow hitch swatch.]
 {\includegraphics[width=40mm]{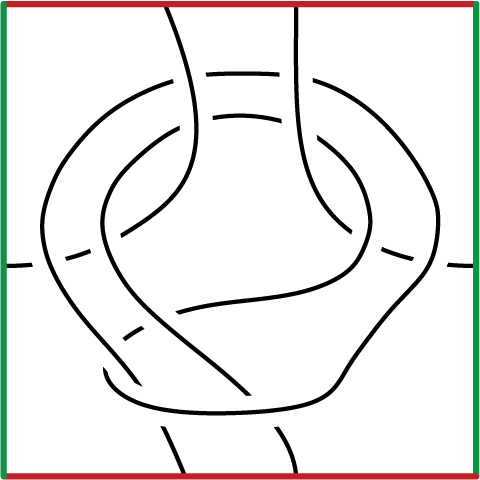} }
 \hspace{3mm}
\subfloat[A purl-like swatch with two twists at the base of the bight.]
 {\includegraphics[width=40mm]{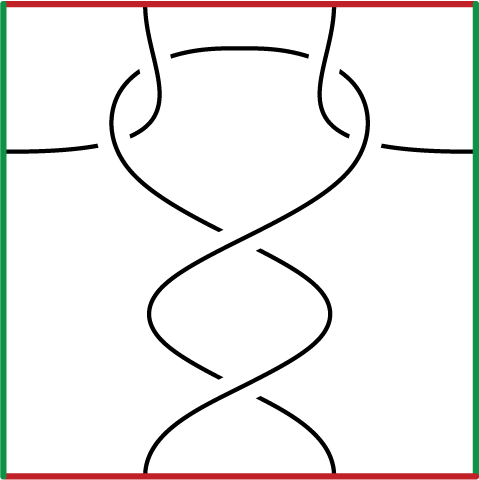} }
\\
\subfloat[A knit-like swatch with a bight that requires all three Reidemeister moves.]
 {\includegraphics[width=40mm]{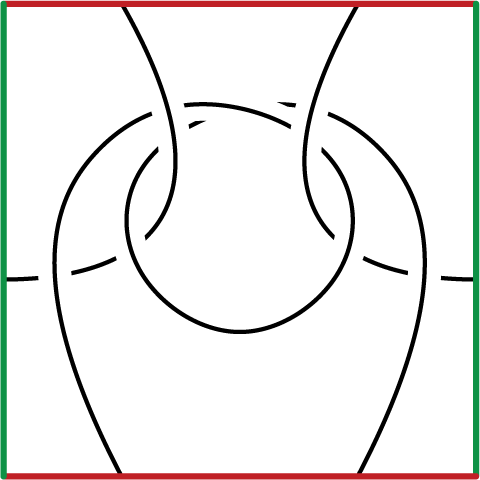} }
\caption{Three $1\times1$ swatches in the generating set of $\mathcal{P}_1$ which are distinct from knit $\&$ purl swatches.}
\label{fig:alphabets_p1}
\end{figure}

The $m\times1$ swatches corresponding to the elements of the free monoid $\mathcal{M}_1$ are not all distinct. The redundancies in swatches are associated with the words related in the following manner: 1) words with the same repeating unit of letters, and 2) words related by cyclic permutations in the order of concatenation of letters. Let $\tilde{\mathcal{M}_1}\subset\mathcal{M}_1$ be the set obtained by modding out the free monoid $\mathcal{M}_1$ with respect to an equivalence relation defined based on the points 1) and 2) above. Then the set $\tilde{\mathcal{M}_1}$ is a submonoid of the set of all $m\times1$ swatches $\mathcal{P}_1$. 

The submonoid $\tilde{\mathcal{M}_1}\subseteq\mathcal{P}_1$ does not account for all the elements in $\mathcal{P}_1$. To see this, note that both the elements in the generating set $\{k,p\}$ denote $1\times1$ swatches, and there exist $m\times1$ swatches, like the one shown in Figure~\ref{fig:cross_stitch}(b) that cannot be generated by the meridional annulus sum because it is not reducible. Moreover, as shown in Figure~\ref{fig:alphabets_p1}(a)-(c) there exist $1\times1$ swatches in the generating set that are distinct from the knit swatch and the purl swatch.  

There are countably infinitely many number of ways of tying knots into bights, and as a result, the number of generating swatches of the set of all $m\times1$ swatches, $\mathcal{P}_1$ is countably infinite. However, unlike hand knitting, the number of elementary mechanical operations a weft-knitting machine is capable of performing is limited. Therefore, the generating set is finite and weft-knitted textiles consisting of $1\times1$ swatches shown in Figure~\ref{fig:alphabets_p1}(a)-(c) cannot be made using a typical weft-knitting machine.

Based on similar arguments as above, one can conclude that the set of generating swatches of the set of all $m\times2$ swatches $\mathcal{P}_2$ consists of countably infinitely many distinct $1\times2$ swatches. And in addition, there exist $m\times2$ swatches with $m\geq2$ that cannot be constructed by the action of meridional annulus sums on $1\times2$ swatches, and therefore, are generating swatches themselves. A subset of $1\times2$ swatches in the generating set of the set of all $m\times2$ swatches $\mathcal{P}_2$ is obtained by the longitudinal annulus sum acting on $1\times1$ swatches. In particular, the set $\{k*_lk, k*_lp, p*_lp\}$ consists of generators composed of only the knit and the purl swatches. Since the longitudinal annulus sum acting on a pair of $m\times1$ swatches yields an $m\times2$ swatch, it is clear that the set $\mathcal{P}_1*_l\mathcal{P}_1$ is contained in the set of all $m\times 2$ swatches $\mathcal{P}_2$. Further, the $1\times2$ swatch shown in Figure~\ref{fig:cross_stitch}(b) implies that the set $\mathcal{P}_2$ is strictly bigger than the set $\mathcal{P}_1*_l\mathcal{P}_1$. 
In general, if $\sum\limits_{i\in\mathcal{I}}n_i = n$ such that $m_i,n\in\mathbb{N}$ for all $i\in\mathcal{I}:=\{1,2,...,k\}$, then 
\begin{equation}
\mathcal{P}_{n_1}*_l\mathcal{P}_{n_2}*_l...*_l\mathcal{P}_{n_k}\subset\mathcal{P}_n. 
\end{equation}
Indeed this must be the case because as we add more number of components, the entanglement complexity in terms of knottiness and linking increases.

\begin{figure}[h!]
\centering
\includegraphics[width = 160mm]{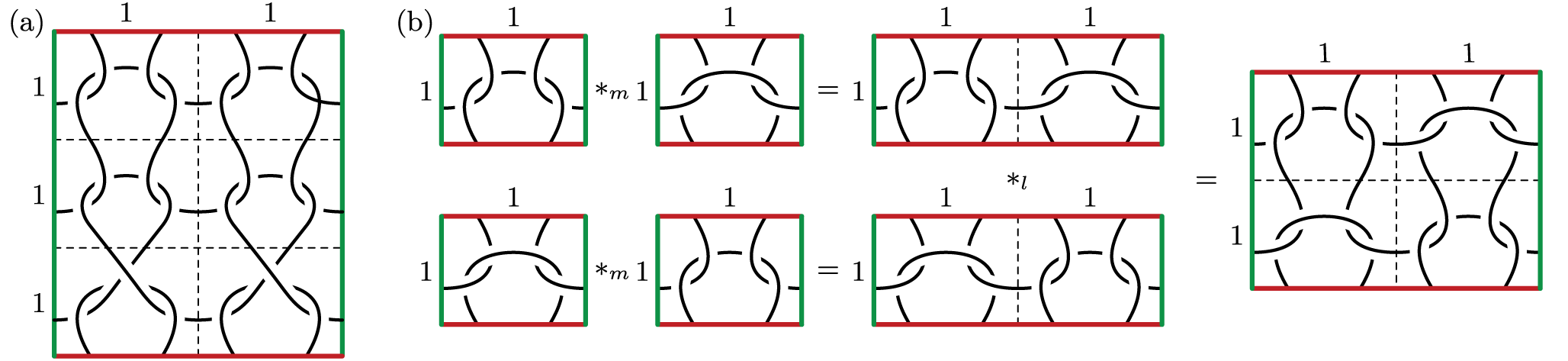}
\caption{\label{fig:grid}(a) A $2\times 3$ swatch composed of $(1,1)$ blocks consisting of $1\times1$ swatches. 
(b) The partition of a $2\times2$ swatch (\emph{seed stitch}) consistent with its decomposition using gluing operations. The integers along the row (top most) and the column (extreme left) of an $m\times n$ swatch indicate the corresponding partition of $m$ and $n$ respectively.} 
\end{figure} 

\begin{figure}[h!]
\centering
 \subfloat[$1\times1$ \emph{cable}: a $2\times1$ irreducible swatch.]{\includegraphics[width = 49mm]{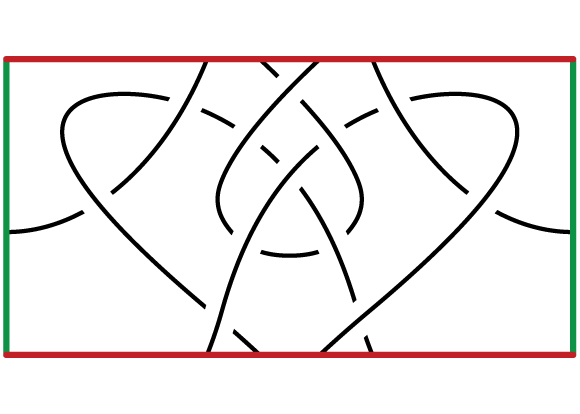}}
\hspace{2.5mm}
\subfloat[A $1\times2$ swatch with a knit-forward-backward motif and a knit-two-together motif.]
{\includegraphics[width = 25mm]{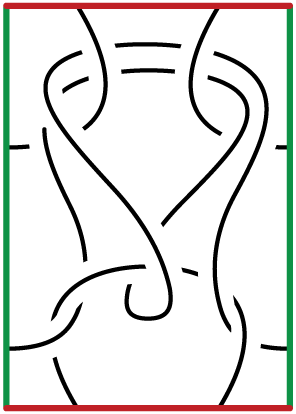}}
\\
\subfloat[A $3\times3$ swatch with a slip-slip-knit motif and a knit-forward-backward-forward motif.]
{\includegraphics[width = 73mm]{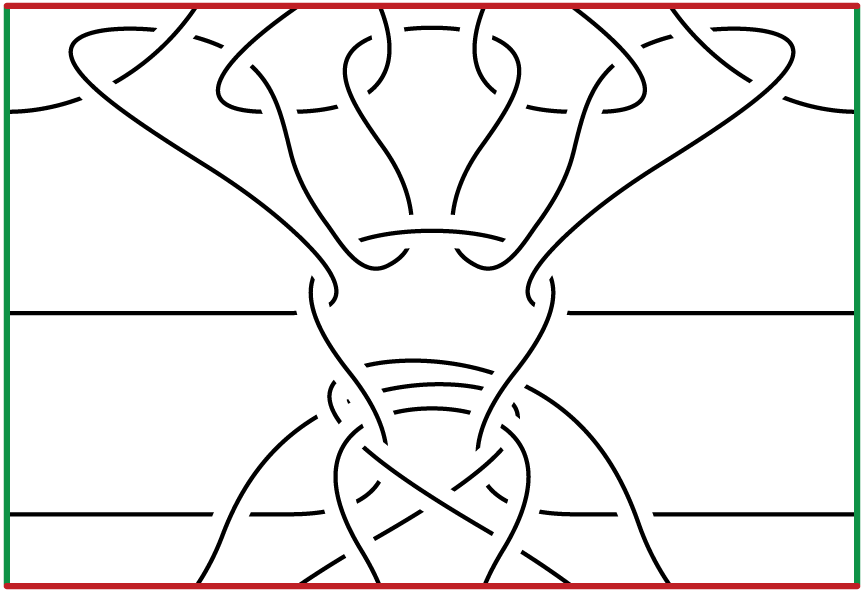}}
\caption{Irreducible swatches} 
 \label{fig:cross_stitch}
 \end{figure}

\subsubsection{Irreducible swatches}
A trivial swatch can be constructed by applying no RM moves but doing only band surgeries (with respect to a set of bands in knitting position). Therefore, the protocol or the procedure of constructing a swatch is unique up to making the trivial swatch with the same number of components and redundant ambient isotopies. As a result, it is possible that a non-empty subset of contractible loops in an unknit do not effect the topology of a given target swatch, in which case that swatch may be realized as an $m\times n$ swatch for infinitely many choices of $m\in\{1,2,\cdots\}$. Here we address this redundancy due to trivial swatches, as we want to be able to get an estimate of number of distinct of ways of constructing an $m\times n$ swatch.
\begin{definition}[Irreducible and compound swatches]\label{defn:irreducible}
An $m\times n$ swatch is irreducible if the following hold:
\begin{itemize}
\item In the 3-step procedure of constructing a swatch, for an unknit in knitting position (as in FIG~\ref{fig:swatch_construction}(c)), we require that none of the $m$ contractible loops can be separated or split from the rest of the unknit while keeping the other components fixed.
\item The swatch cannot be constructed from a set of non-trivial swatches using the longitudinal and the meridional annulus sums.
\end{itemize}
A swatch that is not irreducible is called a compound swatch.
\end{definition} 
\noindent Clearly, all the non-trivial $1\times1$ swatches are irreducible. Below, we briefly describe some knitting operations leading to a few higher order $m\times n$ irreducible swatches (for $m,n\geq2$) that are frequently encountered in the knitting literature \cite{Durant_2015,Shida_2017, Pavilion_2015}:
\begin{enumerate}
\item \emph{Cables or cross-overs:} The simplest example of cable swatch with a single cross-over is shown in Figure~\ref{fig:cross_stitch}(a). In general a cable swatch consists of cross overs involving shift of $l$ slip loops (bights) to the left and $r$ slip loops to the right, where each slip loop is shifted by one block in a given a single row. Here, by slip loops in a row we mean the slip loops belonging to the same yarn, which are aligned along the weft axis of a two-periodic weft-knitted textile.  
\item \emph{Stitch increases and decreases:} The swatch in Figure~\ref{fig:cross_stitch}(b) is obtained by increasing and decreasing the number of slip loops by one in successive rows. In this case, an increase in the number of slip loops is achieved by knitting through front followed knitting through back before pushing the slip loop off of the left needle. A decrease in the number of slip loops results from knitting two or more loops together before pushing them off of the left needle. As illustrated in Figure~\ref{fig:cross_stitch}(c), the increase and the decrease need not happen in adjacent rows or in the same column of slip loops (where the column of loops is determined by the wale axis of a two-periodic weft-knitted textile), and the increase and the decrease in the number of slip loops within a row can be more than one as long as they match when counted over the whole swatch.
\item \emph{Stitch decreases and yarn-overs:} In Figure~\ref{fig:yarnover_decrease}, the $2\times1$ swatch within the larger $4\times2$ swatch illustrates another way of increasing the number of slip loops called the \emph{yarn-over}. In this case, the extra slip loop has the form of  a right-leaning twisted knit motif. The decrease in the number of slip loops is achieved by knitting more than two loops together before taking it off the left needle. Similar to the irreducible swatches in the second item above, the increase and the decrease in the number of slip loops within a row can be more than one as long as they match when counted over the whole swatch.
\item \emph{Slipping stitches:} A slip loop can be transferred from the left needle to the right needle  without pulling another slip loop through it, and this can be repeated over multiple rows before taking it off the needle by pulling through another slip loop. The swatch shown in both %Figure~\ref{fig:slip_stitch}(b) 
and Figure~\ref{fig:bru}(b) is made by slipping bights alternatively exactly once. This operation is essential for constructing non-trivial swatches that are not Brunnian.
\end{enumerate}
\noindent We allow for dual roles for some irreducible swatches as they appear in different forms in decomposition of compound swatches. This we state as a remark below.
\begin{remark}\label{remark:increase_decrease}
In the decomposition of a compound swatch, the irreducible swatches discussed in item number two above can arise in two forms. For example, the swatch shown in Figure~\ref{fig:cross_stitch}(b) can be constructed as a $1\times2$ swatch and a $2\times2$ swatch, and the one shown in Figure~\ref{fig:cross_stitch}(c) can be generated as a $1\times3$ swatch and a $3\times3$ swatch.
\end{remark}

\subsubsection{The number of $m\times n$ swatches}

The concept of annulus sums and irreducible swatches can be used to get a lower bound on the total number of $m\times n$ swatches for some fixed $m,n\in\mathbb{N}$. An $m\times n$ swatch is either an irreducible swatch or a compound swatch. If it is the latter, then there exists a unique way to divide the rectangle of width $m$ and height $n$ into a finite number of subrectangles of width $a_i$ and height $b_i$ such that the $(a_i,b_i)$ subrectangle embeds an $a_i\times b_i$ irreducible swatch for all $a_i\leq m$ and $b_i\leq n$. This decomposition of a bigger swatch into smaller swatches is due to the meridional and the longitudinal annulus sums. For example, the decomposition in Figure~\ref{fig:grid}(a)-(b) show a division into subrectangles of unit height and unit width. Furthermore, the construction of the $2\times2$ swatch corresponding to the seed fabric, starting with the knit swatch and the purl swatch is illustrated in Figure~\ref{fig:grid}(b). In Figure~\ref{fig:yarnover_decrease}, a $3\times2$ swatch is decomposed into a set four $(1,1)$ swatches and a $(2,1)$ irreducible swatch.

\begin{figure}
\centering
\includegraphics[width = 84 mm]{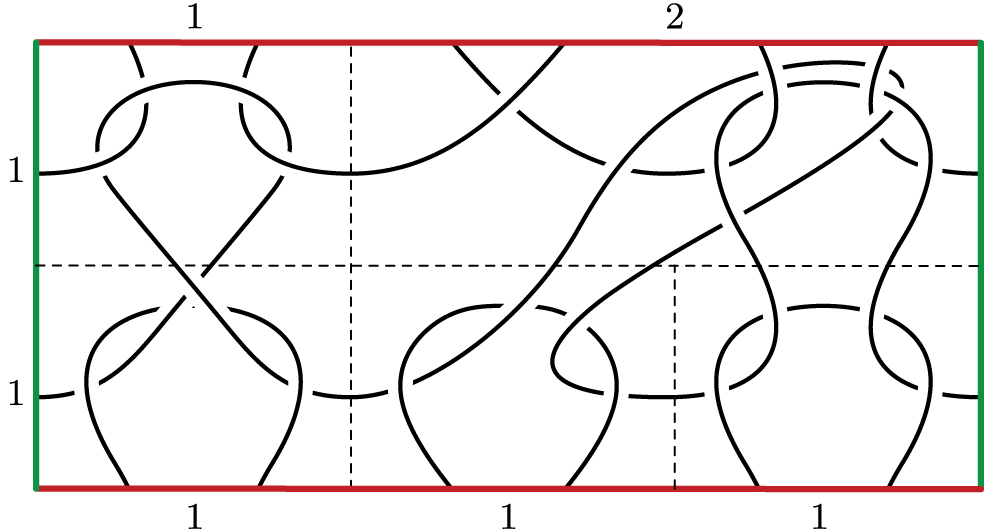}
\caption{\label{fig:yarnover_decrease}A $4\times2$ swatch consisting of a $2\times1$ irreducible swatch resulting from a \emph{stitch decrease by knitting two together} followed by a \emph{right-leaning twisted yarn over}.} 
 \end{figure}

\subsubsection{Complexity of a swatch}
     
For an $m$ by $n$ irreducible swatch, in part, the complexity can be attributed to the values of $m,n\in\mathbb{N}$ in the sense that, higher the values of $m,n$ higher is its complexity. The other part of the complexity is associated with the nature of knot tying mechanism used in making the bights. For instance, in case of a typical weft-knitting machine, the only ways of making bights corresponds to the mechanisms used in constructing the knit swatch and the purl swatch as opposed to the cow hitch swatch which can only be realized through hand knitting techniques. Both of these notions of complexity can be quantified using link invariants such as the hyperbolic volume. In TABLE~\ref{tab:invariants0}, if we compare the hyperbolic volumes of the the knit swatch and the knit-like cow hitch swatch, the latter is higher implying that the latter is more complex. And as expected, the hyperbolic volume of the $1\times1$ cable swatch -- a $2\times1$ irreducible swatch shown in Figure~\ref{fig:cross_stitch}(a) -- is greater than the hyperbolic volume of each and every $1\times1$ swatch listed in TABLE~\ref{tab:invariants0}. 

Our brief analysis of the origin of complexity in irreducible swatches naturally carries over to compound swatches in the following sense: the complexity of a compound swatch is roughly the `sum' of complexities of all the irreducible swatches that it is composed of. For example, in TABLE~\ref{tab:invariants0}, the hyperbolic volume of the link corresponding to $1\times1$ rib swatch Figure~\ref{fig:gluing}(a) is equal to the sum of the hyperbolic volumes of links corresponding to the knit swatch and the purl swatch. Also, we observe that the multivariable Alexander polynomial of the link associated with $1\times1$ rib swatch is equal to the product of the multivariable Alexander polynomials of links associated with the knit swatch and the purl swatch divided by a factor of $(t_1-1)$ (see TABLE~\ref{tab:invariants0}), where $t_1$ is the variable corresponding to the generator of the translation along the meridian in the Abelian cover of the link complement. In our experimental calculations using \texttt{SnapPy}, the relations stated above hold for all the $m\times1$ compound swatches. We propose conjecture~\ref{conjecture:mva} on the multivariable Alexander polynomial, and conjecture~\ref{conjecture:hyp} on the hyperbolic volume of links in $S^3$ that are associated with swatches. 
 
\section{Topological invariants of swatches}
\label{sec:invariants}
We compute few link invariants, and apply the correspondence between swatches and links in $S^3$ to classify and characterize two-periodic weft-knitted textiles. A \emph{link invariant} is the image of a map defined on the set of all links, as ordered sets of knots in $S^3$, which is preserved under ambient isotopies. We seek link invariants that conform with the properties of stitch patterns of two-periodic weft-knitted textiles. Even though for every two-periodic weft-knitted textile there exists a unique swatch embedded inside the fundamental translational unit, the quotient map can be chosen to yield a tiling unit consisting of more than one copy of the fundamental translational unit leading to a swatch which has more than one copy of the minimal swatch (the swatch inside the fundamental translational unit) combined through annulus sums. Therefore, one of the properties of an ideal link invariant for classifying stitch patterns of two-periodic weft-knitted textiles is that it is independent of the number of copies of the minimal swatch. Since such link invariants are rare in the existing literature, we instead study the algebra induced by the annulus sums on a link invariant to gauge and account for its dependency on the number of copies of minimal swatches.

Recall from \textbf{Remark}~\ref{remark:cyclic_permutation} that the decomposition of a compound swatch into a given set of smaller swatches is determined only up to cyclic permutations in the gluing order of their cut link complements. However, the annulus sums are \emph{non-commutative}. For instance, using the notation from section~\ref{subsec:partition}, the $4\times1$ swatch corresponding to the word $kkpp$ is equivalent to swatches denoted by the words $pkkp$, $ppkk$, $kppk$, but not equivalent to the one denoted by $kpkp$. This is because the word $kpkp$ is obtained by swapping the letters $k$ and $p$ in the second and third positions of the original word $kkpp$, which is not a cyclic transposition. Therefore, the gluing order under the annulus sums are cyclically permutable but non-commutative. Thus, in conclusion, we seek link invariants which are independent of the number of copies of the minimal swatches, non-commutative with respect to the gluing order, and unique only up to the cyclic permutations of the gluing order.

Based on our observation and numerical experiments, most of the link invariants do not show any direct correlation with the emergent elastic properties of the textiles. However, we can hypothesize as to which factors play a key role in determining the stiffness of two-periodic textile samples. A common property among textile samples that were experimented on numerically and in the lab by the authors of \cite{singal:2024p2622} is that, having swatches arising due to annulus sums that lead to alternating crossings (an under crossing followed by an over crossing or vice versa), while going from one irreducible swatch to the other tends to be less stiff. For example, this is observed in a comparative analysis done on stockinette, one-by-one rib, garter and seed stitches. The authors attribute this feature to the soft modes of deformation -- modes of deformation of yarn that cost less elastic energy. Thus, the order in which the motifs are put together in a stitch pattern is crucial in understanding the mechanics of textiles. 

In addition, the number of pairs of alternating crossings (that look like clasps) within an irreducible swatch is vital. This is because of the fact that, when stretched these parts of yarn come in contact and push on each other adding to the energy cost of deformation of yarn through bending and compression. These parts also cause energy dissipation due to friction between strands that are in contact. Therefore, a topological invariant that quantifies the number of pairs of alternating crossings while being able to distinguish between the clasp-type crossing motif from the generic alternate crossing motif in a reduced planar diagram, can prove useful in establishing a connection between a numerical invariant and the mechanics of textiles. We will be exploring this aspect of the problem in future.

The number of components of a swatch is a simple proxy for the number of irreducible swatches that generate it under annulus sums. For example, the minimal swatch corresponding to seed stitch pattern has two components and four irreducible swatches. In fact, the canonical decomposition of a swatch that has $n$ components consists of at most $2n$ irreducible swatches. For a swatch, the number of components is equal to the coefficient of the longitude in the homology type of the curves representing the swatch. This holds because, every non-contractible component of an unknit is homologous to the longitude of the base torus and a swatch with $n$ components is constructed from an unknit with $n$ non-contractible loops. Below, we briefly discuss link invariants that were computed for some swatches along with a summary of our observations and results.   
 
\begin{figure*}
\centering
\includegraphics[width = 140 mm]{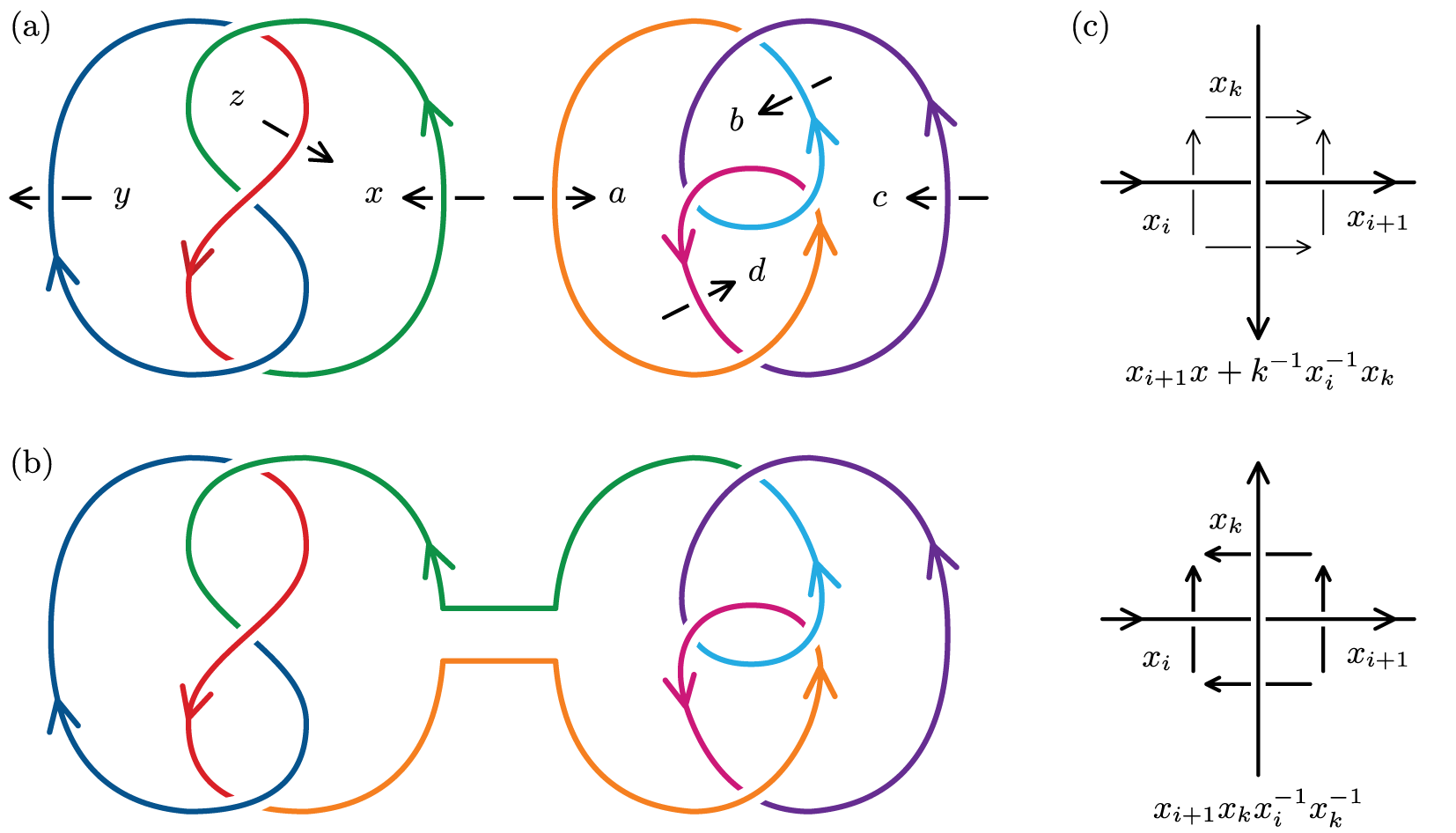}
\caption{\label{fig:Wirtinger}(a) A planar diagram of the right handed trefoil and the figure eight knot. (b) A planar diagram of the knot obtained by band surgery, where the band intersects the separating sphere in a simple arc. In both (a)  $\&$ (b), the colored arcs denote over-strands, and the labels are the letters constituting the alphabet of corresponding Wirtinger presentation. (c) A scheme for computing relations in the Wirtinger presentation.}
\end{figure*}

\subsection{Fundamental group of the link complement}
The link complement, which is introduced in \textbf{Definition}~\ref{definition:lc} is a 3-manifold link invariant, and as a result, the \emph{fundamental group of the link complement} denoted by $\pi_1$ is a link invariant. The fundamental group of a manifold is the set of all loops with a common base point in the manifold partitioned by the homotopy equivalence. The underlying group operation is the concatenation of based loops. The definition and a few examples are provided in section~\ref{appendix:link_group}.

Since the link complement of the Hopf link in $S^3$ is homeomorphic to $T^2\times I$, the link complement of an $(n+2)$-component link $H\cup f(L)\subset S^3$ is homeomorphic to the link complement of the $n$-component link $L\subset T^2\times I$. Recall that the link $f(L)\subset S^3$ is obtained by Dehn filling $L\subset T^2\times I$.
Therefore, any topological property of the link complement $S^3\setminus (H\cup f(L))$ is a topological invariant of $L\subset T^2\times I$ including its fundamental group. The fundamental group of a link complement admits a group presentation which is computed using the following lemma.

\begin{lemma}[Wirtinger presentation \cite{Rolfsen_1976}]\label{lemma:Wirtinger}
If $D(L)$ is an oriented link diagram of a link $L$ with arcs $x_1,x_2...,x_n$ and crossings $c_1,c_2...,c_n$, then
\begin{align}
& \pi_1(X_L)  \cong F[\{x_1,x_2,...,x_n\}]/N[\{r_1,r_2,...,r_{n-1}\}] = \nonumber\\
& \hspace{50pt}\langle x_1,x_2...,x_n|r_1,r_2...,r_{n-1}\rangle,
\end{align}
where each crossing $c_i$ gives a relation among the generators given by either $x_{i+1}x_k^{-1}x_i^{-1}x_k$ or $x_{i+1}x_kx_i^{-1}x_k^{-1}$ according to the relative configuration of the under strand and the over strand.
\end{lemma}  
\noindent Here, $F[X]$ denotes the free group generated by the set $X=\{x_1,\cdots,x_n\}$ with $X\cup X^{-1} = \{x_1,\cdots,x_n,x_1^{-1},\cdots,x_n^{-1}\}$ as the alphabet. The generating set given by $X$ consists of letters denoting the over-strands. As shown in Figure~\ref{fig:Wirtinger}(c), each element in the set of relations denoted as $\{r_1,\cdots,r_{n-1}\}$ is obtained by following the arcs at a crossing in the counter clockwise direction, and $N[\{r_1,r_2,\cdots,r_{n-1}\}]$ denotes the smallest normal subgroup of the free group $F[X]$, containing the set of relations $\{r_1,r_2,\cdots,r_{n-1}\}$. These concepts are discussed in a bit more detail in section~\ref{appendix:group_presentations}. 

The Wirtinger presentations based on the planar diagrams of the right handed trefoil knot, the figure eight knot and their connect sum, which are shown in Figure~\ref{fig:Wirtinger}(a)-(b) are given as follows:

\begin{align}
 \pi_1(S^3\setminus T) & \cong F[\{x,y,z\}]/N[\{yx^{-1}zx, xz^{-1}y^{-1}z\}] = \langle x,y,z | yx^{-1}zx, xz^{-1}y^{-1}z\rangle\nonumber\\
 \pi_1(S^3\setminus E) &\cong F[\{a,b,c\}]/N[\{ bc^{-1}ac, da^{-1}c^{-1}a,  cd^{-1}b^{-1}d\}] \nonumber\\
& = \langle a,b,c,d | bc^{-1}ac, da^{-1}c^{-1}a,  cd^{-1}b^{-1}d \rangle\nonumber\\
 \pi_1(S^3\setminus (T\#E)) & \cong \langle x,y,z,a,b,c,d | yx^{-1}zx, xz^{-1}y^{-1}z, bc^{-1}ac, da^{-1}c^{-1}a,  cd^{-1}b^{-1}d, xa^{-1}\rangle
\end{align}

The generators of the presentation of $\pi_1(S^3\setminus(T\#E))$ is the union of the generators of the Wirtinger presentations of $\pi_1(S^3\setminus T)$ and $\pi_1(S^3\setminus E)$. Further, the set of relations is also given by the union, but with one extra relation resulting from the band surgery or the connect-sum. Similarly, we can compute the Wirtinger presentations for the fundamental groups of the link complements of links resulting from adding the link complements of a pair of swatches under the annulus sums. Although computing Wirtinger presentations of complicated non-alternating links with higher minimal crossing number can be quite tedious by hand, it is readily done in \texttt{SnapPy} \cite{SnapPy} in the \texttt{sagemath} \cite{sagemath} environment by command $.knot\_group()$. Below, we demonstrate how to compute presentations of $\pi_1(S^3\setminus H\cup f(k*_mp))$ and $\pi_1(S^3\setminus H\cup f(k*_lp))$ using Wirtinger presentations of $\pi_1(S^3\setminus H\cup f(k))$ and $\pi_1(S^3\setminus H\cup f(p))$.

\begin{figure*}
\centering
\includegraphics[width = 160mm]{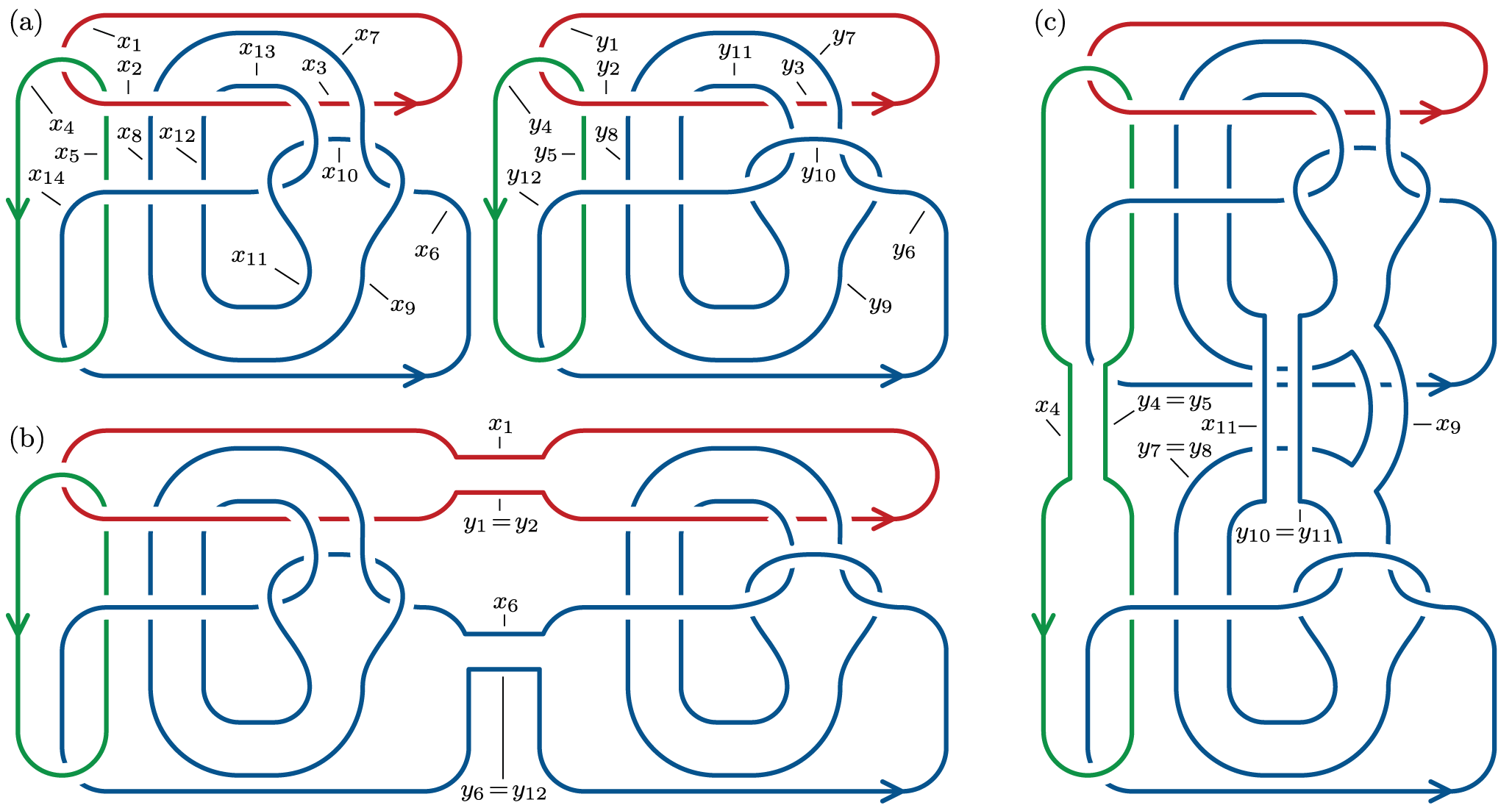}
\caption{\label{fig:bandsum_Wirtinger}The Wirtinger presentations of the fundamental groups of the links corresponding to the $1\times1$ rib and the garter swatches. (a) The link in a 3-sphere corresponding to the knit swatch and the purl swatch. (b) The link corresponding to the $1\times1$ rib swatch. (c) The link corresponding to the garter swatch. The over-strands are labelled and the relations resulting from the band surgery are provided adjacent to the bands.}
\end{figure*}
 
By labelling arcs in the link diagram using the scheme of assigning distinct labels to different over-strands, we get the following Wirtinger presentations of the links corresponding to the knit and the purl swatch:
\begin{align*}
&\pi_1(S^3\setminus H\cup f(k))\cong F[X]/N[R_k]\\
&\pi_1(S^3\setminus H\cup f(p))\cong F[Y]/N[R_p]\\
&\pi_1(S^3\setminus H\cup f(k*_mp))\cong F[X\cup Y]/N[R_r],
\end{align*}
where $X  = \bigcup_{i=1}^{n}\{x_i\}$, $Y = \bigcup_{j=1}^{m}\{y_j\}$ and $R_r = R_k\cup R_p\cup\{y_4,y_5,x_1y_1^{-1}x_6y_6^{-1}\}$. The relations $y_4=1$, $y_5=1$ combined with the relations associated with the crossings involving the green component of the link $H\cup f(p)\subset S^3$ imply that, $y_1=y_2$ and $y_6 = y_{12}$. 

For the presentation of the fundamental group of the link complement $S^3\setminus (H\cup f(k*_lp))$ given as $F[X\cup Y]/N[R_g]$, the set of relations is $R_g = R_k\cup R_p\cup\{y_1,y_2,y_3,x_4y_4^{-1}x_{11}y_{10}^{-1}x_9y_7^{-1}\}$, where $y_1=y_2=y_3=1$ implies that $y_4=y_5$, $y_7=y_8$ and $y_{10}=y_{11}$ due to deletion of the red component of the link $H\cup f(p)$.

The abelianzation of the fundamental group of a manifold yields the first homology group of the manifold. Thus, the first homology group denoted by $H_1(S^3\setminus H\cup f(k*_mp))$ is isomorphic to the abelianization of the presentation, $F[X\cup Y]/N[R_r]$. The first homology group is isomorphic to the Abelian group $\mathbb{Z}\oplus\mathbb{Z}\oplus\mathbb{Z}$, and the generators correspond to the meridians of the boundary tori of the link complement.

Given a swatch $L\subset T^2\times I$, the linking number lk$(f(L),f(l)))$ vanishes, but the sublink $f(L)\cup f(l)$ is not split. This non-split property of the sublink $f(L)\cup f(l)\subset S^3$ can be detected by looking at the presentations of the knots $f(L)$ and $f(l)$ in the Wirtinger presentations of the groups $\pi_1(S^3\setminus f(l)))$ and $\pi_1(S^3\setminus f(L))$ respectively. If $f(L)\cup f(l)\subset S^3$ is split, then $f(L)\subset\pi_1(S^3\setminus f(l))$ and $f(l)\subset\pi_1(S^3\setminus f(L))$ are both contractible or homotopic to a point (null-homotopic). If we analyze the link diagrams corresponding to the knit and the purl swatches shown in \ref{fig:bandsum_Wirtinger}(a), then we observe that in both of these cases, while the knot $f(L)\subset\pi_1(S^3\setminus f(l))$ is null-homotopic, the loop $f(l)\subset\pi_1(S^3\setminus f(L))$ is not contractible. This implies that the sublink $f(L)\cup f(l)\subset S^3$ is not split. 

As with any link invariant, showing fundamental groups of two links are not isomorphic implies the links are not ambient isotopic to each other, but in general proving whether or not two groups are isomorphic is rather difficult. However, when the swatches give rise to hyperbolic links, the geometry of the interior of the link complement is completely determined by its fundamental group \cite{Mostow_73, Prasad_73}. This implies that the geometric invariants are proxies for the knot groups of the link complements of the hyperbolic links, meaning if the link complements are isometric then the knot groups are isomorphic. And if any of the geometric invariants do not agree then the links are not ambient isotopic in $S^3$. Therefore, next we discuss three geometric invariants derived from the hyperbolic structure of the link complements of hyperbolic links in $S^3$.  
 
\subsection{Invariants derived from the hyperbolicity of swatches} 
We say a swatch $L\subset T^2\times I$ is \emph{hyperbolic} if the link $H\cup f(L)\subset S^3$ is hyperbolic. For a 3-manifold, admitting a hyperbolic structure implies existence of a triangulation by ideal tetrahedra glued together \emph{consistently}; a tetrahedron is ideal if its vertices are on the boundary sphere at infinity in the Poincar\'{e} disk model of $\mathbb{H}^3$. Here, the conditions of consistency in gluing yield a finite number of nonlinear equations in an equal number of unknowns, which characterize the ideal tetrahedra. These equations are called the \emph{gluing equations}, and the solutions to the gluing equations describe a hyperbolic 3-manifold. Using the \emph{upper half-space model} of hyperbolic 3-space $\mathbb{H}^3$, where it is parametrized as the space $\mathbb{C}^2\times(0,\infty)$, the solution to the gluing equations can be expressed as a finite set of complex numbers describing the shapes of ideal tetrahedra in the triangulation of the underlying hyperbolic manifold \cite{Purcell_2020}.

\begin{conjecture}[Hyperbolic swatches]
\label{conjecture:hyp}
An (n+2)-component link $H\cup f(L)\subset S^3$ corresponding to an $m\times n$ swatch $L\subset T^2\times I$ that is reducible in terms of knit swatches and purl swatches is hyperbolic.
\end{conjecture}
 
\subsubsection{Hyperbolic volume (vol)}
The hyperbolic volume of a complete finite-volume hyperbolic 3-manifold is simply the sum of the volumes of the ideal tetrahedra, which are calculated using the hyperbolic metric defined by the solution to the underlying set of gluing equations. Hyperbolic volume is a link invariant due to the \emph{Mostow Rigidity theorem}, according to which an isomorphism between the fundamental groups of a pair of complete finite-volume hyperbolic 3-manifolds induces a unique isometry \cite{Mostow_73, Prasad_73}. On the basis of our computation of hyperbolic volumes in \texttt{SnapPy}, we observe that the induced algebra due to the meridional annulus sum is additive for the three component links, or equivalently for $m\times1$ swatches. As a result, we propose the following conjecture:

\begin{conjecture}[Hyperbolic volume of one-component hyperbolic swatches]\label{conjecture:vol_sum}
Let $H\cup f(L_1)\subset S^3$ and $H\cup f(L_2)\subset S^3$ be the three-component hyperbolic links corresponding to swatches $L_1\subset T^2\times I$ and $L_2\subset T^2\times I$. If the link $H\cup f(L_1*_mL_2)\subset S^3$ corresponding to the swatch $L_1*_mL_2\subset T^2\times I$ is hyperbolic, then its hyperbolic volume is equal to the sum of the hyperbolic volumes of links $H\cup f(L_1)\subset S^3$ and $H\cup f(L_2)\subset S^3$.
\end{conjecture}

The hyperbolic volumes for several swatches computed in \texttt{SnapPy} are listed in tables \ref{tab:invariants0}, \ref{tab:invariants1} and \ref{tab:invariants2}. From this data, we conclude that the volume is not additive with respect to the meridional annulus sums for $m\times n$ swatches for $n\geq2$. For instance, the hyperbolic volume of the four component link corresponding to the seed swatch shown in Figure~\ref{fig:grid}(b), is not equal to the sum of hyperbolic volumes of the four component links corresponding to the swatches denoted by $k*_lp$ and $p*_lk$.

The hyperbolic volume is independent of the order in which the meridional annulus sum is performed. For example, \texttt{SnapPy} yields same values for the hyperbolic volume of $4\times1$ swatches denoted by $kkpp$ and $kpkp$. These swatches are not related by cyclic permutation of the order in which the cut link complements of pairs of knit and purl swatches are glued. Therefore, hyperbolic volume is equal for links corresponding to compound swatches related by cyclic permutations in the gluing order and, it is commutative with respect to meridional annulus sum. 

Hyperbolic volume depends on the number of copies of the minimal swatch in a tiling unit of the stitch pattern of a two-periodic weft-knitted textile. It is not yet understood how the hyperbolic volume changes with respect to the number of copies of the minimal swatch, except that the sum of hyperbolic volumes of the swatches that make up a bigger swatch under annulus sums sets a lower bound.

\subsubsection{The fundamental translational units of boundary horospheres: the cusp shapes}
The universal cover of the interior of the link complement of a hyperbolic link is $\mathbb{H}^3$, and therefore, the boundary tori in the link complement tessellate the horospheres \cite{Purcell_2020}. Horospheres are the surfaces (in 3D) and the curves (in 2D) that intersect the geodesics emanating from a point at infinity orthogonally. Thus, the point at infinity through which the geodesics emerge from is the only common point for all of the horospheres that are perpendicular to the geodesics. 
%The horospheres are locus of points that intersect geodesics passing through a point orthogonally. 
They are isometric to euclidean planes $\mathbb{R}^2$, which are tiled by parallelogram shaped fundamental domains. Since the horospheres corresponding to each boundary torus form a continuous family, the corresponding tiles can be assigned a \emph{shape parameter}, which is equal to the ratio of the longitude and the meridian of the torus corresponding to the edges of the tile. The shape parameter has the same value for any horosphere in the continuous family because the tiles are \emph{similar}, and thus, an ordered collection of shape parameters of the boundary tori with the same ordering as the ordered link is a link invariant. This link invariant is called the \emph{cusp shapes}. We compute cusp shapes of hyperbolic links using \texttt{SnapPy} \cite{SnapPy}. Qualitatively, with regard to our criteria for an ideal invariant for describing stitch patterns of two-periodic weft-knitted textile, cusp shapes behaves similar to the hyperbolic volume except we do not have a conjecture about the induced algebraic structure on a closed subset of $m\times1$ swatches, as in \textbf{Conjecture}~\ref{conjecture:vol_sum} for the hyperbolic volume.

\subsubsection{Trace fields $\&$ Invariant trace fields}
An element in the fundamental group of the link complement of a hyperbolic link corresponds to an isometry of its universal cover $\mathbb{H}^3$. 
The isometries of $\mathbb{H}^3$ can be parametrized by elements of $PSL(2,\mathbb{C})$. 
Consequently, the fundamental group of the link complement of a hyperbolic link has a representation in terms of elements in $PSL(2,\mathbb{C})$ \cite{Purcell_2020}.
This representation can be lifted to a representation into $SL(2,\mathbb{C})$.
The number field obtained by extending the field of rational numbers by the traces of the matrices in the image of this representation is a topological invariant \cite{Maclachlan_2003, Coulson_2000}. This field is known as the \emph{trace field}.

We are interested in a link invariant which is independent of the quotient map so that the swatches denoted as $k$, $kk$ or an $m\times1$ swatch obtained by the meridional annulus sum of $m$ knit swatches admit the same value. Such an invariant is called a \emph{commensurability} invariant. It turns out that the number field resulting by extension of the field of rational numbers by the traces of the squares of the matrices from $\tilde{\Gamma}$ is a commensurability invariant \cite{Maclachlan_2003, Coulson_2000}, which is called the \emph{invariant trace field}. Invariant trace field is of great interest to us since it is the only invariant, known to us so far, which satisfies all the properties required to be a good invariant of two-periodic knits. However, \texttt{SnapPy}'s computation of invariant trace field is only an approximation, and we have not been able to verify in our analysis that the swatches $k$ and $kk$ evaluate to the same value of invariant trace field. Working towards being able to accurately compute invariant trace field to desired degree of precision is a crucial step towards finding an ideal invariant of stitch patterns of two-periodic weft-knitted textiles.  
  
\subsection{Multivariable Alexander polynomial ($\textrm{MVA}_L$)}\label{sec:mva}
Multivariable Alexander polynomial is a multivariable Laurent polynomial invariant of links in $S^3$ as ordered sets. Each component of the link is associated with a variable, and this correspondence need not be one to one. R. H. Fox devised an algorithm for calculating \emph{Alexander matrices} \cite{Alexander_1928} using the Wirtinger presentations of the fundamental groups of the link complements. Based on this algorithm and Fox calculus \cite{Fox_62}, he then generalized the Alexander polynomial (a Laurent polynomial in single variable) of links in $S^3$ to the multivariable Alexander polynomial. We use \texttt{SnapPy} \cite{SnapPy} and \texttt{KnotTheory} \cite{KTP} -- a \texttt{Mathematica} package -- to compute these polynomials. Based on our data we propose the following conjecture: 
\begin{conjecture}[MVA of links corresponding to swatches]
Let $L_1\subset T^2\times I$ and $L_2\subset T^2\times I$ be an $m_1\times n$ swatch and an $m_2\times n$ swatch respectively. Let $H\cup f(L_1)\subset S^3$ and $H\cup f(L_2)\subset S^3$ be the corresponding (n+2)-component links. The multivariable Alexander polynomial of the (n+2)-component link $H\cup f(L_1*_mL_2)\subset S^3$ is given by the product of multivariable Alexander polynomials of $H\cup f(L_1)\subset S^3$ and $H\cup f(L_2)\subset S^3$ divided by the factor $(t_1-1)^n$. The variable $t_1$ is associated with the component $m\subset H$ of the Hopf link and $m_1,m_2,n\in\{1,2,\cdots\}$.
\label{conjecture:mva}
\end{conjecture}
\noindent Here is an example of such trio of swatches demonstrating the content of conjecture \ref{conjecture:mva}:

\begin{align}
\textrm{MVA}_{H\cup f(p*_lp)}(t_1,t_2,t_3,t_4)  & = (t_1 - 1)^2(t_2t_3t_4 - t_3t_4 + t_3 + t_4 - 1)\nonumber\\
&(t_2t_3t_4 - t_2t_3 - t_2t_4 + t_2 - 1)\nonumber\\ 
\textrm{MVA}_{H\cup f(k*_lp)}(t_1,t_2,t_3,t_4)  & = (t_1 - 1)^2(t_3t_4 + t_2 - t_3 - t_4 + 1)\nonumber\\
&(t_2t_3t_4 - t_2t_3 - t_2t_4 + t_3t_4 + t_2)\nonumber\\
\hspace{50pt}\textrm{MVA}_{H\cup f((pp)*_l(kp))}(t_1,t_2,t_3,t_4)  & = \textrm{MVA}_{H\cup f((p*_lp)(k*_lp))}(t_1,t_2,t_3,t_4)\nonumber\\
&  = \textrm{MVA}_{H\cup f(p*_lp)}(t_1,t_2,t_3,t_4) \textrm{MVA}_{H\cup f(k*_lp)}(t_1,t_2,t_3,t_4)/(t_1-1)^2, 
\end{align}

\noindent where $t_i$ is associated with the $i^{\textrm{th}}$ component of an ordered $n$-component link. More examples are listed in tables \ref{tab:invariants0}, \ref{tab:invariants1} and \ref{tab:invariants2}. A disadvantage of the multiplicative structure of the multivariable Alexander polynomials of $m\times n$ swatches for fixed $n$ is that, all the links corresponding to swatches obtained by meridional annulus sum of an identical set of swatches evaluate to the same Laurent polynomial. In other words, this invariant is independent of the order in which the annulus sum is performed. Given below is the multivariable Alexander polynomial of 4$\times$1 swatches $kkpp$ and $kpkp$.
\begin{align}
& \textrm{MVA}_{H\cup f(kkpp)}(t_1,t_2,t_3) = \textrm{MVA}_{H\cup f(kpkp)} (t_1,t_2,t_3)\nonumber\\ 
&= (t_3 - 1)(t_2^2 + t_1 - t_2)^2 (-t_1t_2 + t_2^2 + t_1)^2\nonumber\\
&\hspace{60pt}(t_1t_2 - t_2 + 1)^2(t_1t_2 - t_1 + 1)^2.
\end{align}
Thus, despite the fact that these swatches are distinct because of the non-commutative action of meridional annulus sum, their multivariable Alexander polynomials are equal.
\begin{remark}\label{remark:brun}
The multivariable Alexander polynomial of a $(k+2)$-component link $H\cup f(L)\subset S^3$ corresponding to Brunnian swatch $L\subset T^2\times I$ takes the following special form:
\begin{align}\label{eqn:brun}
& \textrm{MVA}_{H\cup f(L)}(t_1,...,t_{k+2}) = \nonumber\\
& (1-t_1)^k\left(\prod\limits_{i=1}^{k+2}t_i^{m_i}+p(t_2,...,t_{k+2})\prod\limits_{j=3}^{k+2}(1-t_j)^{n_j}\right),
\end{align}
where $m_i\in\mathbb{Z}$ $\&$ $n_j\in\mathbb{N}$ for all $1\leq i,j\leq(k+2)$, and the variables $t_1, t_2$ are associated with the components $f(m),f(l)\subset H$ of the Hopf link respectively. 
\end{remark}
\textbf{Remark}~\ref{remark:brun} follows from the Torres formula \cite{Torres_1953}, which for an $n$-component link $L' = \bigcup\limits_{i=1}^nK_i\subset S^3$ gives the following relation:
\begin{align}
\label{eqn:torres}
& \textrm{MVA}_{L'}(t_1,...,t_{i-1},1,t_{i+1},...,t_n) = \nonumber\\
& \left(1- \prod\limits_{j\neq i} t_j^{\textrm{lk}(K_i,K_j)}\right) \textrm{MVA}_{L'\setminus K_i}(t_1,...,t_{i-1},t_{i+1},...,t_n).
\end{align}
Note that a swatch obtained by acting the meridional annulus sum on a pair of Brunnian swatches is Brunnian, and the form of corresponding multivariable Alexander polynomials in equation \ref{eqn:brun} is consistent with this observation. Therefore, we conclude that the set of all $m\times n$ Brunnian swatches forms a closed subset of the set of all $m\times n$ swatches $\mathcal{P}_n$, where both the sets consist of only the swatches generated using the annulus sums. In addition, for a pair of sets that are as mentioned above, $m\in\{1,2,\cdots\}$ is arbitrary and $n\in\{1,2,\cdots\}$ is fixed.

\subsection{Jones polynomial ($\textrm{V}_L$)}
Jones polynomial is a Laurent polynomial in a single variable with integer coefficients. Jones polynomial can be defined in terms of the \emph{bracket polynomial}. Given a link $L\subset S^3$, its bracket polynomial is denoted by $\langle L\rangle$. The bracket polynomial of a link in $S^3$ is calculated using the following set of local rules applied to each crossing in a link diagram successively as illustrated in Figure~\ref{fig:skein_relation}(a): 

\begin{figure}[h!]
\centering

\subfloat[Skein relation for the bracket polynomial.]{\includegraphics[width = 41 mm]{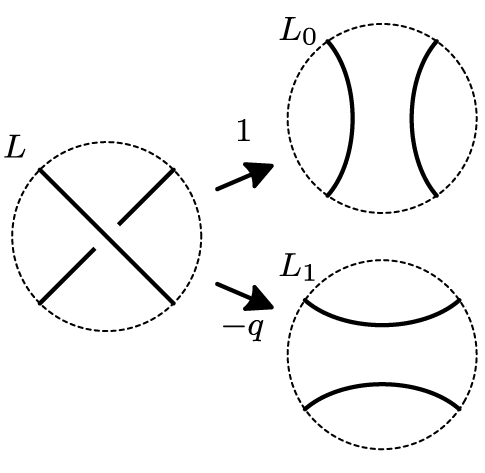}}
\hspace{1mm}
\subfloat[Crossing motifs with positive sign (left) and negative sign (right)]{\includegraphics[width = 41 mm]{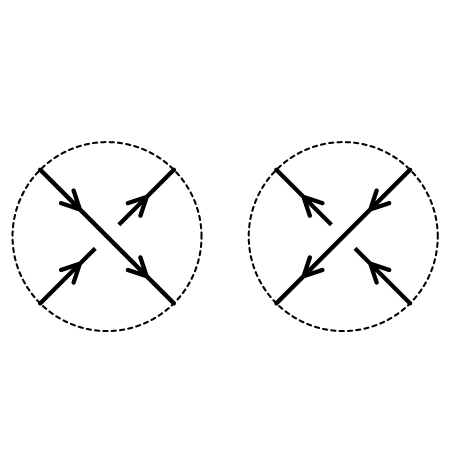}}

\caption{\label{fig:skein_relation} The elements of computation of the Jones polynomial of a link in a 3-sphere.}
\end{figure}

\begin{enumerate}
\item $\langle\bigcirc\rangle$ = 1.
\item If $L\subset S^3$ is a link in a 3-sphere, then $\langle L\rangle$, $\langle L_0\rangle$ and $\langle L_1\rangle$ which differ locally as shown in Figure~\ref{fig:skein_relation}(a) are related to each other by,
\begin{equation*}
\langle L\rangle  = \langle L_0\rangle - q \langle L_1\rangle. 
\end{equation*}
\item $\langle\bigcirc \cup L\rangle = (q+q^{-1})\langle L\rangle$.   
\end{enumerate}

\begin{definition}[Jones polynomial]
The \textit{Jones polynomial} of a link $L\subset S^3$ is given by
\begin{equation}
\emph{V}_{L}(q,q^{-1}) = (-1)^{n_{-}}q^{n_{+}-2n_{-}}\langle L\rangle,
\end{equation}
where the integers $n_{+}$, $n_{-}$ denote the total number of positive, negative crossings assigned as per the convention depicted in Figure~\ref{fig:skein_relation}(b).
\end{definition}

Jones polynomial is not multiplicative with respect to the meridional annulus sum. Nevertheless, it is independent of the order in which the annulus sums are performed. It depends on the quotient map, or on the number of copies of the minimal swatch constituting a textile link. An advantage of using Jones polynomial to describe swatches lies in the fact that it is able to distinguish between links which have equal values for every other link invariant in our study. For example, the value of invariants listed in TABLE~\ref{tab:invariants1} indicate that, except the cusp shapes and Jones polynomials of two versions of the cow hitch swatch, every other link invariant is identical. Since cusp shapes are non-zero for only hyperbolic links, having Jones polynomial to distinguish between seemingly identical (but distinct) swatches is useful.

\subsection{Link determinant (det)}
The link determinant is an integer valued link invariant that can be derived from the multivariable Alexander polynomial according to the following definition.
\begin{definition}[Link determinant, \cite{Rolfsen_1976}]
The positive integer given by $|\emph{MVA}_L(-1,-1,...,-1)|$ is called the link \emph{determinant} of a link $L\subset S^3$, which is denoted by \emph{det(L)}.
\end{definition}
\noindent Naturally, the link determinant is weaker than the multivariable Alexander polynomial because some information is lost when the latter is reduced by substitution. It is worth mentioning that the link determinant is also weaker than the Jones polynomial since it can be obtained by evaluating Jones polynomial at $q=-1$. Based on our \textbf{Conjecture}~\ref{conjecture:mva} about the multivariable Alexander polynomial, it follows that

\begin{equation}\label{eqn:link_det}
 2^N\textrm{det}(H\cup f(L_1*_mL_2*_m...*_mL_{N}) = \textrm{det}(H\cup f(L_1))\textrm{det}(H\cup f(L_2))...\textrm{det}(H\cup f(L_N)),
\end{equation}

where $L_j\subset T^2\times I$ is an $m_j\times n$ swatch, for all $j\in\{1,2,\cdots,N\}$. The relation in equation \ref{eqn:link_det} is consistent with the values of the link determinant in tables \ref{tab:invariants0}, \ref{tab:invariants1} and \ref{tab:invariants2}. We summarize our analysis of link invariants of two-periodic weft-knitted textiles in the following TABLE~\ref{tab:summary_invariants}:
         
\clearpage         
\begin{sideways}
\begin{threeparttable}
\caption{Summary of relevant properties of link invariants of swatches}
\begin{tabular}{|p{2cm}| p{2cm} |p{2cm}  |p{2cm}  |p{2cm} |p{10cm}|}
\hline
Link invariant & Non-commutative (meridional) & Non-commutative (longitudinal) & Indep. of cyclic permutation in the order (both) & Indep. of the quotient map & Classification vs Characterization and other important properties.\\
\hline
Linking number (lk)  & False & False & True & False & Given a link in $T^2\times I$, linking number can identify properties two and three in the list~\ref{list:properties}. It admits same value for all $m\times n$ swatches for a fixed $n\in\{1,2,\cdots\}$.\\
\hline
Fundamental group of the link complement & True & True & True & False & The fundamental group of the link complement is a strong link invariant relative to others used in our study. \texttt{SnapPy} is able to compute Wirtinger presentations in the {\tt sagemath} environment. {\tt SnapPy} can check if the link complements of a pair of hyperbolic links are isometric, which is equivalent to checking whether their fundamental groups are isomorphic.\\
\hline
Multivariable Alexander polynomial (MVA) & False & True & True & False & This invariant takes on a simple form given in equation \ref{eqn:brun} for Brunnian swatches. It is able to distinguish between swatches resulting from the longitudinal annulus sum of the same set of swatches.\\
\hline
The Jones polynomial (V) & False & False & True & False & This invariant is able to distinguish between two versions of the cow hitch swatch while almost every other invariant in the list fails. It is interesting to note that the Jones polynomial is identical for all the swatches obtained through annulus sum of a set of swatches.\\
\hline
Link determinant & False & False & True & False & It follows from the definition that the link determinant is weaker than multivariable Alexander polynomial and Jones polynomial. Nonetheless, equation \ref{eqn:link_det} that follows from \textbf{Conjecture}~\ref{conjecture:mva} might be noteworthy. \\
\hline
Hyperbolic volume (vol) & True & True & True & False & The hyperbolic volume is a good measure of complexity of hyperbolic swatches in agreement with our basic intuition about the complexity of swatches that give rise to hyperbolic links. The hyperbolic volume of a compound swatch is bounded below by the sum of the hyperbolic volumes of irreducible swatches constituting it.\\
\hline
Invariant trace field & True & True & True & True & Invariant trace field fits every criteria which we propose for a good invariant of stitch patterns of wo-periodic weft-knitted textiles, but we are not able to draw conclusions based on the approximate values computed in \texttt{SnapPy}.\\
\hline
The cusp shapes & True & True & True & False & Similar to the hyperbolic volume, cusp shapes are non-zero only if the underlying link is hyperbolic. Therefore, it can be used to identify hyperbolic swatches and their classification.\\
\hline
\end{tabular}
\end{threeparttable}
\label{tab:summary_invariants}
\end{sideways}
\clearpage

\section{Conclusion}
In this paper we state and prove the \textbf{Theorem}~\ref{theorem:ribbon} that encapsulates a characteristic property of stitch patterns of two-periodic weft-knitted textiles by virtue of which they lead to ribbon links. This is a crucial result for a couple of reasons: 1) the proof brings in a novel method of constructing a special class of links in $T^2\times I$ called \emph{swatches}, 2) the mechanics of weft-knitting entails different ways of tying knots into the bight or making different types of slip loops. This basic idea is at the core of constructing a swatch, and thus, we conclude that all the links derived from stitch patterns of two-periodic weft-knitted textiles must be swatches.

By analyzing different topological properties of swatches, we are able to extract novel aspects of stitch patterns of two-periodic weft-knitted textiles. For example, the links in $S^3$ arising from swatches are algebraically unlinked but not split, and a considerable number of swatches lead to \emph{Brunnian} links and \emph{hyperbolic} links. Learning about how swatches that lead to non-Brunnian links and non-hyperbolic links differ from their counterparts is insightful. 
Non-Brunnian swatches arise whenever one or more slip loops are transferred to the other needle without pulling a bight through them. Consequently, these slip loops are taken off of the needle at a later stage after making several rows of stitches. An example of a non-Brunnian swatch is shown in Figure~\ref{fig:bru}(b). With regards to the hyperbolic links and swatches, we state the \textbf{Conjecture}~\ref{conjecture:hyp} proposing that the links corresponding to swatches that are reducible to a finite collection of knit and purl swatches are hyperbolic.  

In an attempt at developing a universal formal language for describing charts that encode weft-knitting stitch patterns \cite{Durant_2015,Shida_2017, Pavilion_2015}, we recognize that the motifs that arise in weft-knitting charts and textiles form an infinite collection of \emph{monoids}, where each monoid is determined by the number of components in the swatches that it contains. 
Therefore, we partition the set of all swatches by assigning an index to each partition equal to the number of components in swatches contained in that partition (see equation \ref{eqn:partition}). These partitions form \emph{monoids}, and therefore, the letters in the weft-knitting alphabet are the set of generating swatches, which are acted upon by the associative binary operation given by the meridional annulus sum. The number of letters in the alphabet generating a given partition is countably infinite, unless the complexity of mechanical moves are limited by specifying a machine knitting protocol. Note that, for partitions of swatches with more than one component, the generating set of swatches consists of a subset of irreducible swatches and a subset of compound swatches. Here, the compound swatches that arise as generating swatches admit decomposition through the action of longitudinal annulus sum.

The simple syntax of stacking 2D motifs side by side and one below the other originate from the meridional and the longitudinal annulus sums of swatches. In this spirit, finite collections of symbols or letters that come up in knitting charts are simply concatenated to encode knitting instructions. These letters and symbols that encode knitting instructions are either short forms of the actual instructions spelled out, or they are caricatures of the visuals of the motifs. 

Apart from concatenation, we briefly describe few complex `grammatical' structures such as, cabling, yarn overs, stitch increases and decreases, slipping a stitch all of which give rise to higher order irreducible swatches -- $m\times n$ swatches with $m,n>1$. We delve into how the complexity of a swatch can be quantified by associating the notion of complexity with some observable features of swatches. This discussion is summarized below. 

Naturally, a compound swatch is more complex than each and every irreducible swatch constituting it, and the combined measure of complexity is roughly the sum of measures of complexity of the irreducible swatches. One can verify this by computing \emph{hyperbolic volume} and cusp shapes. The other notion of complexity is associated with the mechanics of knitting the motifs corresponding to the irreducible swatches. For instance, making a textile sample with the cow hitch stitch as the minimal swatch requires more complicated mechanical moves than making a textile sample with just knit and purl swatches. The former can be achieved only with hand knitting while the latter can be constructed using both the machine knitting and the hand knitting.   

In this work, we started out to find suitable link invariants for describing two-periodic weft-knitted textiles, where the main objective is two-fold: classification and characterization. The commensurability invariant given by the \emph{invariant trace field} is an ideal invariant of two-periodic textiles. A major bottleneck parameter in computing invariant trace field is the minimum number of crossings. \texttt{SnapPy}'s algorithm can only yield approximate values of the invariant trace fields using which we are not able to verify that swatches denoted by $k$ and $kk$ admit the same value of invariant trace field as expected. For almost all of the links in $S^3$ that correspond to the swatches, the minimal number of crossings is higher than that of the knit swatch and the purl swatch. The higher value of minimal number crossings along with the fact that these links are non-alternating implies that the approximate values of the invariant trace fields computed through \texttt{SnapPy} are non-reliable. Despite not being able to find an ideal, computable link invariant (through existing software tools), we propose three conjectures based on the computation and analysis of the link invariants that were examined in this study.

In \textbf{Conjecture}~\ref{conjecture:mva}, we propose that the set of multivariable Alexander polynomials of links associated with swatches having the same number of components form a closed set under multiplication up to a factor. From our analysis of hyperbolic swatches, we propose that the hyperbolic volume of the three component link corresponding to an $m\times1$ compound swatch is equal to the sum of the hyperbolic volumes of the links corresponding to its irreducible swatches \ref{conjecture:vol_sum}. These conjectures illustrate the fact that the complexities of compound swatches scale with the number of irreducible swatches and their individual complexities. In relation to swatches giving rise to hyperbolic links, we propose \textbf{Conjecture}~\ref{conjecture:hyp} according to which the links corresponding to swatches obtained by acting annulus sums on the link complements of knit and purl swatches are hyperbolic. However, it is crucial to note that these are not the only swatches that give rise to hyperbolic links.

\clearpage
\begin{sideways}
\begin{threeparttable}
\caption{A table of link invariants of swatches}
\begin{tabular}{|p{1.5cm} | p{1.85cm} | p{1.85cm} | p{1.85cm} | p{1.85cm} | p{1.85cm} | p{1.85cm} | p{2cm} | p{1.85cm} | p{2.5cm} |}
\hline
Link invariants & $k$ & $p$ & $kp$ & $kk$ &  $pp$ &  $k*_lk$ &  $p*_lp$ & $k*_lp$ & $(kp)*_l(pk)$ $=(k*_lp)( p*_lk)$  \\
\hline
Vol  
& $10.666979$ $1337962..$ 
& $10.666979$ $1337962..$ 
& $21.3339582$ $675925..$ 
& $21.3339582$ $675925..$ 
& $21.3339582$ $675925..$ 
& $21.3339582$ $675925..$ 
& $21.3339582$ $675925..$ 
& $21.3339582$ $675925..$ 
& $45.342172$ $4012022..$ 
\\
& & & & & & & & & \\
\hline
$\textrm{MVA}_L$ 
& $(1-t_1)(t_3^2 + t_2 - t_3)(t_3^2 - t_2t_3 + t_2)$ 
& $(t_1-1)(t_2t_3 - t_3 + 1)(t_2t_3 - t_2 + 1)$ 
& $(1-t_1)(t_2t_3 - t_3 + 1)(t_2t_3 - t_2 + 1)(t_3^2 + t_2 - t_3)(t_3^2 - t_2t_3 + t_2)$ 
& $(t_1-1)(t_3^2 + t_2 - t_3)^2(t_3^2 - t_2t_3 + t_2)^2$ &  $(t_1-1)(t_2t_3 - t_3 + 1)^2(t_2t_3 - t_2 + 1)^2$ 
& $(t_1 - 1)^2 (t_3^2t_4^2 - t_3^2t_4 - t_3t_4^2 + t_3t_4 - t_2)(t_3^2t_4^2 - t_2t_3t_4 + t_2t_3 + t_2t_4 - t_2)$ 
& $(t_1 - 1)^2(t_2t_3t_4 - t_3t_4 + t_3 + t_4 - 1)(t_2t_3t_4 - t_2t_3 - t_2t_4 + t_2 - 1)$ 
& $(t_1 - 1)^2(t_3t_4 + t_2 - t_3 - t_4 + 1)(t_2t_3t_4 - t_2t_3 - t_2t_4 + t_3t_4 + t_2)$ 
& $(t_1 - 1)^2(t_3t_4 + t_2 - t_3 - t_4 + 1)^2(t_2t_3t_4 - t_2t_3 - t_2t_4 + t_3t_4 + t_2)^2$ 
\\
& & & & & & & & & \\
\hline
$\textrm{V}_L$ 
& $-q^{-6}(1+q^2)(1-q+q^3-2q^4-2q^6+q^7-q^9+q^{10})$ 
& $q^{-4}(1+q^2)(1-3q+4q^2-2q^3+4q^4-3q^5+q^6)$ 
& $-q^{-8}(1 + q^2)(1 - 3q + 5q^2 - 4q^3 - q^4 + 7q^5 - 10q^6 + 8q^7 - 10q^8 + 7q^9 - q^{10} - 4q^{11} + 5q^{12} - 3q^{13} + q^{14})$ 
& $q^{-10}(1 + q^2)(1 - 2q + q^2 + 2q^3 - 7q^4 + 6q^5 - q^6 - 4q^7 + 7q^8 - 4q^9 + 7q^{10} - 4q^{11} - q^{12} + 6q^{13} - 7q^{14} + 2q^{15} + q^{16} - 2q^{17} + q^{18})$ 
& $2q^{-6}(1 + q^2)(1 - 3q + 6q^2 - 11q^3 + \frac{27}{2}q^4 - 12q^5 + \frac{27}{2}q^6 - 11q^7 + 6q^8 - 3q^9 + q^{10})$ & $-q^{-11}(1 + q^2)(1 - 2q + q^2 + 2q^4 - 8q^5 + 14q^6 - 14q^7 + 16q^8 - 14q^9 + 14q^{10} - 8q^{11} + 2q^{12} + q^{14} - 2q^{15} + q^{16})$ 
& $-2q^{-7}(1 + q^2)(1 - 4q + \frac{15}{2}q^2 - 8q^3 + 9q^4 - 8q^5 + \frac{15}{2}q^6 - 4q^7 + q^8)$
& $q^{-9}(1 + q^2)(1 - 3q + 2q^2 + 3q^3 - 8q^4 + 8q^5 - 10q^6 + 8q^7 - 8q^8 + 3q^9 + 2q^{10} - 3q^{11} + q^{12})$ 
& $-q^{-13}(1 + q^2)(1 - 3q - 2q^2 + 29q^3 - 77q^4 + 108q^5 - 74q^6 - 24q^7 + 129q^8 - 190q^9 + 210q^{10} - 190q^{11} + 129q^{12} - 24q^{13} - 74q^{14} + 108q^{15} - 77q^{16} + 29q^{17} - 2q^{18} - 3q^{19} + q^{20})$ 
\\
& & & & & & & & & \\
\hline
det$(L)$ 
& 4 
& 36 
& 36 
& 4 
& 324 
& 200 
& 200 
& 72 
& 648 
\\
& & & & & & & & & \\
\hline
cusp $\#0$ & $i0.377964$ $473009227$ 
& $-8.7e^{-17} +$ $i1.51185789..$ 
& $2.63e^{-16} +$ $0.30237157..$ 
& $-2.3e^{-17} +$ $i0.18898223..$ 
& $-3.05e^{-16} +$ $i0.75592894..$ 
& $4.76e^{-16} +$ $i0.75592894..$ 
& $-1.824e^{-15} +$ $i3.02371578..$ 
& $1.27e^{-16} +$ $i1.88982236..$ 
& $2.92e^{-16} +$ $i0.73743829..$ 
\\
\hline
cusp $\#1$ 
& $5.0e^{-16}+$ $i0.66143782..$ 
& $1.554e^{-15}+$ $i2.64575131..$ 
& $-2.665e^{-15} +$ $i3.30718913..$ 
& $-8.88e^{-16} +$ $i1.32287565..$ 
& $-3.664e^{-15} +$ $i5.29150262..$ 
& $5.6e^{-17} +$ $i0.33071891..$ 
& $-1.332e^{-15} +$ $i1.32287565..$ 
& $3.6e^{-17} +$ $i0.52915026..$ 
& $1.42e^{-16} +$ $i1.35604565..$ 
\\
\hline
cusp $\#2$ 
& $1.11e^{-15}+$ $i2.64575131..$ 
& $-1.998e^{-15}+$ $i2.64575131..$ 
& $-1.643e^{-14}+$ $i5.29150262..$ 
& $-1.110e^{-14} +$ $i5.29150262..$ 
& $-1.494e^{-14} +$ $i5.29150262..$ 
& $4.441e^{-15} +$ $i2.64575131$ 
& $-3.331e^{-15} +$ $i2.64575131..$ 
& $4.885e^{-15} +$ $i2.64575131..$ 
& $-4.578e^{-15}+$ $i5.67593337..$ 
\\
\hline
cusp $\#3$ 
& 
& 
& 
& 
& 
& $-2.614e^{-15} +$ $i2.64575131..$ 
& $2.22e^{-15} +$ $i2.64575131..$ 
& $1.554e^{-15} +$ $i2.64575131..$ 
& $1.3318e^{-14} +$ $i5.67593337..$ 
\\
\hline
\end{tabular}
\label{tab:invariants0}
\end{threeparttable}
\end{sideways}
\clearpage

\clearpage
\begin{sideways}
\begin{threeparttable}
\caption{A table of link invariants of swatches (continued)} 
\begin{tabular}{|p{1.5cm} | p{2.75cm} | p{2.75cm} | p{2.75cm} | p{2.75cm} | p{2.75cm} | p{3.75cm} |}
\hline
Link invariants 
& Swatch $p_{2t}$ (shown in Figure~\ref{fig:alphabets_p1}(b)) 
& Swatch $k_{2w}$ (shown in Figure~\ref{fig:alphabets_p1}(c)) 
&  $p_{2t}k_{2w}$ 
&  Cow hitch swatch (purl-like) 
& Cow hitch swatch (knit-like; shown in Figure~\ref{fig:cow_hitch}) 
& $k*_lp*_lp_t$\\ 
\hline
Vol  
& $11.6638549$ $138147..$ 
& $14.8269324$ $660918..$ 
& $26.4907873$ $799065..$ 
& $16.411246$ $0614635..$ 
& $16.411246$ $0614635..$ 
& $32.3533147$ $753432..$\\ 
& & & & & &\\ 
\hline
$\textrm{MVA}_L$ 
& $(t_1-1)(t_2t_3 - t_3 + 1)(t_2t_3 - t_2 + 1)$ 
& $(1-t_1)(t_3^3 + t_2 - t_3)(-t_2t_3^2 + t_3^3 + t_2)$ 
&  $(1- t_1)(t_2t_3 - t_3 + 1)(t_2t_3 - t_2 + 1)(t_3^3 + t_2 - t_3)(-t_2t_3^2 + t_3^3 + t_2)$ 
& $(t_1 - 1)$ 
& $(t_1 - 1)$ 
& $(t_1 - 1)^3(t_3^2t_4t_5 - t_3^2t_4 - t_3^2t_5 - t_3t_4t_5 + t_3^2 + t_3t_4 - t_2t_5 + t_3t_5 - t_3) (t_2t_3t_4t_5 - t_2t_3t_4 + t_3^2t_4 - t_2t_3t_5 - t_2t_4t_5 + t_2t_3 + t_2t_4 + t_2t_5 - t_2)$\\
& & & & & &\\ 
\hline
$\textrm{V}_L$ 
& $q^{-2}(1 + q^2)(2 - 3q + 4q^2 - 2q^3 + 3q^4 - 3q^5 + q^6$ 
& $-q^{-8}(1 + q^2)(1 - q + 2 q^3 - 2 q^4 - q^6 - q^8 - 2 q^9 + 2 q^{10} - q^{12} + q^{13}$ 
& $-q^{-12}(1 + q^2)(1 - 3 q + 4 q^2 - 2 q^3 - 2 q^4 + 6 q^5 - 9 q^6 + 11 q^7 - 7 q^8 + 2 q^{10} - 5 q^{11} + 5 q^{12} - 5 q^{13} + 4 q^{15} - 4 q^{16} + 2 q^{17})$ 
& $(-q^{-7})(1 + q^2)(-1 + 3q - 2q^2 - 2q^3 + 3q^4 - 2q^5 + q^6 - 4q^7 + 2q^8 + 2q^9 - 3q^{10} + q^{11})$ 
& $(q^{-9})(1 + q^2)(-1 + q + q^2 - 3q^3 + 2q^4 + 2q^5 - 2q^6 + q^7 + 3q^9 - 2q^{10} - 2q^{11} + 3q^{12} - q^{13} - q^{14} + q^{15})$ 
& $q^{-10}(1 + q^2) (1 - 4 q + 4 q^2 + 9 q^3 - 32 q^4 + 59 q^5 - 69 q^6 + 75 q^7 - 69 q^8 + 57 q^9 - 32 q^{10} + 8 q^{11} + 4 q^{12} - 4 q^{13} + q^{14})$ \\
& & & & & &\\
\hline
det$(L)$ 
& 36 
& 4 
& 36 
& 4 
& 4 
& 784\\
& & & & & &\\
\hline
cusp $\#0$ 
& $-0.2377114..+$ $i1.8443448..$ 
& $0.0382437..+$ $i0.4799797..$ 
& $0.03419881..+$ $i0.38101155..$ 
& $0.0085110..+$ $i1.9097742..$ 
& $0.0021277..+$ $i0.47744356..$ 
& $-0.10867972709..+$ $i3.4521316931894..$\\
\hline
cusp $\#1$ 
& $0.27496080..+$ $i2.1333530..$ 
& $-0.0412387..+$ $i0.5175695..$ 
& $-0.3161995..+$ $i2.65092256..$ 
& $-0.009334..+$ $i2.0944468..$ 
& $-0.002333..+$ $i0.52361170..$ 
& $0.047787517734..+$ $i0.4106072115318..$ \\
\hline
cusp $\#2$ 
& $0.27496080..+$ $i2.1333530..$ 
& $2.0724434..+$ $i2.2724864..$ 
& $1.79748262..+$ $i4.40583950..$ 
& $-0.009334..+$ $i2.0944468..$ 
& $-0.009334..+$ $i2.0944468..$ 
& $-0.20183912206..+$ $i2.750825167148..$\\
\hline
cusp $\#3$ 
& 
& 
& 
& 
& 
& $0.522445767750..+$ $i2.247859040741..$ \\
\hline
cusp $\#4$ 
& 
& 
& 
& 
& 
& $0.0908135925076..+$ $i2.609945699389..$\\
\hline
\end{tabular}
\begin{tablenotes}
\item The invariants listed after link determinant (det) are the cusp shapes. The numbering scheme is fixed throughout. The boundary torus of a neighborhood of the meridian ($f(m)\subset H$) and the longitude ($f(l)\subset H$) are denoted by cusp$\#0$ and cusp$\#1$ respectively. The rest are numbered starting from the top (along the wale direction aligned along $y$-axis) and they are associated with the components of sublink $f(L)\subset H\cup f(L)\subset S^3$.
\end{tablenotes}
\label{tab:invariants1}
\end{threeparttable}
\end{sideways}
\clearpage

\clearpage
\begin{sideways}
\begin{threeparttable}
\caption{A table of link invariants of swatches (continued)} 
\begin{tabular}{|p{1.5cm} | p{2.cm} | p{2.cm} | p{2.75cm} | p{3cm} | p{2.5cm} | p{3cm} | p{3.cm} |}
\hline
Link invariants 
& $k_t$ 
& $p_t$ 
& $kpp_t$ 
&  $pkp_t$ 
& 1$\times$1 cable swatch shown in Figure~\ref{fig:cross_stitch}(a)
& Stitch decrease and increase swatch (shown in Figure~\ref{fig:cross_stitch}(b)) 
& $p*_lk*_lp_t$ \\
\hline
Vol  
& $11.0821666$ $243737..$ 
& $11.0821666$ $243737..$ 
& $32.4161248$ $919662..$ 
& $32.4161248$ $919662..$ 
& $24.3051826$ $391952..$
& $34.7090324905077..$ 
& $32.3533147753432..$ \\
& & & & & & &\\
\hline
$\textrm{MVA}_L$ 
& $(1-t_1)(-t_2 + t_3 - 1)(t_2t_3 - t_2 + t_3)$ 
& $(t_1 - 1)(t_2 + t_3 - 1)(t_2t_3 - t_2 - t_3)$ 
& $(1- t_3)(t_1 + t_2 - 1)(t_2^2 + t_1 - t_2)(-t_1t_2 + t_2^2 + t_1)(t_1t_2 - t_2 + 1)(t_1t_2 - t_1 + 1)(t_1t_2 - t_1 - t_2)$
&  $(1- t_3)(t_1 + t_2 - 1)(t_2^2 + t_1 - t_2)(-t_1t_2 + t_2^2 + t_1)(t_1t_2 - t_2 + 1)(t_1t_2 - t_1 + 1)(t_1t_2 - t_1 - t_2)$ 
& $(t_1 - 1)(t_2^2 + t_3^2 - 2t_3 + 1)(t_2^2t_3^2 - 2t_2^2t_3 + t_2^2 + t_3^2)$
& $(t_1 - 1)^2 (t_3^3t_4^2 - t_3^3t_4 - t_3^2t_4^2 + t_3^2t_4 - t_3t_4 - t_2 + t_3 + t_4 - 1)(t_2t_3^3t_4^2 - t_2t_3^3t_4 - t_2t_3^2t_4^2 + t_3^3t_4^2 + t_2t_3^2t_4 - t_2t_3t_4 + t_2t_3 + t_2t_4 - t_2)$ 
& $(t_1 - 1)^3(t_3t_4t_5 - t_3t_4 - t_3t_5 - t_4t_5 - t_2 + t_3 + t_4 + t_5 - 1)(t_2t_3t_4t_5 - t_2t_3t_4 - t_2t_3t_5 - t_2t_4t_5 + t_3t_4t_5 + t_2t_3 + t_2t_4 + t_2t_5 - t_2)$ \\
& & & & & & &\\
\hline
$\textrm{V}_L$ 
& $q^-{5}(1 + q^2)(1 - q + 2q^3 - q^4 + q^5 - q^6 + q^7 - q^9 + q^{10})$ 
& $-q^{-3}(1 + q^2)(1 - 4q + 3q^2 - 3q^3 + 3q^4 - 3q^5 + q^6)$ 
& $q^{-9}(1 + q^2) (1 - 6 q + 16 q^2 - 27 q^3 + 27 q^4 - 10 q^5 - 19 q^6 + 51 q^7 - 69 q^8 + 76 q^9 - 68 q^{10} + 42 q^{11} - 8 q^{12} - 16 q^{13} + 26 q^{14} - 23 q^{15} + 13 q^{16} - 5 q^{17} + q^{18})$ 
& $q^{-9}(1 + q^2) (1 - 6 q + 16 q^2 - 27 q^3 + 27 q^4 - 10 q^5 - 19 q^6 + 51 q^7 - 69 q^8 + 76 q^9 - 68 q^{10} + 42 q^{11} - 8 q^{12} - 16 q^{13} + 26 q^{14} - 23 q^{15} + 13 q^{16} - 5 q^{17} + q^{18})$ 
& $-q^{-9}(1 + q^2)(1 - 2q + q^2 + 2q^3 - 2q^4 - 3q^5 + 8q^6 - 11q^7 + 10q^8 - 12q^9 + 10q^{10} - 5q^{11} - q^{12} + 4q^{13} - 3q^{14} + q^{15})$
& $-q^{-13}(1 + q^2) (1 - 3 q + 3 q^2 + 3 q^3 - 13 q^4 + 19 q^5 - 17 q^6 + 6 q^7 + 12 q^8 - 31 q^9 + 51 q^{10} - 55 q^{11} + 49 q^{12} - 33 q^{13} + 17 q^{14} - 4 q^{15} - 4 q^{16} + 4 q^{17} - 2 q^{19} + q^{20})$ 
& $q^{-10}(1 + q^2) (1 - 4 q + 4 q^2 + 9 q^3 - 32 q^4 + 59 q^5 - 69 q^6 + 75 q^7 - 69 q^8 + 57 q^9 - 32 q^{10} + 8 q^{11} + 4 q^{12} - 4 q^{13} + q^{14})$ \\
& & & & & & &\\
\hline
det$(L)$ 
& 4 
& 36 
& 324 
& 324 
& 100
& 392 
& 784 \\
& & & & & & &\\
\hline
cusp $\#0$ 
& $-0.057783..+$ $i0.42256082..$ 
& $-0.23113..+$ $i1.6902433..$ 
& $-0.0052516..+$ $i0.25709702..$ 
& $-0.0052516..+$ $i0.25709702..$ 
& $0.324347..+$ $i0.51116178..$
& $-0.3326528..+$ $i0.98590206..$ 
& $-0.2749292..+$ $i3.5331243..$ \\
\hline
cusp $\#1$ 
& $0.0794180..+$ $i0.58077069..$ 
& $0.317672..+$ $i2.3230827..$ 
& $0.31767219..+$ $i5.6302719..$ 
& $0.31767219..+$ $i5.6302719..$ 
& $-0.885015.. +$ $i1.39475926..$
& $0.00125017..+$ $i0.41202862..$ 
& $0.02098504..+$ $i0.4296608..$ \\
\hline
cusp $\#2$ 
& $0.3176721..+$ $i2.32308279..$ 
& $0.317672..+$ $i2.3230827..$ 
& $0.31767219..+$ $i7.6145854..$ 
& $0.31767219..+$ $i7.6145854..$ 
& $0.703435.. +$ $i4.97731847..$
& $0.11903618..+$ $i3.33245802..$ 
& $-0.2018391..+$ $i2.7508251..$\\ 
\hline
cusp $\#3$ 
& 
& 
& 
& 
&
& $-1.8545762..+$ $i3.84088297..$ 
& $0.52244576..+$ $i2.2478590..$ \\
\hline
cusp $\#4$ 
& 
& 
& 
& 
&
& 
& $0.09081359..+$ $i2.6099456..$ \\
\hline
\end{tabular}
\begin{tablenotes}
\item Similar to swatch $p_{2t}$ shown in Figure~\ref{fig:alphabets_p1}(b), swatches $p_t$ and $k_t$ denote purl-like and knit-like swatches with a single twist at the base of the slip loop (or the bight).
\end{tablenotes}
\label{tab:invariants2}
\end{threeparttable}
\end{sideways}
\clearpage

\section{Appendix}

\subsection{The fundamental group}
\label{appendix:link_group}

\begin{definition}[Based homotopy of loops]
Given a topological space $X$ and
\begin{align*}
\gamma_0 : [0,1] & \rightarrow X\\
\gamma_1 : [0,1] & \rightarrow X
\end{align*}
such that $\gamma_0(0) = \gamma_0(1) = x_0$ and $\gamma_1(0) = \gamma_1(1) = x_0$ for some $x_0\in X$. We say $\gamma_0$ is homotopic to $\gamma_1$ if there exists a continuous map
\begin{align*}
H:[0,1]\times[0,1]\rightarrow X
\end{align*}
where $H(t,s) = H_t(s) = \gamma_t(s)$ defines a set of loops based at $x_0$ i.e $H_t(0) = H_t(1) = x_0$ for all $t\in[0,1]$, $x_0\in X$ is called the base point.
\end{definition}
The property of based homotopy between two closed curves is an equivalence relation. We denote an equivalence class formed under homotopy by $\left[ \gamma \right]$ where $\gamma$ is a representative closed curve based at $x_0\in X$.
Given two loops $\gamma_1$ and $\gamma_2$ based at some point $x_0\in X$ we define \textit{concatenation} of loops as
\begin{align*}
\gamma(s) = \gamma_1* \gamma_2 = 
\begin{cases}
\gamma_1(2s)\hspace{40pt} s\in[0,1/2]\\
\gamma_2(2s-1)\hspace{22pt} s\in[1/2,1]
\end{cases}
\end{align*}
The operation of concatenation as defined above is well defined. To see this let's consider two representative loops from two classes $\alpha_1,\alpha_2\in [\alpha]$ and $\beta_1,\beta_2\in [\beta]$. Suppose concatenation of $\alpha_1$ and $\beta_1$ belongs $[\gamma]$, now we claim that concatenation of $\alpha_2$ and $\beta_2$ belongs to $[\gamma]$ as well, then 
\begin{align*}
[\alpha_2*\beta_2] & = [\alpha_2]*[\beta_2]\\
 & = [\alpha_1]*[\beta_1]\\
 & = [\alpha_1*\beta_1] \\
 & = [\gamma]
\end{align*}
\begin{definition}[Fundamental group]
The set of based homotopy equivalence classes of loops in a topological space forms a group with concatenation as the group operation.
\end{definition}
The fundamental group of $X$ denoted by $\pi_1(X, x_0)$, where $x_0$ is the base point. The identity element of the group is the constant function, $\gamma(s) = x_0$ for all $s\in[0,1]$. The inverse of a loop $\gamma$ is given by $\overline{\gamma}$ where
\begin{align*}
\gamma*\overline{\gamma} = 
\begin{cases}
\gamma(2s) \hspace{40pt} s\in[0,1/2] \\
\gamma(2 - 2s)\hspace{22pt} s\in[1/2,1]
\end{cases}
\end{align*}
It can be shown that $\gamma*\overline{\gamma}$ is homotopic to a constant loop, and thus, belongs to the equivalence class of identity element. If the underlying topological space is path connected $\implies$ fundamental group is independent of the choice of base point. Here are some examples of fundamental groups for some path connected topological spaces.
\begin{enumerate}
\item For $n > 1$
\begin{equation*}
\pi_1(\mathbb{R}^n) = \{\mathbf{0}\} 
\end{equation*}
\item 
\begin{equation*}
\pi_1(S^n) = 
\begin{cases}
(\mathbb{Z},+) & n = 1\\
\{\bf{0}\} & n > 1
\end{cases}
\end{equation*}
\item For genus $g$ handle-bodies
\begin{equation*} 
\pi_1(T^2_g) = (\mathbb{Z}_{2g},*)
\end{equation*}
where $*$ denotes the addition operation modulo $2g$ for $g\in\mathbb{N}$.
\item  $\pi_1(T^n) \cong (\mathbb{Z}_{n},*)$
where $*$ denotes the addition operation modulo $n$ for $n\in\mathbb{N}$.
\item $\pi_1(A)  \cong (\mathbb{Z},+)$
\end{enumerate} 
$A$ stands for \textit{Annulus}, $T^n\cong S^1\times S^1...\times S^1$ is $n$-torus, $T^2_{2g}$ is the boundary of the handlebody with $g$ handles.

\subsection{Group presentations}
\label{appendix:group_presentations}
\begin{definition}[Alphabet]
An alphabet $X$ consists a finite set of symbols or letters. 
\end{definition}
\begin{definition}[Syllable]
A syllable is a symbol of the form $a^n$ for $a\in X$ where $X$ is an alphabet and $n\in\mathbb{Z}$. 
\end{definition}
\begin{definition}[Word]
A word is a finite ordered sequence of syllables i.e. a word consists of a finite number of syllables placed next to each other. 
\end{definition}
\textbf{Example:} If $X = \{a,b,c,t\}$, then $a^3a^{-1}b^0b^{2}c^{-1}t$ is a word. Note that syllables are words of length one and a word with no syllable is denoted by $1$ and called the \textit{empty word}. We use concatenation to combine two words. Given the words $w_1$, $w_2$ from the alphabet $X$, $w_1w_2$, $w_2w_1$ are also words from $X$; $w_1w_2$ is not necessarily equal to $w_2w_1$.
 \begin{definition}[Reduced word]
We define two operations of contracting words: (i) $w_1a^pa^qw_2\sim w_1a^{p+q}w_2$, (ii) $w_1a^0w_2\sim w_1w_2$. The words $w_1a^{p+q}w_2$, $w_1w_2$ are called reduced words.
\end{definition}
Here $u\sim v$ denotes that $u$ and $v$ are equivalent. Given an alphabet $X$. The set of all reduced words $F[X]$ forms a group with concatenation of words as the group multiplication, $F[X]$ is called the free group generated by $X$ \cite{Nielsen_1917}. For example, $F[{x}]$ is the \textit{free group} of single generator.

\begin{definition}[Group homomorphism]
Given two groups $G_1$, $G_2$, a mapping $h$ between $G_1$ and $G_2$ is called a homomorphism if it preserves group multiplication i.e., $h(g_1\cdot g_2) = h(g_1)*h(g_2)$. A bijective homomorphism is called an \textit{isomorphism}. 
\end{definition}

Let $\{g_1,g_2,...,g_n\}$ be a generating set for a group $G$. Let $X$ be an alphabet and $f$ be an onto map from $X$ to $\{g_1,g_2,...,g_n\}$. Let $h$ be the natural extension of $f$ from the free group on X, $F[X]$ to $G$. Then it can be shown that the \emph{kernel of h} given by
\begin{equation}
\textrm{ker}(h)=\{g\in F(X):h(g) = e\}
\end{equation}
is a normal subgroup of $F[X]$, where $e\in G$ is the identity element. Moreover, using the group isomorphism property, it follows that the quotient group
$F[X]/ker(h)$ is isomorphic to the image of $h$, $Im(h) = G$. In this setting, a group presentation of $G$ is given by, $\langle x_1,x_2,...,x_k\rangle/\langle r_1,r_2,...,r_m\rangle$ where $\langle r_1,r_2,...,r_m\rangle = ker(h)$ denotes the smallest normal subgroup of $F[X]$ consisting of $\{r_i\}_{i=1}^m$, called the set of relations and $\langle x_1,x_2,...,x_k\rangle$ denotes the free group $F[X]$ on $X= \{x_1,x_2,...,x_k\}$.

\textbf{Examples:}
\begin{itemize}
\item Infinite cyclic group, $\mathbb{Z} \cong F[\{x\}]:= \langle x\rangle$
\item Finite cyclic group of order $n$, $\mathbb{Z}_n \cong F[\{x\}]/\langle x^n\rangle:= \langle x|x^n\rangle$
\item Dihedral group of order $n$, $D_n\cong\langle x,y|x^n, y^2, xyxy\rangle$
\end{itemize}

%\bibliographystyle{alpha}
%\bibliography{knits}
\newcommand{\etalchar}[1]{$^{#1}$}

\end{document}